\documentclass{amsart}

\usepackage[T1]{fontenc}
\usepackage{amsmath, amsfonts, amssymb, amsthm, mathrsfs, wasysym, 
			a4wide, 
			graphics, graphicx, 
			url, hyperref, hypcap, %
			xargs, xifthen, %
			enumerate,
			multicol, multirow, array, blkarray,
			xfrac, pifont,
			}

\usepackage[noabbrev,capitalise]{cleveref}
\hypersetup{bookmarks=true, bookmarksopen=false,%
    colorlinks=true,%
    linkcolor=darkblue,%
    citecolor=darkblue,%
    linktoc=page,%
}
\usepackage[all]{xy}
\usepackage[bottom]{footmisc}
\usepackage{tikz}
\usepackage{tkz-graph}
\usetikzlibrary{trees, decorations, decorations.markings, shapes, arrows, matrix, calc, fit, intersections, patterns, angles}
\usepackage{caption}
\newtheorem{theorem}{Theorem}[section]
\newtheorem{corollary}[theorem]{Corollary}
\newtheorem{proposition}[theorem]{Proposition}
\newtheorem{lemma}[theorem]{Lemma}

\newtheorem*{theorem*}{Theorem}

\theoremstyle{definition}

\newtheorem{example}{Example}
\newtheorem*{examplebis}{\cref{main ex} revisited}
\newtheorem{remark}[theorem]{Remark}

\newtheorem{case}{Case}
\newtheorem{subcase}{Subcase}[case]

\newcommand{\set}[2]{\left\{ #1 \;\middle|\; #2 \right\}} 

\newcommand{\ie}{\textit{i.e.}~} 
\newcommand{\eg}{\textit{e.g.}~} 
\definecolor{darkblue}{rgb}{0,0,0.7} 
\definecolor{green}{RGB}{57,181,74} 
\definecolor{violet}{RGB}{147,39,143} 

\renewcommand{\S}{\mathbf{S}}
\newcommand{\M}{\mathbf{M}}
\renewcommand{\L}{\mathbf{L}}
\renewcommand{\P}{\mathbf{P}}
\newcommand{\newword}[1]{\textbf{\textit{#1}}}
\newcommand{\hy}{\hat{y}}
\newcommand{\A}{{\mathcal A}}
\renewcommand{\AA}{{\mathbf A}}

\newcommand{\g}{\mathbf{g}}

\newcommand{\arc}{\mathord{\operatorname{arc}}}

\newcommand{\Pp}{P_\purple}
\renewcommand{\Pr}{P_\red}
\newcommand{\Pb}{P_\blue}
\newcommand{\yo}{\hy_\orange}
\newcommand{\cross}{\text{cr}}

\newcommand{\covered}{\lessdot}
\newcommand{\covers}{\gtrdot}
\newcommand{\purple}{{\textup{purple}}}
\newcommand{\red}{{\textup{red}}}
\newcommand{\blue}{{\textup{blue}}}
\newcommand{\orange}{{\textup{orange}}}

\usepackage{todonotes}
\makeatletter
\def\part{\@startsection{part}{1}%
\z@{.7\linespacing\@plus\linespacing}{.8\linespacing}%
{\LARGE\sffamily\centering}}
\@addtoreset{section}{part}
\makeatother

\makeatletter
\def\l@section{\@tocline{1}{2pt}{0pc}{}{}}
\makeatother
\let\oldtocpart=\tocpart
\renewcommand{\tocpart}[2]{\bf\large\oldtocpart{#1}{#2}}
\let\oldtocsection=\tocsection
\renewcommand{\tocsection}[2]{\bf\oldtocsection{#1}{#2}}

\title[Posets for $F$-polynomials in cluster algebras from surfaces]{Posets for $F$-polynomials\\ in cluster algebras from surfaces}

\thanks{VP is partially supported by the Spanish grant PID2022-137283NB-C21 of MCIN/AEI/10.13039/501100011033 / FEDER, UE, by Departament de Recerca i Universitats de la Generalitat de Catalunya (2021 SGR 00697), by the French grant CHARMS (ANR-19-CE40-0017), and by the French--Austrian projects PAGCAP (ANR-21-CE48-0020 \& FWF I 5788).
NR is partially supported by the National Science Foundation under award numbers DMS-1500949 and DMS-2054489 and by the Simons Foundation under award number 581608. SS is partially supported by the DFG through
the project SFB/TRR 191 Symplectic Structures in Geometry, Algebra and Dynamics (Projektnummer 281071066-
TRR 191).
}

\author{Vincent Pilaud}
\address{Universitat de Barcelona}
\email{vincent.pilaud@ub.edu}
\urladdr{\url{https://www.ub.edu/comb/vincentpilaud/}}

\author{Nathan Reading}
\address{North Carolina State University}
\email{reading@math.ncsu.edu}
\urladdr{\url{https://nreadin.math.ncsu.edu}}

\author{Sibylle Schroll}
\address{Universit\"at zu K\"oln, Germany and NTNU, Norway}
\email{schroll@math.uni-koeln.de}
\urladdr{\url{https://sites.google.com/site/sibylleschroll/}}

\begin{document}

\begin{abstract}
We prove a simple formula for arbitrary cluster variables in the marked surfaces model.
As part of the formula, we associate a labeled poset to each tagged arc, such that the associated $F$-polynomial is a weighted sum of order ideals.
Each element of the poset has a weight, and the weight of an ideal is the product of the weights of the elements of the ideal.
In the unpunctured case, the weight on each element is a single $\hy_i$, in the usual sense of principal coefficients.
In the presence of punctures, some elements may have weights of the form $\hy_i/\hy_j$.
Our search for such a formula was inspired by the Fundamental Theorem of Finite Distributive Lattices combined with work of Gregg Musiker, Ralf Schiffler, and Lauren Williams that, in some cases, organized the terms of the $F$-polynomial into a distributive lattice.
The proof consists of a simple and poset-theoretically natural argument in a special case, followed by a hyperbolic geometry argument using a cover of the surface to prove the general case.
\end{abstract}

\vspace*{.2cm}
\maketitle


Cluster variables are multivariate Laurent polynomials given by a multidirectional recurrence whose initial data is a skew-symmetrizable matrix.
The cluster variables generate a ring called a cluster algebra.
Cluster algebras were defined by Sergey Fomin and Andrei Zelevinsky \cite{ca1}, motivated by total positivity and dual canonical bases in Lie theory, and have since found connections to many mathematical areas.
(See, for example, \cite{LW}.)
Given a skew-symmetrizable matrix, a basic problem is to solve the recurrence by giving explicit formulas for the cluster variables.

%

Cluster variable formulas were given by James Propp~\cite{ProppArxiv}, Gabriel Carroll and Gregory Price~\cite{CarrollPrice}, and Ralf Schiffler~\cite{Schiffler-clusterExpansion} for the cluster algebra associated to the triangulations of a polygon. 
The formulas were extended to a more general class of cluster algebras, associated to triangulated surfaces~\cite{FominShapiroThurston, FominThurston}, first by Gregg Musiker and Ralf Schiffler in terms of perfect matchings on tile graphs~\cite{MusikerSchiffler}, and more generally by Gregg Musiker, Ralf Schiffler and Lauren Williams in terms of perfect matchings on snake graphs~\cite{MusikerSchifflerWilliams,MSWBases}. 
Further reformulations appeared in~\cite{CeballosPilaud-pseudotriangulationsTypeD,Yurikusa-clusterExpansiontypeA,Yurikusa-clusterExpansion}. 
Cluster variable formulas are often called ``cluster expansion formulas'', but we prefer to avoid confusion with the notion of cluster expansion in~\cite{FominZelevinsky-Ysystems,affdenom}.

The objective of this paper is to prove a more condensed cluster variable formula for triangulated surfaces. 
We present the formula independently of snake graphs, but we were motivated~by the observation by James Propp \cite[Theorem~2]{ProppArxiv} that the set of perfect matchings on a graph forms a distributive lattice.
(This observation was already mentioned in the context of snake graphs by Gregg Musiker, Ralf Schiffler, and Lauren Williams \cite[Theorem~5.2]{MSWBases}.
It was also crucial in \cite{CanSchr}, where it was used to characterize the canonical submodule lattice of a string module.
The present paper got its start from conversations among the authors about the results of~\cite{CanSchr}.)

The fact that the Laurent monomials in the cluster variable form a distributive lattice calls to mind the Fundamental Theorem of Finite Distributive lattices \cite[Theorem~3.4.1]{EC1}, due to Birkhoff~\cite{Birkhoff}.
The relevant direction of the theorem is that, for any finite distributive lattice~$L$, there is a unique poset $P$ such that $L$ is isomorphic to the inclusion order on order ideals in~$P$.

The cluster variables in a triangulated surface are indexed by tagged arcs in the surface.
For a tagged arc $\alpha$, we want to understand the poset $P_\alpha$ that encodes the Laurent monomials in the cluster variable~$x_\alpha$.
Optimistically, one would hope to write down~$P_\alpha$ directly from $\alpha$, and $x_\alpha$ would be essentially a sum of monomials indexed by order ideals in~$P_\alpha$.
More optimistically, one would hope that the monomials could be recovered from order ideals in~$P_\alpha$ simply by giving a weight to each element of $P_\alpha$ such that the monomial corresponding to an order ideal of~$P_\alpha$ is the product of the weights of its elements.
That is precisely what happens.  
Most optimistically, each element~$e$ of~$P_\alpha$ would be associated to a tagged arc in the initial tagged triangulation, with the weight of $e$ being the monomial~$\hy_i$ associated to that arc (in the usual principal-coefficients notation of \cite{ca4}).
This most optimistic possibility occurs whenever the surface has no punctures and more generally when the triangulation has no self-folded triangles.
In the general case, the situation is nearly as good:  Each element $e$ of $P_\alpha$ is again associated to an arc, but in some cases when $i$ and $j$ index the edges of a self-folded triangle, the weight on an arc is of the form~$\hy_i/\hy_j$.

Details on the construction of the poset $P_\alpha$ appear in the next section.
For now, we state the general idea and give an example.
The poset $P_\alpha$ has a central piece that we will call $P_\alpha^\cross$.
The undirected graph underlying the Hasse diagram of $P_\alpha^\cross$ is a path, recording essentially the sequence of arcs of $T^\circ$ that $\alpha$ crosses.
In many cases when $\alpha$ is tagged notched at one or more endpoints, we adjoin to $P_\alpha^\cross$ one or more chains, with each adjoined chain recording the sequence of arcs that are incompatible with $\alpha$ because of its tagging at that endpoint. 
When $\alpha$ agrees, up to tagging, with an arc in $T$, then the construction ``degenerates'' in various ways that can be precisely described.

\begin{example}\label{main ex}
The left picture in \cref{main ex fig} shows an example of a marked surface (a 4-punctured disk) with a triangulation $T^\circ$ (with arcs numbered $1$ through $11$) and a tagged arc $\alpha$ (shown thicker, in purple).  
The cluster variable $x_\alpha$ is $\frac{x_5x_6x_8}{x_1x_4x_7x_9}$ times the weighted sum of all order ideals in the poset $P_\alpha$ shown in the right picture.
(This weighted sum is the $F$-polynomial of $x_\alpha$, in the sense of~\cite{ca4}.)
The weight of an element is $\hy_i$ if the element is labeled~$i$.
The weight of the element labeled $\frac23$ is $\frac{\hy_2}{\hy_3}$.
The poset $P_\alpha^\cross$ is like the poset shown, but deleting the chain labeled $4\covered1\covered8\covered11$ on the right.
\begin{figure}
\includegraphics{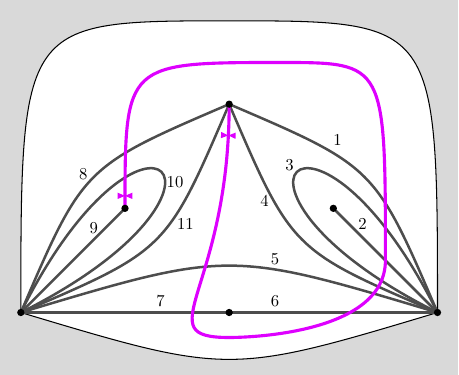}
\qquad
\scalebox{1.1}{\includegraphics{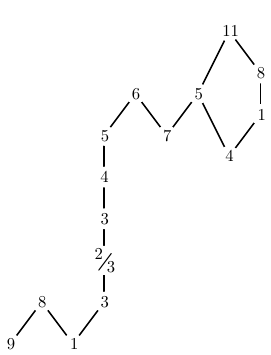}}
\caption{A tagged arc and the corresponding labeled poset.}
\label{main ex fig}
\end{figure}
We explain this example in full detail in \cref{sec:formula}.
\end{example}

The unpunctured case of the cluster variable formula can also be obtained from \cite[Theorem~5.4]{MSWBases}.
(See also \cite{Claussen}.)
Near the beginning of our work on this project, Jon Wilson \cite{Wilson} constructed similar weighted posets and proved similar cluster algebra variable formulas in the ``nondegenerate'' cases (with some restrictions on the surface), proving his results based on the cluster variable formulas in terms of the snake graphs of~\cite{MusikerSchifflerWilliams}.
After we had conjectured our main theorem among ourselves, but before we had proved it, Ezgi Kantarc{\i} O{\u g}uz and Emine Y{\i}ld{\i}r{\i}m \cite{OY} independently stated the nondegenerate case (with no restrictions on the surface) and proved it by expressing the cluster variables as products of matrices.
More recently, similar posets also appear in \cite{Weng}, where they are used to compute Donaldson-Thomas $F$-polynomials.

We conclude this introduction with an outline of the paper and a discussion of background references.

\medskip

\noindent
\textit{\cref{sec:formula}.}  
We state the formula (\cref{main}) for the cluster variable associated to a tagged arc and illustrate it with numerous examples.
We also make a standard, slight reduction of the theorem, showing that it needs only be proved for a subset of the tagged arcs.

%
%
%

\medskip

\noindent
\textit{\cref{sec:special}.}  
We prove in a special case that the cluster variable formula satisfies the correct exchange relations.
The point is that certain exchange relations correspond to simple, natural decompositions of the set of order ideals of the corresponding posets (\cref{subsec:F}), and to natural relations between the $\g$-vectors (\cref{subsec:shear}).
(Exchange relations, in the guise of skein relations, are also characterized in terms of snake graphs in \cite{MSWBases,CS1,CS2,CS3}.)

\medskip

\noindent
\textit{\cref{sec:hyperbolic}.}  
We prove the analog of \cref{main} in the coefficient-free case.
This coefficient-free analog is a consequence of \cref{main} and is not needed for the proof of \cref{main}.
However, it is included here as a simplified version of the proof of \cref{main}.
More specifically, the coefficient-free case uses only the fairly standard hyperbolic geometry of marked surfaces (avoiding the complications of opened surfaces and the corresponding lambda lengths and tropical lambda lengths).
It is hoped that the separate treatment of this simpler coefficient-free case will serve to illustrate a key idea of the proof before introducing the complications of the full proof.
The basic structure of the proof is an induction on the number of elements of $P_\alpha$.
The key idea is to pass to a cover of part of the surface and lift the hyperbolic metric to the cover, so that the exchange relations established in \cref{sec:special} serve as the inductive step.
(In the simplest case of plain-tagged arcs, the combinatorics of passing to a cover is similar to the construction in \cite[Section~7]{MusikerSchifflerWilliams}.
The construction here for arbitrary arcs and the lifting of the hyperbolic metric appear to be new.)

\medskip

\noindent
\textit{\cref{sec:tropical}.}  
We prove the full version of \cref{main}.
The proof structure is the same as the simplified proof in \cref{sec:hyperbolic} and reuses many of the tools, but with the full machinery of opened surfaces, lambda lengths, and tropical lambda lengths.

\medskip

\noindent
\textit{\cref{geo sec}.}  
The construction that takes an arc and defines the associated cluster variables as a laminated lambda length \cite[Definition~15.3]{FominThurston} is valid more generally for tagged geodesics connecting marked points (relaxing conditions such as the requirement that the geodesic not cross itself).
We extend \cref{main} to these more general tagged geodesics as \cref{extended}.
We also give a combinatorial characterization of tagged geodesics as \cref{tagged geo prop}.

\medskip

Throughout the paper, we assume the most basic background on marked surfaces in the sense of Sergey Fomin, Michael Shapiro, and Dylan Thurston \cite{FominShapiroThurston,FominThurston}.
Specifically, we use without explanation the material on (tagged) arcs, (tagged) triangulations signed adjacency matrices from \cite[Sections~2--4~\&~7]{FominShapiroThurston} (summarized in \cite[Section~5]{FominThurston}) and the material on laminations and shear coordinates from \cite[Sections~12--13]{FominThurston}.
Later, in support of the proof, we will review results of \cite{FominThurston} that use hyperbolic geometry and tropical hyperbolic geometry to realize cluster variables (\cite[Sections~7--8]{FominThurston} for the simpler coefficient-free case in \cref{sec:hyperbolic} and \mbox{\cite[Sections~9--10~\&~14--15]{FominThurston}} for the full proof in \cref{sec:tropical}).
We also assume the most basic background on cluster algebras of geometric type, $F$-polynomials, and $\g$-vectors from \cite[Sections~2--3~\&~5--6]{ca4}.
Aside from this background, the paper is completely self-contained.


\section{The theorem}
\label{sec:formula}
Let $(\S,\M)$ be a marked surface and let $T$ be a tagged triangulation of $(\S,\M)$.
We will take $T$ to have all arcs tagged plain, except possibly at some punctures $p$ incident to exactly two tagged arcs of $T$, identical except for opposite taggings at $p$.
(This is a standard reduction.
If $T$ does not have this property, we can modify it by reversing all taggings at certain punctures, the only cost being that we must  make the same modification of taggings in all arcs $\alpha$ for which we compute the cluster variable $x_\alpha$.)  
As usual, we will pass between $T$ and the ordinary triangulation $T^\circ$ with self-folded triangles replacing pairs of coinciding arcs with opposite tagging at one puncture, considering $T$ and $T^\circ$ to be different ways of viewing the same combinatorial data.

Let~$B(T)$ be the signed adjacency matrix of $T$, with entries $b_{\alpha\gamma}$ indexed by arcs $\alpha$ and $\gamma$ in $T$.
We associate a principal-coefficients cluster algebra $\A_\bullet(T)$ to $(\S,\M)$ and $T$ as usual.
Specifically, extend $B(T)$ by appending the identity matrix below it.
Let $\set{x_\gamma}{\gamma\in T}$ be the initial cluster variables, and let $\set{y_\gamma}{\gamma\in T}$ be the tropical variables (which are also the coefficients at the initial seed).
For each $\gamma\in T$, let $\hy_\gamma=y_\gamma\prod_{\beta\in T}x_\beta^{b_{\beta\gamma}}$.

Given a tagged arc $\alpha$ in $(\S,\M)$, let $x_\alpha$ be the associated cluster variable.
For each tagged arc~$\alpha$, we will define the Laurent monomial $\g_\alpha$ in the initial cluster variables, the poset $P_\alpha$, and the map~$w:P_\alpha\to\set{\hy_\beta}{\beta\in T}\cup\set{\hy_\beta/\hy_\gamma}{\beta\neq\gamma\in T}$ that appear in the following theorem.

\begin{theorem}\label{main}
The principal-coefficients cluster variable associated to a tagged arc $\alpha$ is
\begin{equation}
\label{main eq}
x_\alpha=\g_\alpha\cdot\sum_I \prod_{e \in I} w(e),
\end{equation}
where the sum is over all order ideals $I$ in $P_\alpha$.
\end{theorem}

The situation is simpler in the unpunctured case, or more generally when the triangulation~$T^\circ$ has no self-folded triangles, because in that case the weights $w(e)$ are all of the form $\hy_\alpha$ for $\alpha\in T$.
We write $\arc(e)$ for the tagged arc $\alpha\in T$ such that $w(e)=\hy_{\arc(e)}$.
The following corollary is a special case of \cref{main}.
\begin{corollary}\label{main unpunct}
If the triangulation $T^\circ$ has no self-folded triangles, then the principal-coefficients cluster variable associated to an arc $\alpha$ is
\begin{equation*}
x_\alpha=\g_\alpha\cdot\sum_I \prod_{e\in I}\hy_{\arc(e)},
\end{equation*}
where the sum is over all order ideals $I$ in $P_\alpha$.
\end{corollary}

In the remainder of this section, we provide the precise definitions of the Laurent monomial $\g_\alpha$, the poset $P_\alpha$, and the weight map~$w$ that appear in \cref{main,main unpunct}.

\subsection{Definition of the Laurent monomial~$\g_\alpha$}
\label{subsec:definition_g_alpha}

To define $\g_\alpha$, we define a curve $\kappa(\alpha)$ that coincides with $\alpha$ except near the endpoints of $\alpha$.
Specifically, if $\alpha$ has an endpoint $p$ at a marked point on the boundary of $\S$, then $\kappa(\alpha)$ ends on a boundary segment, at a point $q$ very near $p$, chosen so that the path along the boundary from~$p$ to $q$ keeps $\S$ on the left.
If $\alpha$ has an endpoint at a puncture $p$, where $\alpha$ is tagged plain, then $\kappa(\alpha)$ spirals clockwise into $p$.
If $\alpha$ is tagged notched at $p$, then $\kappa(\alpha)$ spirals counterclockwise into~$p$.
(This definition follows \cite[Section~5]{unisurface} and is a modification of \cite[Definition~17.2]{FominThurston}.)

We define 
\begin{equation}
\label{g def}
\g_\alpha=\prod_{\gamma\in T}x_\gamma^{-b_{\gamma}(T,\kappa(\alpha))},
\end{equation}
where $b_{\gamma}(T,\kappa(\alpha))$ stands for the $\gamma$-entry of the shear coordinates of $\kappa(\alpha)$ with respect to $T$.

\begin{example}\label{main ex g}
\cref{main ex g fig} shows $\kappa(\alpha)$ for $\alpha$ as in \cref{main ex}.
Also shown are the contributions ($-1$,~$0$~or $1$, abbreviated as $-$, $\circ$, or $+$) to $b_\gamma(T,\kappa(\alpha))$ for each $\gamma$.
As part of the definition of shear coordinates, to find the shear coordinate $b_9(T,\alpha)$, we must reverse the spiral on $\alpha$ inside the self-folded triangle containing arc $9$ and reverse the roles of arcs $9$ and $10$ (as shown in the right picture in \cref{main ex g fig}).
Similarly, to find the shear coordinate $b_2(T,\alpha)$, we must reverse the roles of arcs $2$ and $3$, but since $\alpha$ has no spiral inside the self-folded triangle containing arc $2$, this amounts to setting $b_2(T,\alpha)$ equal to $b_3(T,\alpha)$.
Keeping in mind the negative sign in \eqref{g def}, we have justified the assertion in \cref{main ex} that $\g_\alpha=\frac{x_5x_6x_8}{x_1x_4x_7x_9}$.
\begin{figure}
\includegraphics{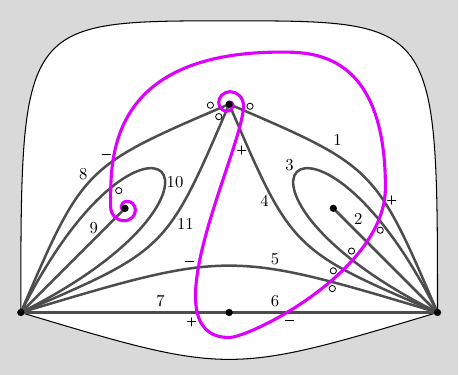}
\qquad
\scalebox{1.1}{\includegraphics{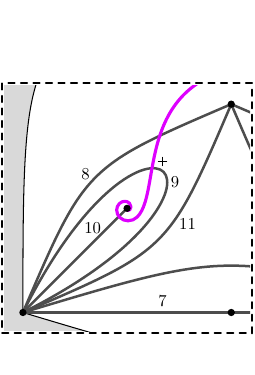}}
\caption{Finding the monomial $\g_\alpha$.}
\label{main ex g fig}
\end{figure}
\end{example}

Versions of the following proposition appear in \cite[Theorem~10.0.5]{DL-F} and \cite[Corollary~6.15]{MSWBases}.
The full proposition appears as \cite[Proposition~5.2]{unisurface} and also as \cite[Lemma~8.6]{FeTu}, where it is generalized to orbifolds.

\begin{proposition}\label{g shear}
Given a tagged arc $\alpha$, the $\g$-vector of $x_\alpha$ has $\gamma$-entry $-b_{\gamma}(T,\kappa(\alpha))$ for all $\gamma\in T$.
\end{proposition}

In light of \cref{g shear}, we call $\g_\alpha$ the \newword{$\g$-vector} of~$\alpha$.
The proofs of \cref{g shear} in \cite[Proposition~5.2]{unisurface} and \cite[Lemma~8.6]{FeTu} both rely on general structural results on cluster algebras.
\cref{g shear} is an immediate corollary of \cref{main}, by the definition of the $\g$-vector.
Since we do not use \cref{g shear} in any arguments in this paper, the proof of \cref{main} given here in particular constitutes a new, direct proof of \cref{g shear}.

By \cite[Corollary~6.3]{ca4}, the content of \cref{main} is \cref{g shear} plus the assertion that $\sum_I \prod_{e\in I}w(e)$ is the $F$-polynomial of $x_\alpha$ in the variables $\hy_i$.
(When we apply \cite[Corollary~6.3]{ca4}, the denominator is $1$ because we are in the principal-coefficients case.)

\subsection{Definition of the poset~$P_\alpha$ and the weight map~$w$}
\label{subsec:definition_P_alpha_and_w}

We now proceed to define~$P_\alpha$ and~$w$.
The key point is that we can describe $\alpha$ completely by giving the sequence of arcs of $T^\circ$ that~$\alpha$ crosses and noting whether~$\alpha$ is tagged notched or plain at its endpoints.

There are degenerate cases in the construction of $P_\alpha$, occurring when $\alpha$ coincides, up to tagging, with an arc in $T$.
We assume for the moment that we are not in a degenerate case.

To define $P_\alpha$, we first define a smaller poset $P^\cross_\alpha$.
The tagged arc $\alpha$ is defined up to isotopy, and (as is standard) we assume up to isotopy that as we move along $\alpha$ in a monotone fashion, when $\alpha$ crosses an arc $\gamma$ of $T^\circ$, it crosses a different tagged arc before ever crossing $\gamma$ again.
In particular, there is a finite set of points where $\alpha$ crosses arcs of $T^\circ$.
This set of points is the ground set of~$P^\cross_\alpha$.
(The endpoints of $\alpha$ are explicitly excluded from the ground set of~$P^\cross_\alpha$.)
Two points in the ground set are called \newword{adjacent} if one can move from one point to the other along $\alpha$ without crossing any other arc of $T^\circ$.
The cover relations in~$P^\cross_\alpha$ come from pairs of adjacent points:
Given adjacent points $e$ and $f$ with $e$ in an arc $\beta$ of $T$ and $f$ in an arc $\gamma$ of $T^\circ$, as one moves along $\alpha$ from $e$ to~$f$, one cuts off a corner of a triangle of $T^\circ$.
We set $e>f$ if the cut corner is to the right of $\alpha$ (as we move along $\alpha$ from $e$ to $f$) or $e<f$ if the cut corner is on the left.
We define $P^\cross_\alpha$ to be the transitive closure of these cover relations.

In most cases, the weight function on $P^\cross_\alpha$ has $w(e)=\hy_\beta$, where $\beta$ is the arc of $T^\circ$ containing~$e$.
There are two kinds of exceptions, both having to do with self-folded triangles.
First, if $\beta$ is the interior edge of a self-folded triangle of $T^\circ$ and $\gamma$ is the exterior edge of the self-folded triangle and if~$\alpha$ passes through $\gamma$ before and after passing through $\beta$ at $e$, then $w(e)=\frac{\hy_\beta}{\hy_\gamma}$.
(In keeping with the usual convention for passing between the tagged triangulation $T$ and the ordinary triangulation~$T^\circ$, the ordinary arc $\gamma$ corresponds to the tagged arc that agrees with $\beta$ but is tagged notched at the interior point of the self-folded triangle.)
Second, suppose $\alpha$ is tagged notched at a puncture $p$ inside a self-folded triangle with interior edge $\beta$ and exterior edge~$\gamma$.
If $e\in\gamma$ is the first point where $\alpha$ crosses an arc of $T^\circ$ after leaving $p$, then $w(e)=\hy_\beta$.

In many cases, we define $P_\alpha=P^\cross_\alpha$.
Specifically, in the non-degenerate cases, this happens if and only if both endpoints of $\alpha$ have the following property:
Either $\alpha$ is tagged plain at the endpoint or the endpoint is a puncture inside a self-folded triangle.

Before defining $P_\alpha$ in the cases where it is larger than $P^\cross_\alpha$, we give an example of the simpler case (\cref{simple Palpha ex}).
In this example (as in \cref{main ex} and all further examples), we suppress the actual elements of $P_\alpha$ (which are certain points in $\S$) and instead show labels that indicate weights:
The arcs of $T$ are numbered, and an element of $P_\alpha$ is labeled either with an integer~$i$ to indicate a weight $\hy_i$ or with a formal quotient $\frac ij$ to indicate a weight $\frac{\hy_i}{\hy_j}$.

\begin{example}\label{simple Palpha ex}
This example is just like \cref{main ex}, except that $\alpha$ is tagged plain at the top-middle puncture in the picture, as shown in left picture of \cref{simple Palpha ex fig}.
In this case, $\g_\alpha=\frac{x_5x_6x_8}{x_1x_7x_9}$ and~$P_\alpha=P^\cross_\alpha$ is shown on the right of \cref{simple Palpha ex fig}.
\begin{figure}
\includegraphics{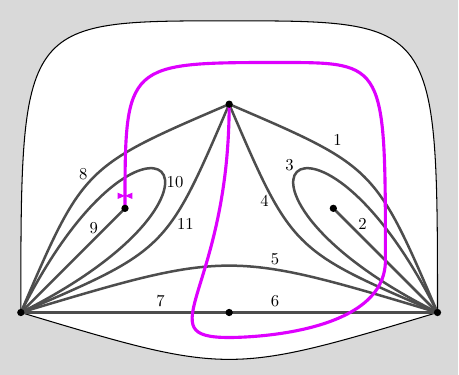}
\qquad\quad
\scalebox{1.1}{\includegraphics{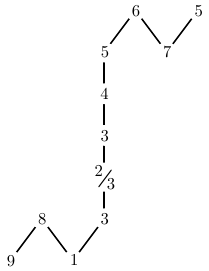}}
\caption{An example where $P_\alpha=P^\cross_\alpha$.}
\label{simple Palpha ex fig}
\end{figure}
\end{example}

The Hasse diagram of $P^\cross_\alpha$ is, as a graph, isomorphic to a path.
Thus it has two ends, given by the last points where $\alpha$ crosses an arc of $T^\circ$ before its endpoints.
At each endpoint $p$ of $\alpha$ that is tagged notched and is not inside a self-folded triangle, we augment $P^\cross_\alpha$ by adding another chain~$C$.
This chain $C$ is constructed in a similar manner to $P^\cross_\alpha$, as follows.
From a point on $\alpha$ near the endpoint~$p$, we circle around $p$, crossing all arcs of $T^\circ$ that are incompatible with $\alpha$ because of the disagreement in tagging at the endpoint.
The elements of $C$ are the points of crossing, they are ordered according to whether they cut corners left or right just as in $P^\cross_\alpha$ (necessarily with either all left or all right), and they are labeled just as the points in $P^\cross_\alpha$.
Let $e\in P^\cross_\alpha$ be the last point where $\alpha$ crosses an arc of $T^\circ$ before reaching $p$.
We attach $C$ to $P^\cross_\alpha$ with the top element of $C$ covering $e$ and the bottom element of $C$ covered by $e$.
(Note that because $p$ is not inside a self-folded triangle, $C$ has at least two elements, so its top and bottom elements are different.)
In the simplest case, where $\alpha$ actually is one of the arcs in $T$, then $\g_\alpha=x_\alpha$ and $P_\alpha$ is the empty poset.

\begin{examplebis}\label{chain ex one}
We continue and explain \cref{main ex}, which contrasts with \cref{simple Palpha ex}.
We take $T^\circ$ and $\alpha$ as shown in the left picture of \cref{main ex fig}.
As explained in \cref{main ex g}, $\g_\alpha=\frac{x_5x_6x_8}{x_1x_4x_7x_9}$.
The poset $P^\cross_\alpha$ is the poset shown in \cref{simple Palpha ex fig}.
Since $\alpha$ is tagged notched at the top-middle puncture, we adjoin a $4$-element chain $C$, labeled $4,1,8,11$ from bottom to top, as shown in the right picture of \cref{main  ex fig}.
There is no chain attached to the other end of $P_\alpha^\cross$, because that endpoint is the interior vertex of a self-folded triangle.
\end{examplebis}

\begin{remark}\label{polygon heuristic}
We offer a heuristic explanation for the addition of a chain at an endpoint where~$\alpha$ is tagged notched (at a puncture not inside a self-folded triangle).
The heuristic is equally good in all such cases, but we explain it using \cref{main ex,main ex fig}.
In the surfaces model, when an arc is tagged notched, one often thinks of it as ``going around'' the puncture rather than ending there.  
As $\alpha$ approaches the puncture (at the top of \cref{main ex fig}), we have recorded $7\covered5$ in~$P^\cross_\alpha$, but instead of stopping there and later adding a chain to make $P_\alpha$, we make $P_\alpha$ directly by continuing around the puncture (in either direction) and adding elements for the arcs we cross.
Thus we record $5\covered11\covers8\covers1\covers4$.
Then continuing along $\alpha$ in the opposite direction, we see that the element labeled $4$ is covered by the same element as before, labeled $5$.
\end{remark}

\pagebreak

\begin{example}\label{chain ex both}
We also give an example where neither endpoint of $\alpha$ is inside a self-folded triangle and both endpoints are tagged notched, so that two chains are adjoined to $P_\alpha^\cross$ to make $P_\alpha$.
For $T^\circ$ and $\alpha$ as shown in \cref{chain ex both fig}, we have $\g_\alpha=\frac{x_6}{x_4x_5}$.
The poset $P_\alpha^\cross$ is a chain labeled $3\covered2\covered1\covered3\covered6$.
At the top element of $P_\alpha^\cross$, we adjoin a chain labeled $4\covered7$, and the bottom, we adjoin a chain labeled $5\covered4\covered6$, to obtain $P_\alpha$ as shown in \cref{chain ex both fig}.
The cover relation $4\covered7$ in one adjoined chain ``disappears'' in the Hasse diagram (hence we show it as a dotted line), because $4\covered6\covered7$ in~$P_\alpha$.
\begin{figure}
\includegraphics{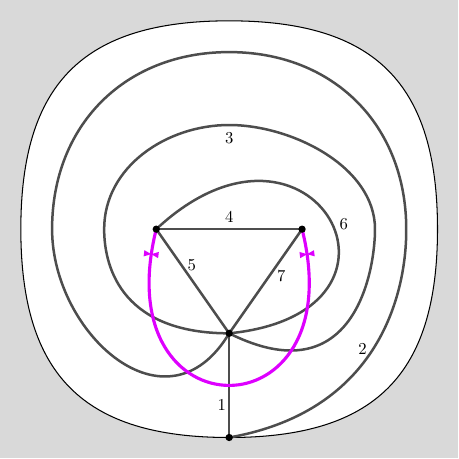}
\qquad\quad
\scalebox{1.1}{\includegraphics{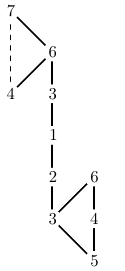}}
\caption{An example where $P_\alpha$ is $P^\cross_\alpha$ plus two chains.}
\label{chain ex both fig}
\end{figure}
\end{example}

We now consider the degenerate cases, where $\alpha$ coincides, up to tagging, with an arc in $T$.
If~$\alpha$ is actually an arc in $T$, then the $\g_\alpha$ is the initial cluster variable $x_\alpha$ and $P_\alpha$ is the empty poset, whose only order ideal is the empty set, with weight $1$.
The other degenerate cases are where $\alpha$ coincides with an arc of $T$ but has different taggings.

First, consider the case where $\alpha$ is tagged notched at only one endpoint and otherwise agrees with an arc $\beta$ of $T$ that is not the interior edge of a self-folded triangle in $T^\circ$.
Then $\alpha$ and $\beta$ are compatible, so $\beta$ is not in the chain $C$ coming from arcs of $T$ that are incompatible with $\alpha$ because of tagging at that endpoint.
In this case, $P_\alpha$ consists only of the chain $C$, as illustrated in \cref{in T one}.

\begin{example}\label{in T one}
The arc $\alpha$ shown in \cref{in T one fig} agrees with arc $1$ except for its tagging at one point. 
Thus $P_\alpha$ is the chain $C$ illustrated in the figure.
In this case, $\g_\alpha=\frac{x_2}{x_3}$.
\begin{figure}
\includegraphics{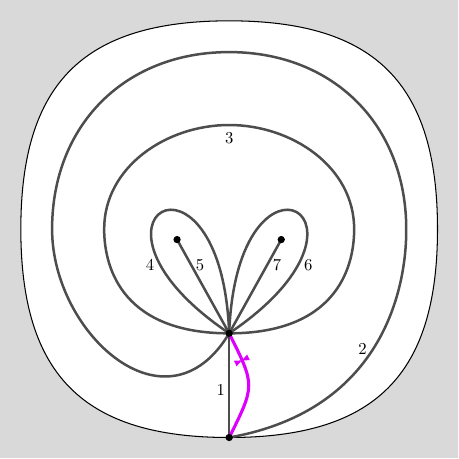}
\qquad\quad
\scalebox{1.05}{\includegraphics{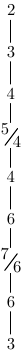}}
\caption{An example where $P^\cross_\alpha$ is empty but a chain is ``attached'' to it.}
\label{in T one fig}
\end{figure}
\end{example}

We next consider the case where $\alpha$ coincides with the interior edge of a self-folded triangle in~$T^\circ$.
In other words, $\alpha$ coincides with two arcs $\beta,\beta'$ of $T$ that are the same up to tagging, but~$\alpha$ is tagged differently than both.  
In this case, since $\alpha$ does not cross any arcs in $T^\circ$, the poset $P_\alpha^\cross$ is empty.  
Exactly one of $\beta$ and $\beta'$ is compatible with $\alpha$, without loss of generality $\beta'$.
As before, $P_\alpha$ is only the chain $C$ consisting of the arcs that are incompatible with $\alpha$ at its notched endpoint, but in this case, two elements labeled $\hy_\beta$ appear in the chain, one at the top and one at the bottom, as illustrated in \cref{self-fold coinciding}.

\begin{example}\label{self-fold coinciding}
In \cref{self-fold coinciding fig}, the arc $\alpha$ shown coincides with arc $7$ except for tagging.
It also coincides with arc $6$ except for tagging, when the latter is represented as a tagged arc rather than an ordinary arc.
The arc $\alpha$ is compatible with arc $6$ and incompatible with arc~$7$.
The poset $P^\cross_\alpha$ is empty, and~$P_\alpha$ is a chain as shown in the figure.
In this case, $\g_\alpha=\frac{1}{x_7}$.
\begin{figure}
\includegraphics{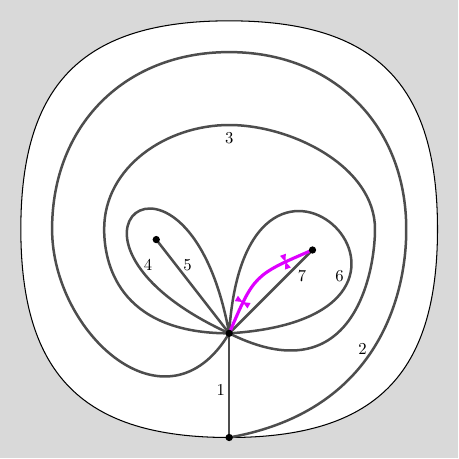}
\qquad\quad
\scalebox{1.05}{\includegraphics{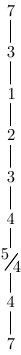}}
\caption{An example where $\alpha$ coincides with the interior edge of a self-folded triangle.}
\label{self-fold coinciding fig}
\end{figure}
\end{example}

\begin{remark}\label{another heuristic}
Note that the heuristic of \cref{polygon heuristic} applies equally well to the cases treated in \cref{in T one,self-fold coinciding}, the only difference being that the poset~$P^\cross_\alpha$ is empty.
\end{remark}

Finally, if $\alpha$ agrees, except for tagging, with an arc $\beta$ of $T$ that is not the interior edge of a self-folded triangle in $T^\circ$ but is tagged notched at both endpoints, then two chains are ``attached'' to the empty poset $P_\alpha^\cross$.
Two elements weighted $\hy_\beta$ are in each chain, and the chains intersect and interact.
Each chain is obtained as before by following a path around the puncture, except that the path starts and ends by crossing $\beta$, so that an element weighted $\hy_\beta$ is at both the top and bottom of the chain.
The poset $P_\alpha$ is a union of the two chains, identified at their top and bottom elements, with two additional cover relations, each having the atom in one chain below the coatom in the other chain, as illustrated in \cref{in T two}.

\pagebreak

\begin{example}\label{in T two}
In \cref{in T two fig}, the arc $\alpha$ coincides with arc $6$ except that $\alpha$ is tagged notched at both endpoints (neither endpoint being inside a self-folded triangle).
The poset $P_\alpha$ is the union of two chains labeled $6\covered1\covered2\covered6$
and $6\covered3\covered4\covered5\covered6$.
The top elements of the chains are identified, as are the bottom elements, and there are additional cover relations as shown.
In this case, $\g_\alpha=\frac{1}{x_5}$.
\end{example}
\begin{figure}
\includegraphics{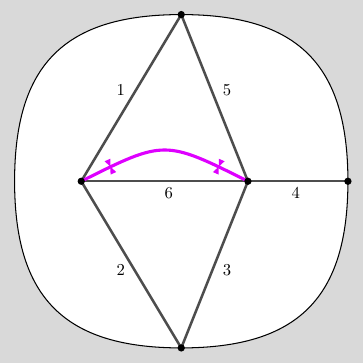}
\qquad\quad
\scalebox{1.05}{\includegraphics{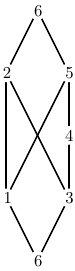}}
\caption{An example where $P^\cross_\alpha$ is empty but two chains are ``attached'' to it.}
\label{in T two fig}
\end{figure}

\begin{remark}\label{go around both}
We offer a heuristic explanation of the form of $P_\alpha$ in the cases where $\alpha$ agrees with an arc in $T$ (not the interior edge of a self-folded triangle) but is notched at both endpoints.
We will give the explanation specifically in the context of \cref{in T two}, but the explanation works equally well in general.
The idea is to think of~$P_\alpha$ as the union of four ``polygons'', similar to the polygons that have appeared in earlier examples.

Since the arc $\alpha$ ends at the left puncture with the wrong tagging, it is incompatible with arc $1$ of $T$.
Imagine that the arc $\alpha$, instead of ending at the left puncture, was incompatible with arc $1$ by actually crossing arc $1$.
In this case, $P_\alpha^\cross$ would have an element labeled $1$, to which we would adjoin a chain to make the polygon $6\covered3\covered4\covered5\covers1\covers6$.
But also, the notch in $\alpha$ at the left puncture makes it incompatible with arc $2$, so, drawing $\alpha$ below arc $6$ and letting it cross arc $2$ instead of ending at the puncture, we obtain instead the polygon $3\covered 4\covered5\covered6\covers2\covers3$.
By the same process at the right puncture, we obtain the polygons $1\covered2\covered6\covers5\covers1$ and $6\covered1\covered2\covers3\covers6$.
The poset $P_\alpha$ is the union of these four polygons.
\end{remark}

We now offer a more complicated example of the case where $\alpha$ agrees with an arc in $T$ (not the interior edge of a self-folded triangle) but is notched at both endpoints.

\begin{example}\label{in T two same}
In \cref{in T two same fig}, the arc $\alpha$ coincides with arc $3$ except that $\alpha$ is tagged notched at both endpoints (neither endpoint being inside a self-folded triangle).
The poset $P_\alpha$ is the union of two chains labeled 
\[3\covered2\covered1\covered3\covered6\covered7/6\covered6\covered4\covered5/4\covered4\covered3\]
and
\[3\covered6\covered7/6\covered6\covered4\covered5/4\covered4\covered3\covered2\covered1\covered3,\]
with top elements identified and bottom elements identified, and with additional cover relations as shown.
In this case, $\g_\alpha=\frac{1}{x_3}$.
\begin{figure}
\includegraphics{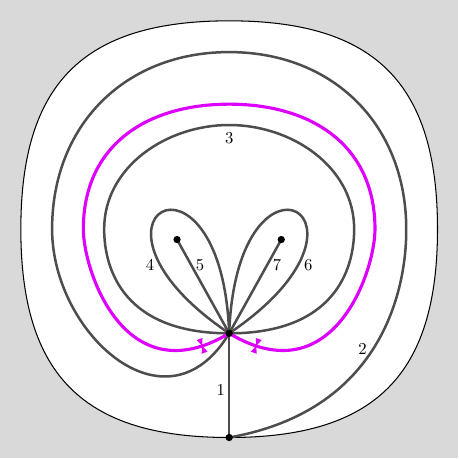}
\qquad\quad
\scalebox{1.05}{\includegraphics{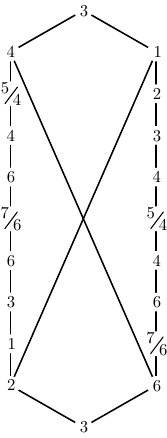}}
\caption{Another example where $P^\cross_\alpha$ is empty but two chains are ``attached'' to it.}
\label{in T two same fig}
\end{figure}
\end{example}

To conclude the discussion of the case where $\alpha$ agrees with an arc in $T$ (not the interior edge of a self-folded triangle) but is notched at both endpoints, we give an example where one of the two chains contains only three elements, so that cover relations ``disappear''. 
This happens when~$\alpha$ agrees with an arc in $T$ that is in two triangles of $T$ that combine to form a once-punctured digon.
It is impossible for both chains to contain only three elements, because if so, then $\S$ is a sphere with three punctures, which is ruled out in the definition of a marked surface.

\begin{example}\label{disappear1}
In \cref{dis 1 fig}, $P_\alpha$ is the union of two chains labeled $1\covered 2\covered 1$ and $1\covered3\covered4\covered1$ with top elements identified and bottom elements identified.
\begin{figure}
\includegraphics{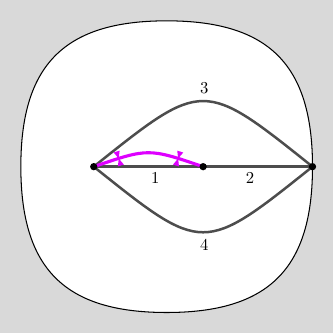}
\qquad\quad
\scalebox{1.05}{\includegraphics{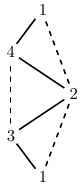}}
\caption{$P^\cross_\alpha$ is empty, two chains are ``attached'' to it, and cover relations ``disappear''.}
\label{dis 1 fig}
\end{figure}
Because of the additional cover relations that connect the atom in one chain to the coatom in the other chain, two cover relations between elements labeled $1$ and $2$ disappear and the cover relation between elements labeled $3$ and $4$ also disappears (hence we show them as dotted lines).
In this example, $\g_\alpha=\frac1{x_1}$.

The cover relation involving $3$ and $4$ disappears because the chain labeled $1\covered3\covered4\covered1$ has only four elements.  
In cases where one chain has three elements and the other has more than four, the only cover relations that disappear are those in the chain with three elements.
\end{example}

\subsection{A slight reduction of the theorem}\label{red sec}
We now reduce \cref{main} to the case where $\alpha$ has the following property:
When $\alpha$ has an endpoint that is the interior vertex of a self-folded triangle in $T^\circ$, the tagging of $\alpha$ is plain at that endpoint.
This kind of argument is standard in the marked surfaces model.

First, we mention a general combinatorial symmetry that is built into the marked surfaces model \cite{FominShapiroThurston,FominThurston}.
(Indeed, we have already used this symmetry to assume that the tagged triangulation $T$ has all arcs tagged plain, except possibly at punctures $p$ incident to exactly two tagged arcs of $T$ with opposite taggings at $p$.)

Suppose $T$ is a tagged triangulation and suppose $p$ is a puncture.
For any tagged arc $\gamma$, let~$\gamma'$ be the arc obtained by switching the tagging of $\gamma$ at $p$ (or leaving $\gamma$ unchanged if it does not have an endpoint at $p$).
Let $T'$ be the tagged triangulation $\{\gamma':\gamma\in T\}$.
Suppose we are given a tagged arc $\alpha$ and an expression for the cluster variable $x_\alpha$ with principal coefficients at $T$, in terms of the initial cluster variables $x_\gamma$ for $\gamma\in T$ and the tropical variables $y_\gamma$ for $\gamma\in T$.
If we take that expression and replace every $x_\gamma$ with $x_{\gamma'}$ and every $y_\gamma$ with $y_{\gamma'}$, we obtain an expression for the cluster variable~$x_{\alpha'}$ with principal coefficients at $T'$, in terms of the initial cluster variables~$x_{\gamma'}$ for~$\gamma'\in T'$ and the tropical variables $y_{\gamma'}$ for $\gamma'\in T'$.

This symmetry exists because the signed adjacency matrix of $T$ and $T'$ are the same by definition (except for being indexed by arcs $\gamma$ or arcs $\gamma'$) and because flips of tagged arcs commutes with tag-switching at $p$.
Thus, if we obtain $\alpha$ by a sequence of flips starting with $T$ and write an expression for $x_\alpha$ using exchange relations, the corresponding sequence of flips starting with $T'$ writes the same expression for $x_{\alpha'}$.

Now, suppose $T$ is a tagged triangulation with all arcs tagged plain, except possibly at some punctures $p$ incident to exactly two tagged arcs of $T$, identical except for opposite taggings at~$p$.
Suppose $p$ is one such puncture, so that $p$ is the interior vertex of a self-folded triangle in the corresponding ordinary triangulation $T^\circ$.
If $\beta$ and $\gamma$ are the two tagged arcs associated to the self-folded triangle, then the operation of switching tags at such a $p$ simply switches $\beta$ and $\gamma$.
(That is, $\beta'=\gamma$ and $\gamma'=\beta$.)

We see that if $\alpha$ is tagged notched at a puncture $p$ that is the interior vertex of a self-folded triangle, then an expression for $x_\alpha$ in terms of the initial cluster variables and principal coefficients at $T$ can be obtained from any such expression for $x_{\alpha'}$ by replacing all instances of $x_\beta$ by $x_\gamma$ and vice versa.
We check that the right side of \cref{main eq} has the same symmetry.
First, $\kappa(\alpha)$ and~$\kappa(\alpha')$ are related by reversing their spirals at $p$.
By the definition of shear coordinates with respect to tagged triangulations \cite[Definition~12.1]{FominThurston}, we see that $\g_\alpha$ and $\g_{\alpha'}$ are related by replacing all instances of $x_\beta$ by $x_\gamma$ and vice versa.
By inspection of the definition in \cref{subsec:definition_P_alpha_and_w}, we also see that the weighted posets $P_\alpha$ and $P_{\alpha'}$ are related in the same way.
(In the non-degenerate case, see \cref{simple Palpha ex} and in the degenerate case, see \cref{self-fold coinciding}.)

We have established the following reduction of \cref{main}.

\begin{proposition}\label{reduce}
Assume that \cref{main eq} holds for every tagged arc $\alpha$ with the following property:
When $\alpha$ has an endpoint that is the interior vertex of a self-folded triangle in $T^\circ$, the tagging of $\alpha$ is plain at that endpoint.
Then \cref{main eq} holds for every tagged arc.
\end{proposition}


\section{Tidy exchange relations}\label{sec:special}
Given a finite poset $P$ with weights $w:P\to\set{\hy_\beta}{\beta\in T}\cup\set{\hy_\beta / \hy_\gamma}{\beta\neq\gamma\in T}$, write 
\[{F(P)=\sum_I \prod_{e\in I}w(e)},\]
where the sum is over all ideals $I$ in $P$.
(This notation suppresses the dependence of $F(P)$ on the map $w$.)  
The content of \cref{main} is that for any tagged arc $\alpha$, the cluster variable indexed by $\alpha$ has $\g$-vector monomial $\g_\alpha$ and $F$-polynomial $F(P_\alpha)$, where the weight function on $P_\alpha$ is~$w$ as defined in \cref{subsec:definition_P_alpha_and_w}.
By \cite[Proposition~5.1]{ca4}, the $F$-polynomials of cluster variables are uniquely determined by the condition that initial cluster variables have $F$-vector $1$ and the exchange relations \cite[(5.3)]{ca4}.

Write $F_\alpha$ for the $F$-polynomial of a tagged arc $\alpha$.
In this section, we prove that for certain special choices of $\alpha$ and $T$ and an initial arc $\gamma\in T$, the exchange relation that writes $F_\alpha\cdot F_\gamma$ in terms of other $F$-polynomials $F_\beta$ also holds with $F_\alpha$, $F_\gamma$, and each $F_\beta$ replaced by $F(P_\alpha)$, $F(P_\gamma)$ and the corresponding polynomials $F(P_\beta)$.
We also show that this exchange relation records a natural decomposition of the set of order ideals of $\alpha$.
Finally, we show that the relationship between $\g_\alpha\cdot\g_\gamma$ and the $\g_\beta$ is correct.

\subsection{Exchange relation on $F$-polynomials}\label{subsec:exchange}
Exchange relations on $F$-polynomials in general are given in \cite[Proposition~5.1]{ca4}.
In this section, we interpret them in the surface model using the principal-coefficients version of the material in \cite[Sections~12--13]{FominThurston} that realizes geometric coefficients in terms of laminations.
(See also \cite[Theorem~15.6]{FominThurston}.)

Throughout the section, we assume that the tagged triangulation~$T$ has all taggings plain.
Thus also the ordinary triangulation $T^\circ$ has no self-folded triangles.
Since in this case $T$ and $T^\circ$ coincide, we will only use the notation $T$ in this section.
In this section, we have no need of hyperbolic geometry, but rather we use results of \cite{FominThurston} that were proved using hyperbolic geometry. 

Given an arc $\gamma\in T$, the \newword{elementary lamination} associated to $\gamma$ is a curve $L_\gamma$ that is defined precisely as $\kappa(\gamma)$ was, except with the right and left directions swapped.
Thus $L_\gamma$ agrees with~$\gamma$ except very close to the endpoints, where $L_\gamma$ turns slightly to the right to end on a boundary segment or spiral counterclockwise into a puncture.
(In this section, we only allow plain taggings.  
In general, if $\gamma$ is tagged notched at one of its endpoints, then $L_\gamma$ spirals clockwise into that puncture.)
Starting with principal coefficients at $T$, the extended exchange matrix at another tagged triangulation~$T'$ has coefficient rows given by shear coordinates of elementary laminations with respect to $T'$.
Specifically, in the coefficient row indexed by $\gamma\in T$, the entry in the column indexed by $\alpha\in T'$ is $b_{\alpha}(T',L_\gamma)$, the $\alpha$-entry of the shear coordinates of $L_\gamma$ with respect to $T'$.

Since the coefficients of principal-coefficients extended exchange matrices are determined by shear coordinates of elementary laminations, we can describe exchange relations between $F$-polynomials geometrically in terms of arcs and elementary laminations.
Two distinct tagged arcs admit an exchange relation if and only if one of the following descriptions hold:
\begin{itemize}
\item The two arcs intersect exactly once in their interiors and, if they share endpoints, have the same tagging at shared endpoints; or
\item The two arcs share an endpoint where their taggings disagree, do not intersect in their interiors, and, if they also share the other endpoint, they have the same tagging there.
\end{itemize}

\cref{exch quad} illustrates the exchange relation for arcs that intersect in their interior.
The arcs being exchanged are shown in purple and are horizontal and vertical.  
\begin{figure}
\begin{tabular}{ccc}
\scalebox{1}{\includegraphics{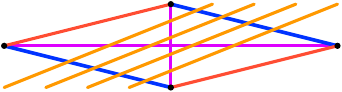}}
&\qquad\quad&
\scalebox{1}{\includegraphics{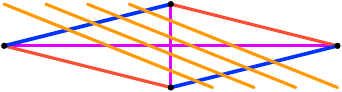}}
\end{tabular}
\caption{Exchange relations for crossing arcs.}
\label{exch quad}
\end{figure}
Each picture of \cref{exch quad} shows four more arcs, colored red and blue and forming a quadrilateral whose diagonals are the purple arcs, and some orange curves crossing the quadrilateral that represent elementary laminations associated to arcs of~$T$.  
There are some complications in general that we are able to avoid here by assuming that the purple arcs have, between them, four distinct endpoints.  
(Otherwise, some of the blue or red arcs may coincide or some red or blue arc may need to be replaced, in exchange relations, by the product of two tagged arcs that coincide except for the tagging at one endpoint.)
There is one complication that we will not be able to avoid:
If any endpoints of purple arcs are tagged notched, then the orange curves agree with the elementary laminations for the arcs of $T$, except that all spirals at those endpoints are reversed.

Elementary laminations (possibly with some spirals reversed) play a role in the exchange relation if they cross a blue arc, then the purple arcs, then the other blue arc, or if they similarly cross a red arc, then the purple arcs, then the other red arc.
Since we have assumed that $T$ has no notched taggings, the elementary laminations associated to arcs in $T$ can be taken to be disjoint from each other.
(This is \emph{almost} true in any case; in general, two of these elementary laminations only intersect when they are associated to a pair of arcs that coincide except for tagging at one endpoint.)
Because the elementary laminations are disjoint from each other, it is impossible to have some laminations crossing blue-purple-blue and at the same time other laminations crossing red-purple-red.
(In the general case, this is also true, for only slightly more complicated reasons.)
Since we have the choice of which pair of arcs to color red and which to color blue, \emph{we take the convention that the orange curves only have blue-purple-blue crossings}, and the pictures in \cref{exch quad} are colored accordingly.

We adopt the notation $F_\purple$ for the product of the $F$-polynomials of two purple arcs shown in \cref{exch quad}, and similarly $F_\blue$ and $F_\red$.
Each orange curve shown is associated to an arc $\gamma\in T$, and $\hy_\orange$ refers to the product of the variables $\hy_\gamma$ for all $\gamma\in T$ whose associated orange curve has a blue-purple-blue intersection with the quadrilateral.
In general, a factor $\hy_\gamma$ can appear in~$\hy_\orange$ more than once, but we will consider those exchange relations where each $\hy_\gamma$ appears at most once.
In light of our coloring convention, the \newword{$F$-polynomial exchange relation} is
\begin{equation}\label{exch rel F}
F_\purple=F_\blue+\hy_\orange\cdot F_\red.
\end{equation}

\cref{exch digon} illustrates the exchange relation for arcs that share a vertex where their taggings disagree.
Each picture of the figure shows the two arcs (in purple and horizontal) being exchanged and a red and blue arc.
\begin{figure}
\begin{tabular}{ccc}
\scalebox{1}{\includegraphics{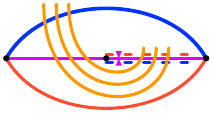}}
&\qquad\quad&
\scalebox{1}{\includegraphics{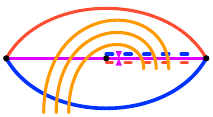}}
\end{tabular}
\caption{Exchange relations for arcs with disagreeing tags.}
\label{exch digon}
\end{figure}
We avoid complications by assuming that the two purple arcs have three distinct endpoints.
The same rule on reversing spirals of elementary laminations applies when non-shared endpoints of the purple arcs are tagged notched. 
The notation $F_\purple$ again refers to the product of the $F$-polynomials of two purple arcs shown.
The figure also shows dotted red and blue arcs (coinciding with one of the purple arcs).
The notation $F_\blue$ refers to the $F$-polynomial of the one \emph{non-dotted} blue arc and $F_\red$ refers to the $F$-polynomial of the one \emph{non-dotted} red arc.
The dotted red and blue arcs are an aid to understanding the orange curves, specifically because we care about blue-purple-blue or red-purple-red crossings, with one of these blues or reds being dotted.
It is again impossible to have both laminations crossing blue-purple-blue and laminations crossing red-purple-red, so we again take the convention that orange curves have only blue-purple-blue crossings.
The notation $\hy_\orange$ again refers to the product of the variables $\hy_\gamma$ for all $\gamma\in T$ whose associated orange curve crosses blue-purple-blue with one of these blues being dotted.
(Again, in general, a factor $\hy_\gamma$ can appear in $\hy_\orange$ more than once, but we will avoid that case.)
Once again, the exchange relation is given by \eqref{exch rel F}.

\begin{remark}\label{F or x}
\cref{exch rel F} gives the exchange relation on $F$-polynomials, interpreting \cite[Proposition~5.1]{ca4} in terms of elementary laminations.
Later, in the proof of \cref{main}, we will also need the principal-coefficients case of the exchange relation on cluster variables, which can be obtained from \cref{exch rel F} by replacing each $F$-polynomial by the corresponding cluster variable and replacing $\hy_\orange$ by $y_\orange$, the product of the $y_\gamma$  for all $\gamma\in T$ whose associated orange curve has a blue-purple-blue intersection.
\end{remark}

\subsection{Tidyness}\label{subsec:tidy}

We now describe conditions on $T$ and $\alpha$ that turn certain exchange relations into decompositions of the set of order ideals of $P_\alpha$.
We have not tried to minimize these conditions, because we can easily satisfy them using the covering arguments in later sections.
We assume as usual that an isotopy representative of $\alpha$ has been chosen to minimize intersections with arcs of~$T$.

We say that $(T,\alpha)$ is \newword{tidy} if:
\begin{itemize}
\item $T$ consists of plain-tagged arcs;
\item each triangle of $T$ has three distinct vertices and three distinct edges (so that in particular each arc of $T$ has two distinct endpoints);
\item $\alpha$ has two distinct endpoints;
\item no arc in $T$ intersects the interior of $\alpha$ more than once;
\item no arc in $T$ both intersects the interior of $\alpha$ and shares an endpoint with $\alpha$;
\item if an arc of $T$ has the same two endpoints as $\alpha$, then $\alpha$ coincides with that arc except possibly for tagging.
\end{itemize}
This tidyness condition implies the following simplifications:
\begin{itemize}
\item $T$ has no self-folded triangles, so the labels of $P_\alpha$ are all of the form $\hy_\gamma$ for $\gamma\in T$.
\item There are no repeated labels in $P_\alpha$, except in the case where $\alpha$ coincides with an arc $\beta$ (not the interior edge of a self-folded triangle) and~$\alpha$ has both ends notched, so that the label $\hy_\beta$ appears twice (as the minimal and maximal element of $P_\alpha$, see \eg \cref{in T two fig}).
\item An arc~$\gamma \in T$ is exchangeable with~$\alpha$ if and only if it crosses~$\alpha$ or it shares with~$\alpha$ an endpoint notched in~$\alpha$ (but doesn't coincide with $\alpha$ up to tagging). In other words, there is an exchange relation exchanging $x_\alpha$ and $x_\gamma$ if and only if $\hy_\gamma$ appears exactly once as a label in $P_\alpha$.
\item For any arc~$\gamma \in T$, the corresponding elementary lamination contributes at most once to~$\hy_\orange$.
\end{itemize}


We now consider the exchange relations for exchanging arcs $\alpha$ and $\gamma$ in the special case when~$(T,\alpha)$ is tidy and $\gamma \in T$ is exchangeable with~$\alpha$.
The top picture of \cref{uncomp} shows the case where $\alpha$ and~$\gamma$ cross in their interiors ($\alpha$ is the horizontal purple arc, while $\gamma$ is the vertical purple arc).
Also pictured (in gray) are the arcs of $T$ that $\alpha$ crosses or shares an endpoint with, and (in black) additional arcs that complete some triangles in $T$, and (in orange) the elementary lamination $L_\gamma$ (which fixes our color choices according to the convention that it should cross the quadrilateral as blue-purple-blue).
\begin{figure}
\scalebox{1}{\includegraphics{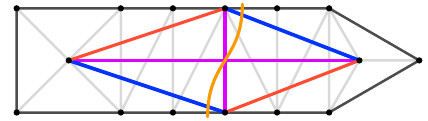}}\\[5pt]
\scalebox{1}{\includegraphics{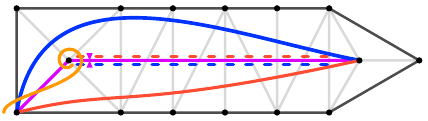}}
\caption{Exchange relations for tidy arcs and initial arcs.}
\label{uncomp}
\end{figure}
In this case, it is impossible for $\alpha$ to coincide with an arc of~$T$ except for tagging, because if so, $\alpha$ could not cross another arc $\gamma$ of $T$.
It is possible that $\alpha$ is tagged notched at either or both endpoints, in which case, the red and blue arcs at those endpoints are also notched.
However, since $(T,\alpha)$ is tidy, the endpoints of $\gamma$ are not tagged notched.

The bottom picture of \cref{uncomp} shows the case where $\alpha$ and $\gamma$ share an endpoint where they disagree on tagging, with additional grey and black arcs and orange curve as in the top picture.
In this case, it is possible that $\alpha$ is also tagged notched at the endpoint it does not share with~$\gamma$, in which case the red and blue arcs at that endpoint are also notched.
Again in this case, the endpoints of $\gamma$ are not tagged notched.

It is important to consider the question of how general the pictures in \cref{uncomp} are.
Here are the ways that $(T,\alpha)$ can differ from the pictures in the tidy case.
\begin{itemize}
\item
The sequence of triangles between endpoints of $\alpha$ can vary.
(This variation corresponds to the ways that $P_\alpha^\cross$ can be any poset whose Hasse diagram, as a graph, is a path.)
\item
One or both of the endpoints of $\alpha$ may be on the boundary.
When an endpoint~$p$ of $\alpha$ is on the boundary, one or more of the triangles incident to $p$ are deleted and correspondingly two of the arcs of $T$ incident to $p$ become boundary segments.
In this case, $\alpha$ is tagged plain at $p$, so the arcs incident to $p$ are irrelevant to the definition of $P_\alpha$.
\item
The number of triangles incident to each endpoint $p$ of $\alpha$ may vary (at least $2$ when $p$ is a puncture and at least $1$ when $p$ is on the boundary).
\item
Some black edges may be identified with each other and/or some marked points (excluding the endpoints of $\alpha)$ may be identified with each other.
\pagebreak
\item
One or both endpoints of $\alpha$ may be the puncture in a once-punctured digon formed as the union of two triangles of $T$.
\cref{uncomp digon} is a version of \cref{uncomp} with these digons at both endpoints of $\alpha$ in each picture.
\begin{figure}
\scalebox{1}{\includegraphics{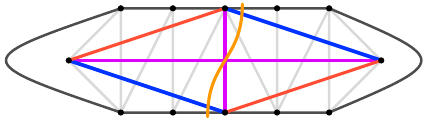}}\\[5pt]
\scalebox{1}{\includegraphics{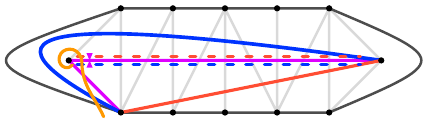}}
\caption{Exchange relations for tidy arcs and initial arcs, with digons.}
\label{uncomp digon}
\end{figure}
The case with digons is not qualitatively different, but instead simply achieves the lower-bound on the number of arcs in $T$ incident to endpoints of~$\alpha$.
(For that reason, we will continue to draw the cases without digons, as in \cref{uncomp}.)
\item
$\alpha$ may coincide with an arc of~$T$.
\end{itemize}

\subsection{Tidy exchange relation on posets}\label{subsec:F}

Assume now that~$(T,\alpha)$ is tidy and consider~$\gamma \in T$ exchangeable with~$\alpha$.
We define $\Pr$ to be the weighted poset~$P_\beta$ when there is a single red arc~$\beta$, or the disjoint union of the two weighted posets $P_\beta$ and~$P_{\beta'}$ when there are two red arcs~$\beta$ and~$\beta'$ in the exchange relation.
Note that~$F(\Pr)$ is then the product of the $F(P_\beta)$ for the (one or two) red arcs in the exchange relation.  
We define $\Pb$ similarly.  
(We could also define $\Pp$, but since $P_\gamma$ is the empty poset, $\Pp$ would always equal $P_\alpha$.)

\begin{proposition}\label{ideal decomp}
Suppose $(T,\alpha)$ is tidy and $\gamma \in T$ is exchangeable with~$\alpha$, and denote by~$e_\gamma$ the element of $P_\alpha$ labeled $\hy_\gamma$.
Then 
\begin{itemize}
\item $P_\blue$ is obtained from $P_\alpha$ by deleting all elements weakly above~$e_\gamma$,
\item $P_\red$ is obtained from $P_\alpha$ by deleting all elements weakly below~$e_\gamma$,
\item $\hy_\orange$ is the product of the weights of the elements weakly below $e_\gamma$ in~$P_\alpha$.
\end{itemize}
Hence, the weighted sum~$F(P_\alpha)$ of order ideals in $P_\alpha$ decomposes into
\begin{itemize}
\item
the weighted sum $F(\Pb)$ of order ideals in $P_\alpha$ that do not contain $e_\gamma$ and 
\item
the weighted sum $\hy_\orange\cdot F(\Pr)$ of order ideals in $P_\alpha$ that contain $e_\gamma$.
\end{itemize}
This yields the \newword{poset exchange relation}
\begin{equation}\label{ideal decomp eq}
F(P_\alpha)=F(\Pb)+\hy_\orange\cdot F(\Pr).
\end{equation}
\end{proposition}

\begin{example}\label{oranges}
To illustrate \cref{ideal decomp} and its proof, we give a series of examples in Figures~\ref{orange1}--\ref{orange8}.
Each figure has purple, red, and blue arcs as in \cref{uncomp} or \cref{uncomp digon}.
The top-left picture shows the arcs of $T$ that intersect the purple arc $\alpha$, colored orange if the corresponding laminations (with spirals reversed as appropriate) have blue-purple-blue intersections and colored gray otherwise.
The top-right picture shows the corresponding laminations, with the same color coding.  
The purple arc $\gamma$ that is being exchanged with $\alpha$ will always appear orange in the pictures, so we identify $\gamma$ explicitly in the figures.
The arc $\gamma$ can also be located using the positions of the red and blue arcs.
In other cases where an arc should be two colors (either red and orange or purple and orange), both colors are shown side by side.
Below the pictures, we identify $\gamma$, give the monomial $\hy_\orange$, and picture the posets $P_\alpha$, $P_\blue$, and $P_\red$.

\begin{figure}[p]
\begin{tabular}{cc}
\scalebox{1}{\includegraphics{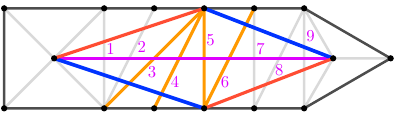}}
&
\scalebox{1}{\includegraphics{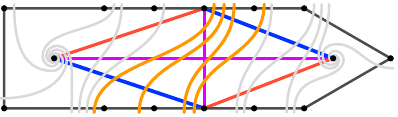}}\\[5pt]
Arc $5$ is $\gamma$
&
$\hy_\orange=\hy_3\hy_4\hy_5\hy_6$\\[5pt]
\end{tabular}
\begin{tabular}{ccc}
\scalebox{1}{\includegraphics{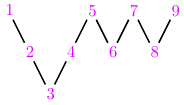}}
&\hspace*{20pt}
\scalebox{1}{\includegraphics{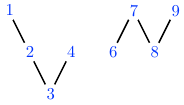}}
\hspace*{20pt}&
\scalebox{1}{\includegraphics{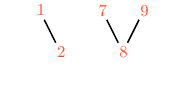}}\\
$P_\alpha$&$P_\blue$&$P_\red$
\end{tabular}
\caption{First example of \cref{ideal decomp}.}
\label{orange1}
\end{figure}

\begin{figure}[p]
\begin{tabular}{cc}\\[5pt]
\scalebox{1}{\includegraphics{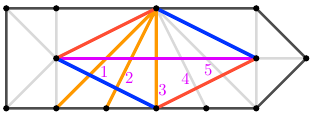}}
&
\scalebox{1}{\includegraphics{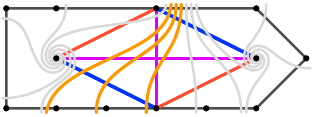}}\\[10pt]
Arc $3$ is $\gamma$
&
$\hy_\orange=\hy_1\hy_2\hy_3$\\[5pt]
\end{tabular}
\begin{tabular}{ccc}
\scalebox{1}{\includegraphics{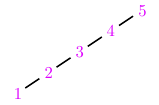}}
&\hspace*{20pt}
\scalebox{1}{\includegraphics{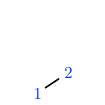}}
\hspace*{20pt}&
\scalebox{1}{\includegraphics{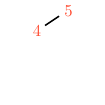}}\\
$P_\alpha$&$P_\blue$&$P_\red$
\end{tabular}
\caption{Second example of \cref{ideal decomp}.}
\label{orange2}
\end{figure}

\begin{figure}[p]
\begin{tabular}{cc}\\[5pt]
\scalebox{1}{\includegraphics{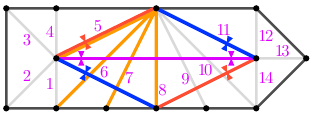}}
&
\scalebox{1}{\includegraphics{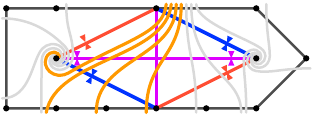}}\\[10pt]
Arc $8$ is $\gamma$
&
$\hy_\orange=\hy_5\hy_6\hy_7\hy_8$\\[5pt]
\end{tabular}
\begin{tabular}{ccc}
\scalebox{1}{\includegraphics{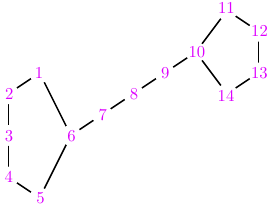}}
&\hspace*{20pt}
\scalebox{1}{\includegraphics{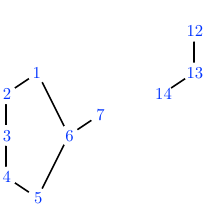}}
\hspace*{20pt}&
\scalebox{1}{\includegraphics{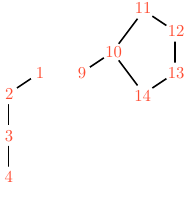}}\\
$P_\alpha$&$P_\blue$&$P_\red$
\end{tabular}
\caption{Third example of \cref{ideal decomp}.}
\label{orange3}
\end{figure}

\begin{figure}[p]
\begin{tabular}{cc}
\scalebox{1}{\includegraphics{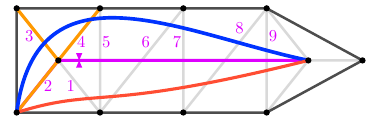}}
&
\scalebox{1}{\includegraphics{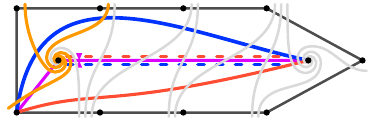}}\\
Arc $2$ is $\gamma$
&
$\hy_\orange=\hy_2\hy_3\hy_4$
\end{tabular}
\begin{tabular}{ccc}
\scalebox{1}{\includegraphics{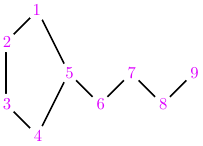}}
&\hspace*{20pt}
\scalebox{1}{\includegraphics{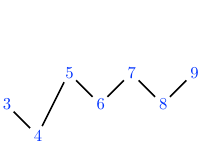}}
\hspace*{20pt}&
\scalebox{1}{\includegraphics{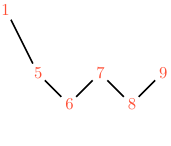}}\\
$P_\alpha$&$P_\blue$&$P_\red$
\end{tabular}
\caption{Fourth example of \cref{ideal decomp}.}
\label{orange4}
\end{figure}

\begin{figure}[p]
\begin{tabular}{cc}\\[10pt]
\scalebox{1}{\includegraphics{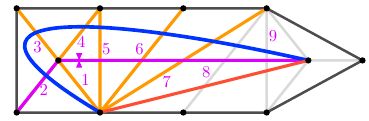}}
&
\scalebox{1}{\includegraphics{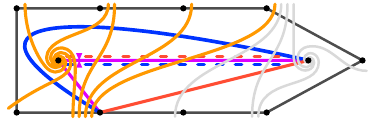}}\\
Arc $1$ is $\gamma$
&
$\hy_\orange=\hy_1\hy_2\hy_3\hy_4\hy_5\hy_6\hy_7$
\end{tabular}
\begin{tabular}{ccc}
\scalebox{1}{\includegraphics{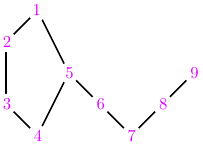}}
&\hspace*{20pt}
\scalebox{1}{\includegraphics{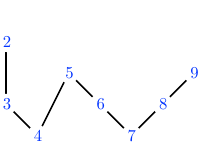}}
\hspace*{20pt}&
\scalebox{1}{\includegraphics{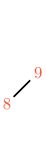}}\\
$P_\alpha$&$P_\blue$&$P_\red$
\end{tabular}
\caption{Fifth example of \cref{ideal decomp}.}
\label{orange5}
\end{figure}

\begin{figure}[p]
\begin{tabular}{cc}\\[10pt]
\scalebox{1}{\includegraphics{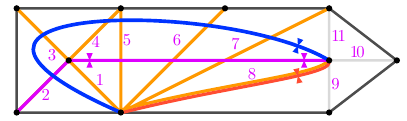}}
&
\scalebox{1}{\includegraphics{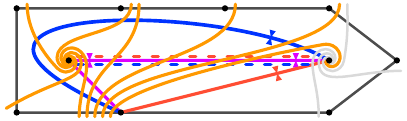}}\\
Arc $1$ is $\gamma$
&
$\hy_\orange=\hy_1\hy_2\hy_3\hy_4\hy_5\hy_6\hy_7\hy_8$
\end{tabular}
\begin{tabular}{ccc}
\scalebox{1}{\includegraphics{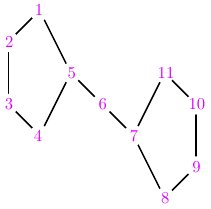}}
&\hspace*{20pt}
\scalebox{1}{\includegraphics{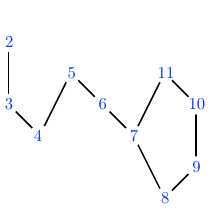}}
\hspace*{20pt}&
\scalebox{1}{\includegraphics{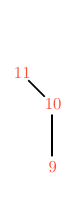}}\\
$P_\alpha$&$P_\blue$&$P_\red$
\end{tabular}
\caption{Sixth example of \cref{ideal decomp}.}
\label{orange6}
\end{figure}

\begin{figure}
\begin{tabular}{cc}
\scalebox{1}{\includegraphics{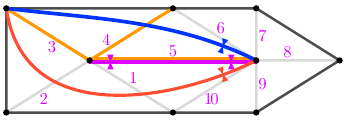}}
&
\scalebox{1}{\includegraphics{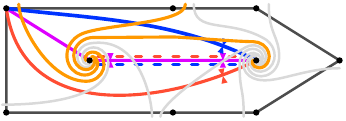}}\\
Arc $3$ is $\gamma$
&
$\hy_\orange=\hy_3\hy_4\hy_5$
\end{tabular}
\begin{tabular}{ccc}
\scalebox{1}{\includegraphics{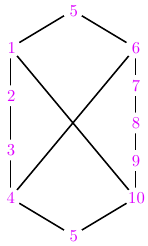}}
&\hspace*{20pt}
\scalebox{1}{\includegraphics{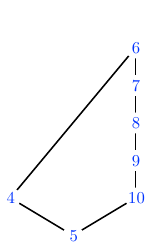}}
\hspace*{20pt}&
\scalebox{1}{\includegraphics{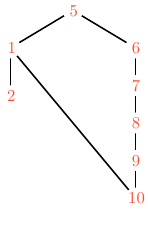}}\\
$P_\alpha$&$P_\blue$&$P_\red$
\end{tabular}
\caption{Seventh example of \cref{ideal decomp}.}
\label{orange7}
\end{figure}

\begin{figure}
\begin{tabular}{cc}\\[10pt]
\scalebox{1}{\includegraphics{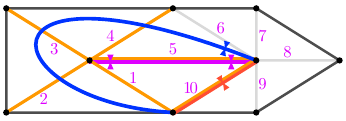}}
&
\scalebox{1}{\includegraphics{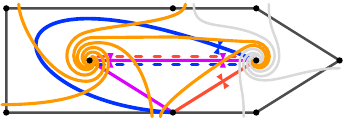}}\\
Arc $1$ is $\gamma$
&
$\hy_\orange=\hy_1\hy_2\hy_3\hy_4\hy_5\hy_{10}$
\end{tabular}
\begin{tabular}{ccc}
\scalebox{1}{\includegraphics{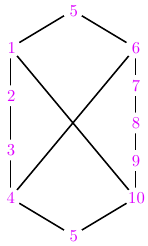}}
&\hspace*{20pt}
\scalebox{1}{\includegraphics{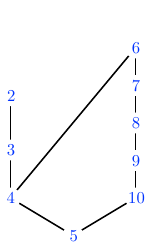}}
\hspace*{20pt}&
\scalebox{1}{\includegraphics{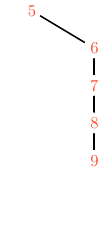}}\\
$P_\alpha$&$P_\blue$&$P_\red$
\end{tabular}
\caption{Eighth example of \cref{ideal decomp}.}
\label{orange8}
\end{figure}
\end{example}

\pagebreak

\begin{proof}[Proof of \cref{ideal decomp}]
We continue the notation $e_\delta$ for the element of $P_\alpha$ labeled $\hy_\delta$, for $\delta\in T$.
In the one case where there are two elements of $P_\alpha$ labeled $\hy_\delta$, namely the case where $\alpha$ coincides with $\delta$ except that $\alpha$ has both endpoints tagged notched, we take $e_\delta$ to mean the bottom element of $P_\alpha$, not the top element.

%

For the statement about $\hy_\orange$, we need to show that, for every arc $\delta\in T$, we have $e_\delta\le e_\gamma$ if and only if the lamination associated to $\delta$ (with spirals reversed where it shares an endpoint with $\alpha$ that is notched in $\alpha$) has a blue-purple-blue intersection, or equivalently, $\hy_\delta$ is a factor of~$\hy_\orange$.
It will also be clear that these factors appear with multiplicity~$1$.
(In the case where~$\alpha$ coincides with $\delta$ except that $\alpha$ has both endpoints notched, the top element of $P_\alpha$ labeled~$\hy_\delta$ is of course never below~$e_\gamma$.)
One of the purple arcs is $\alpha$, so if the lamination associated to $\delta\in T$ has a blue-purple-blue intersection, then in particular $\delta$ has some intersection with $\alpha$, so we will consider only arcs of $T$ that intersect~$\alpha$.
It is impossible for $\alpha$ to coincide with $\gamma$ except for tagging, because in that case~$\hy_\gamma$ occurs as a label in $P_\alpha$ either 0 or 2 times.
This contradicts the hypothesis that~$(T,\alpha)$ is tidy and $\gamma \in T$ is exchangeable with~$\alpha$.

For the statement about~$\Pr$, suppose $\rho$ is one of the red arcs.
We can choose an isotopy representative of $\rho$ such that every element $e_\delta$ of $P_\rho$ is also an element of $P_\alpha$.
Furthermore, two elements of $P_\rho$ form a cover relation in $P_\rho$ if and only if they form a cover relation in $P_\alpha$.
If $\rho$ and~$\sigma$ are the two red arcs, then $P_\rho$ and $P_\sigma$ are disjoint subsets of $P_\alpha$, and $P_\red$ is their disjoint union.
We will show that each $e_\delta$ appears in $P_\red$ if and only if $e_\delta\not\le e_\gamma$ in $P_\alpha$.
When we have showed that, since we obtain $P_\red$ from $P_\alpha$ by deleting a lower order ideal, and since the cover relations in $P_\red$ are the same as in $P_\alpha$, we can conclude that the partial order on $P_\red$ is the restriction of the partial order on $P_\alpha$, as desired.

There are two main cases to consider:  The arc $\gamma$ may cross $\alpha$ or it may share an endpoint with~$\alpha$ where~$\alpha$ is tagged notched.  
In either case, if $\delta\in T$ does not intersect $\alpha$ or if $\delta$ shares an endpoint of $\alpha$ where $\alpha$ is tagged plain, then neither $P_\alpha$ nor $P_\red$ contain an element $e_\delta$, and the associated lamination can only make red-purple-blue intersections.
See Figures~\ref{orange1}, \ref{orange2}, \ref{orange4}, and~\ref{orange5}.
Thus we will not consider such arcs $\delta$ below.

In both cases, and in every case for $\delta$, we give a simple condition on $\delta$ that is easily equivalent to $e_\delta\le e_\gamma$, to the lamination associated to $\delta$ having a blue-purple-blue intersection, and to $e_\delta$ not being an element of $P_\red$.
To avoid unnecessary repetition of words, in every case, we simply state below what that condition is.

\begin{case}
$\gamma$ crosses $\alpha$.
In this case, $\alpha$ does not coincide with an arc of $T$.

\begin{subcase}
$\delta$ crosses $\alpha$.
See Figures~\ref{orange1}, \ref{orange2}, and~\ref{orange3}.
As we move from $e_\gamma$ along $\alpha$ to $e_\delta$, each arc that we cross has a left and right endpoint.
The condition is that $\gamma$ and $\delta$ share their right endpoints and each arc of $T$ crossed in between $\gamma$ and $\delta$ also shares that right endpoint.
\end{subcase}

\begin{subcase}
$\delta$ shares an endpoint $p$ with $\alpha$ such that $\alpha$ is tagged notched at~$p$. 
See \cref{orange3}.
Let $t$ be the triangle in $T$ that has $p$ as a vertex and contains part of $\alpha$ in its interior.
(This~$t$ is bounded by the last arc of $T$ crossed by $\alpha$ and by two arcs of $T$ that share the endpoint $p$ with $\alpha$.)
The condition is that all three of the following statements hold:
every arc crossed by~$\alpha$ between~$e_\gamma$ and $p$ shares its right endpoint with $\gamma$; 
the arc $\delta$ is an edge of $t$; 
and $\delta$ shares its right endpoint with $\gamma$.
(The notion of ``right endpoint'' for $\delta$ is from the point of view of someone moving along~$\alpha$ from $e_\gamma$ to $p$.
Thus $p$ is the \emph{left} endpoint of $\delta$.)
Equivalently, there is a much simpler condition:~$\delta$ coincides with the red arc whose endpoint is $p$, except that that red arc is notched at $p$ where~$\delta$ is plain.
(To consider blue-purple-blue intersections in this case, the reversal of spirals at the notched endpoint is crucial.)
\end{subcase}
\end{case}

\begin{case}
$\gamma$ shares an endpoint $p$ with $\alpha$ where $\alpha$ is tagged notched.
This includes the cases where~$\alpha$ coincides with $\delta$ except for tagging, and in those cases, no arcs in $T$ cross $\alpha$.

\begin{subcase}
$\delta$ shares the endpoint $p$.
See Figures~\ref{orange4}, \ref{orange5}, \ref{orange6}, \ref{orange7}, and~\ref{orange8}.
The condition is that one can move from~$\gamma$ to $\delta$ clockwise around $p$ without crossing $\alpha$.

We clarify the condition in cases where $\delta$ coincides with $\alpha$ except for tagging:
We have already supposed that $\alpha$ is tagged notched at $p$.
If $\alpha$ coincides with $\delta$ and is tagged plain at its other endpoint, then there is no element $e_\delta$ in $P_\alpha$ or $P_\red$.
If $\alpha$ is tagged notched at both ends and otherwise coincides with $\delta$, recall that the notation $e_\delta$ means the bottom element of $P_\alpha$, not the top element.
(In any case, the top element is not below $e_\gamma$.)
In this case, we consider the condition to be fulfilled (because, moving from~$\gamma$ to $\delta$ clockwise around $p$, we arrive at $\alpha$ but do not cross it.
\end{subcase}

\begin{subcase}
$\delta$ crosses $\alpha$.  
See Figures~\ref{orange4}, \ref{orange5}, and~\ref{orange6}.
Again, arcs crossed by $\alpha$ have a left endpoint and a right endpoint from the point of view of someone moving along~$\alpha$ from $e_\delta$ to~$p$.
Let~$t$ be the triangle in $T$ that has $p$ as a vertex and contains part of $\alpha$ in its interior.
The condition is that the following three statements hold:
the arc $\gamma$ is an edge of $t$; 
every arc crossed by $\alpha$ between $e_\delta$ and $p$ shares its left endpoint with $\gamma$; 
and $\delta$ shares its left endpoint with~$\gamma$.
(The ``left endpoint'' of $\gamma$ is the other endpoint, not $p$.)
\end{subcase}

\begin{subcase}
$\alpha$ is also tagged notched at its other endpoint $q$ and $\delta$ shares the endpoint $q$ with~$\alpha$.
See Figures~\ref{orange5}, \ref{orange6}, \ref{orange7}, and~\ref{orange8}.
Let $u$ be the triangle of $T$ having $q$ as a vertex and containing part of $\alpha$ in its interior.
We use the terms left and right from the point of view of someone moving along $\alpha$ from $q$ to $p$.
The condition is that the following four statements hold:
the arc $\gamma$ is an edge of $t$;
the arc $\delta$ is an edge of $u$; 
every arc of $T$ crossed by $\alpha$ (if there are any) shares its left endpoint with~$\gamma$; 
and $\delta$ shares its left endpoint with $\gamma$.
(The left endpoint of $\delta$ is the other endpoint, not~$q$.)
\end{subcase}
\end{case}

This concludes our proof of the statement about~$\Pr$ and the statement about~$\yo$.
Reversing the orientation of $\S$ exchanges the red arcs with the blue arcs and replaces the partial order $P_\alpha$ with the dual partial order.
The statement about $\Pb$ thus follows from the statement about~$\Pr$ and this orientation-reversal symmetry.
Finally, decomposing the order ideals of~$P_\alpha$ depending on whether or not they contain~$e_\gamma$ yields the poset exchange relation of \cref{ideal decomp eq}.
\end{proof}

%


\subsection{$\g$-vectors in tidy exchange relations}\label{subsec:shear}

We next describe the relation among the $\g$-vectors appearing in the exchange.
We again assume that~$(T,\alpha)$ is tidy and consider~$\gamma \in T$ exchangeable with~$\alpha$.
Write $\g_\blue$ for the product of monomials $\g_\beta$ associated to the blue arcs, and similarly~$\g_\red$.
The notation of specialization to ``$y=1$'' in the following proposition is shorthand for setting every tropical variable to $1$.

\begin{proposition}\label{g prop}
If~$(T,\alpha)$ is tidy and $\gamma \in T$ is exchangeable with~$\alpha$, then 
\begin{enumerate}[\quad\rm\bf1.]
\item\label{g prop blue}
$\g_\blue=x_\gamma\cdot\g_\alpha$ and 
\item\label{g prop red}
$\g_\red=x_\gamma\cdot\g_\alpha\cdot\hy_\orange\big|_{y=1}$.
\end{enumerate}
\end{proposition}

The proof of \cref{g prop} has a similar feel as that of \cref{ideal decomp}.
We separate the proofs of the two items.

\begin{proof}[Proof of \cref{g prop}.\ref{g prop blue}]
We show that $\g_\blue=x_\gamma\cdot\g_\alpha$, in two cases, depending on whether~$\gamma$ crosses $\alpha$ or shares an endpoint of $\alpha$ where $\alpha$ is tagged notched.
In either case, let $Q$ be the quadrilateral consisting of $\gamma$ and the two triangles of $T$ having $\gamma$ as an edge.

\setcounter{case}{0}
\begin{case}
\label{g prop blue case 1}
$\gamma$ crosses $\alpha$.
Write $\beta$ and $\beta'$ for the two blue arcs.
The curve $\kappa(\beta)$ follows $\kappa(\alpha)$ from near one endpoint of $\alpha$ through the same intersections with triangles of $T$ until they both enter~$Q$ through the same side of $Q$ and then both cross $\gamma$ before $\kappa(\beta)$ ends near an endpoint of $\gamma$ and~$\kappa(\alpha)$ continues.
The curve $\kappa(\beta')$  has the same relationship with $\kappa(\alpha)$, but starting from the other endpoints.
We see that the shear coordinates of $\kappa(\beta)$ and $\kappa(\beta')$ add up to a vector that agrees with the shear coordinates of $\kappa(\alpha)$ except at the $\gamma$-entry.
It remains only to check, in all four subcases of how $\kappa(\alpha)$ can cross $Q$, that the $\gamma$-entry of the sum of the shear coordinates of $\kappa(\beta)$ and $\kappa(\beta')$ is one \emph{less} than the $\gamma$-entry of the shear coordinates of $\kappa(\alpha)$.
This simple check is shown in \cref{gamma quad}.
In the figure, $\kappa(\alpha)$ and the contribution to its shear coordinates at $\gamma$ are shown in purple, while $\kappa(\beta)$ and $\kappa(\beta')$ and the contributions to their shear coordinates at $\gamma$ are shown in blue.
\begin{figure}
\includegraphics{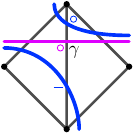}
\qquad
\includegraphics{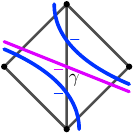}
\qquad
\includegraphics{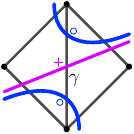}
\qquad
\includegraphics{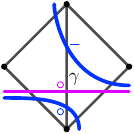}
\caption{The $\gamma$-entry of $\g_\alpha$ and $\g_\blue$, when $\gamma$ crosses $\alpha$, \cref{g prop blue case 1} in the proof of \cref{g prop}.\ref{g prop blue}.}
\label{gamma quad}
\end{figure}
\end{case}

\begin{case}
\label{g prop blue case 2}
$\gamma$ shares an endpoint of $\alpha$ where $\alpha$ is tagged notched.
Let $\beta$ be the one blue arc.
Then~$\kappa(\beta)$ follows $\kappa(\alpha)$ from near the other endpoint (not the shared endpoint) of $\alpha$ through the same intersections with triangles of $T$ until they both enter $Q$ through the same side of $Q$.
At that point, $\kappa(\beta)$ turns right and leaves $Q$ just to the left of the other endpoint (not the shared endpoint) of $\gamma$ and either hits a boundary segment or spirals clockwise into that endpoint, while~$\kappa(\alpha)$ turns left and spirals counterclockwise into the shared endpoint.
Thus the shear coordinates of $\kappa(\beta)$ agree with the shear coordinates of $\kappa(\alpha)$ except at the $\gamma$-entry.
We check that the $\gamma$-entry of the shear coordinates of $\kappa(\beta)$ is one \emph{less} than the $\gamma$-entry of the shear coordinates of $\kappa(\alpha)$.
The check is again simple and is shown in \cref{gamma digon}, where there are two subcases corresponding to the two possible sides through which $\kappa(\beta)$ and $\kappa(\alpha)$ might enter $Q$.
\qedhere
\begin{figure}
\includegraphics{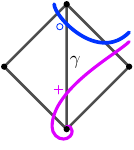}
\qquad
\includegraphics{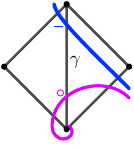}
\caption{The $\gamma$-entry of $\g_\alpha$ and $\g_\blue$, when $\gamma$ shares an endpoint with  $\alpha$, notched in $\alpha$, \cref{g prop blue case 2} in the proof of \cref{g prop}.\ref{g prop blue}.}
\label{gamma digon}
\end{figure}
\end{case}
\end{proof}

\begin{proof}[Proof of \cref{g prop}.\ref{g prop red}]
To show that $\g_\red=x_\gamma\cdot\g_\alpha\cdot\hy_\orange\big|_{y=1}$, it is convenient to rewrite the identity as $\g_\red/\g_\alpha=x_\gamma\cdot\hy_\orange\big|_{y=1}$.

Given an orange arc (an arc $\delta$ such that $e_\delta\le e_\gamma$ in $P_\alpha$), there are two triangles of $T$ containing~$\delta$ as an edge. 
Since $\hy_\delta=y_\delta\prod_{\beta\in T}x_\beta^{b_{\beta\delta}}$, the Laurent monomial $\hy_\delta\big|_{y=1}$ is the product $x_\beta x_{\beta'}$ for the two edges $\beta$ and $\beta'$ (one in each triangle) that are counterclockwise of $\delta$, divided by the similar product of the two edges that are clockwise, except that if any of these edges are boundary segments, then we leave them out.
Thus $\hy_\orange\big|_{y=1}$ is the product of these contributions, over all orange edges.

Because the exponent vector of $\g_\alpha$ is the negative of the shear coordinates of $\kappa(\alpha)$ and $\g_\red$ is the negative of the (sum of the) shear coordinates of $\kappa$ applied to the red arc(s), we can rephrase the assertion that $\g_\red/\g_\alpha=x_\gamma\cdot\hy_\orange\big|_{y=1}$ in terms of shear coordinates.
Specifically, taking the shear coordinates of $\kappa(\alpha)$ minus the (sum of the) shear coordinate(s) of $\kappa$ applied to the red arc(s), the entry at each arc must equal the corresponding exponent in $\hy_\orange\big|_{y=1}$, plus $1$ in the entry for $\gamma$.

We will verify that this equality between the difference of shear coordinates and the exponents on $x_\gamma\cdot\hy_\orange\big|_{y=1}$.
We again argue in two cases, depending on whether $\gamma$ crosses $\alpha$ or shares an endpoint of $\alpha$ where $\alpha$ is tagged notched.

\setcounter{case}{0}
\begin{case}
\label{g prop red case 1}
$\gamma$ crosses $\alpha$.
From the proof 
of \cref{ideal decomp} (Case 1), we see that every orange arc~$\delta$ shares an endpoint with $\gamma$.
Indeed, for each endpoint $p$ of $\gamma$, there is a sequence (possibly a singleton) containing $\gamma$ and all arcs of $T$ that are clockwise of $\gamma$ at $p$ up to a particular arc (not skipping any arcs).
The sequence does not extend beyond the red arc, but its last entry might be the red arc (in which case, the red arc and $\alpha$ are tagged notched at the other end of the red arc, opposite from $p$).
If $\delta$ is this last orange arc at $p$, consider the triangle $t$ that has $\delta$ as an edge and is clockwise of $\delta$ at $p$.
Let $\eta$ be the edge of $t$ that shares the endpoint $p$ with $\delta$ and let $\zeta$ be the other edge.
Similarly, let $\eta'$ and $\zeta'$ be the analogous arcs associated with the other endpoint of $\gamma$.

There are three subcases, depending on whether there is a non-singleton sequence of arcs at neither end of $\gamma$, at one end, or at both.
\cref{hy quad} shows, in these three cases, all of the contributions, positive and negative, to the exponents of cluster variables in $x_\gamma\cdot\hy_\orange\big|_{y=1}$.
These contributions are on the orange arcs and the arcs $\eta$, $\zeta$, $\eta'$, and $\zeta'$.
\begin{figure}
\scalebox{0.949}{\includegraphics{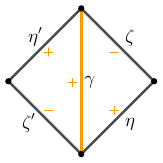}}
\,\,
\scalebox{0.949}{\includegraphics{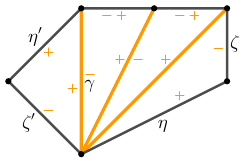}}
\,\,
\scalebox{0.949}{\includegraphics{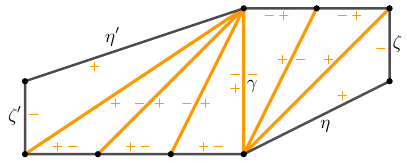}}
\caption{Exponents in $x_\gamma\cdot\hy_\orange\big|_{y=1}$, \cref{g prop red case 1} in the proof of \cref{g prop}.\ref{g prop red}.}
\label{hy quad}
\end{figure}

The exponent of $x_\eta$ in $\hy_\orange\big|_{y=1}$ is $1$ and the exponent of $x_\zeta$ is $-1$.
Since $\delta$ is the last orange arc, $\kappa(\alpha)$ does not cross $\eta$, but rather enters the triangle $t$ through $\zeta$.
(If $\delta$ coincides with the red arc, we are using the fact that the other end of the red arc is notched.)
There are two possibilities for where $\kappa(\alpha)$ comes from before crossing $\zeta$.
\cref{zeta eta} shows that the shear coordinates of $\kappa(\alpha)$ minus the sum of the shear coordinates of $\kappa$ applied to the red arcs has the correct entry at $\zeta$ and~$\eta$.
\begin{figure}
\scalebox{0.949}{\includegraphics{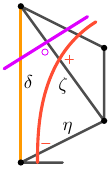}}
\qquad\qquad
\scalebox{0.949}{\includegraphics{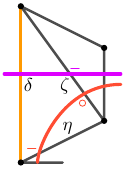}}
\caption{Shear coordinates at $\zeta$ and $\eta$, \cref{g prop red case 1} in the proof of \cref{g prop}.\ref{g prop red}.}
\label{zeta eta}
\end{figure}
(It is possible that $\eta$ is a boundary segment, in which case the contribution to $\eta$ disappears both in the shear coordinates and in $\hy_\orange\big|_{y=1}$.)
The same is true replacing $\zeta$ and $\eta$ with $\zeta'$ and~$\eta'$.
Since $\kappa(\alpha)$ coincides with $\kappa$ of one of the red arcs until they both cross $\zeta$, they have the same shear coordinates for every arc of $T$ crossed by $\kappa(\alpha)$ before it crosses $\zeta$.

It remains only to show that the shear coordinates of $\kappa(\alpha)$ minus the sum of the shear coordinates of $\kappa$ applied to the red arcs has the correct entries at the orange arcs.
Indeed, $\kappa$ applied to the red arcs has is disjoint from the orange arcs, so the point is the shear coordinates of $\kappa(\alpha)$.
This shear coordinate computation is shown in \cref{alpha quad}, in the same three cases as \cref{hy quad}, and with the same result.
\begin{figure}
\scalebox{0.949}{\includegraphics{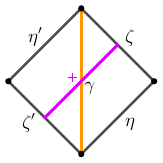}}
\,\,
\scalebox{0.949}{\includegraphics{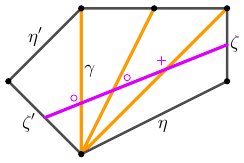}}
\,\,
\scalebox{0.949}{\includegraphics{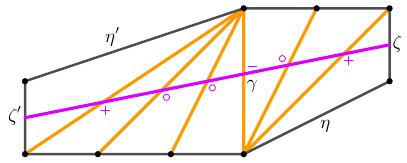}}
\caption{Shear coordinates of $\alpha$ on the orange arcs, \cref{g prop red case 1} in the proof of \cref{g prop}.\ref{g prop red}.}
\label{alpha quad}
\end{figure}
\end{case}

\begin{case}
\label{g prop red case 2}
$\gamma$ shares an endpoint of $\alpha$ where $\alpha$ is tagged notched.
Again, the proof 
of \cref{ideal decomp} (this time Case 2), indicates that every orange arc $\delta$ shares an endpoint with $\gamma$.
Let $p$ be the endpoint common endpoint of $\gamma$ and $\alpha$ (where $\alpha$ is notched) and let $q$ be the other endpoint of $\gamma$ (not of $\alpha$).
There are five subcases.

\begin{subcase}
\label{g prop red subcase 2.1}
$\gamma$ is the only orange arc.
This case is illustrated in \cref{hydigon1}.
\begin{figure}
\begin{tabular}{ccc}
\scalebox{0.949}{\includegraphics{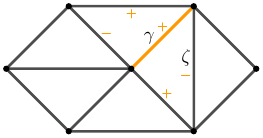}}
&
\scalebox{0.949}{\includegraphics{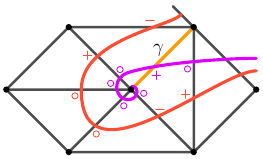}}
&
\scalebox{0.949}{\includegraphics{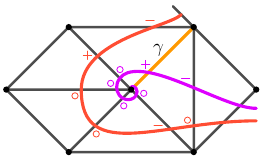}}
\end{tabular}
\caption{Exponents in $x_\gamma\cdot\hy_\orange\big|_{y=1}$ and shear coordinates of $\alpha$ and the red arc, \cref{g prop red subcase 2.1} in the proof of \cref{g prop}.\ref{g prop red}.}
\label{hydigon1}
\end{figure}
The left picture of the figure shows all contributions to the exponent vector of $x_\gamma\cdot\hy_\orange\big|_{y=1}$.
The curve $\kappa(\alpha)$ crosses the curve labeled $\zeta$ after leaving $p$ and there are two possibilities for where $\kappa(\alpha)$ goes after crossing~$\zeta$.
The middle and right picture verify, for these two possibilities, that the shear coordinates of $\kappa(\alpha)$ minus the shear coordinates of $\kappa$ of the red arc have entries that agree with the exponent vector of $x_\gamma\cdot\hy_\orange\big|_{y=1}$.
The configuration shown in the right picture includes the possibility that $\alpha$ coincides with an arc of $T$ except for being notched at one side.
\end{subcase}

\begin{subcase}
\label{g prop red subcase 2.2}
All orange arcs share the endpoint $p$ and there are at least $2$ non-orange arcs at~$p$.
This case is illustrated and verified in \cref{hydigon2}, where the pictures are analogous to those in \cref{hydigon1}.
\begin{figure}
\begin{tabular}{ccc}
\scalebox{0.949}{\includegraphics{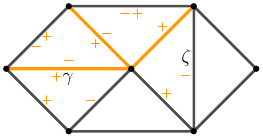}}
&
\scalebox{0.949}{\includegraphics{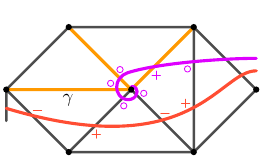}}
&
\scalebox{0.949}{\includegraphics{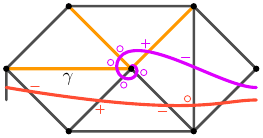}}
\end{tabular}
\caption{Exponents in $x_\gamma\cdot\hy_\orange\big|_{y=1}$ and shear coordinates of $\alpha$ and the red arc, \cref{g prop red subcase 2.2} in the proof of \cref{g prop}.\ref{g prop red}.}
\label{hydigon2}
\end{figure}
In this case, the configuration shown in the middle picture includes the possibility that~$\alpha$ coincides with an arc of $T$ except for being notched at both sides and the right picture again includes the possibility that $\alpha$ coincides with an arc of $T$ except for being notched at one side.
\end{subcase}

\begin{subcase}
\label{g prop red subcase 2.3}
All orange arcs share the endpoint $p$ and there is exactly $1$ non-orange arc at $p$.
This case is illustrated and verified in \cref{hydigon3}, with pictures analogous to previous cases.
\begin{figure}
\begin{tabular}{ccc}
\scalebox{0.949}{\includegraphics{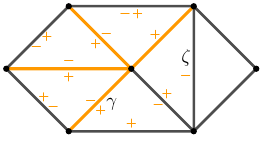}}
&
\scalebox{0.949}{\includegraphics{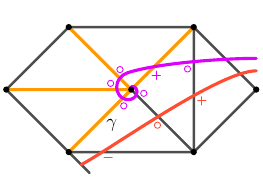}}
&
\scalebox{0.949}{\includegraphics{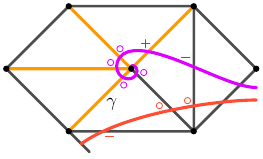}}
\end{tabular}
\caption{Exponents in $x_\gamma\cdot\hy_\orange\big|_{y=1}$ and shear coordinates of $\alpha$ and the red arc, \cref{g prop red subcase 2.3} in the proof of \cref{g prop}.\ref{g prop red}.}
\label{hydigon3}
\end{figure}
Again, the middle picture includes the possibility that $\alpha$ coincides with an arc of $T$ except for being notched at both sides and the right picture includes the possibility that $\alpha$ coincides with an arc of $T$ except for being notched at one side.
\end{subcase}

\begin{subcase}
\label{g prop red subcase 2.4}
All arcs at $p$ are orange and there is exactly one orange arc at $q$ aside from $\gamma$. 
This case is \cref{hydigon4}.
\begin{figure}
\begin{tabular}{ccc}
\scalebox{0.949}{\includegraphics{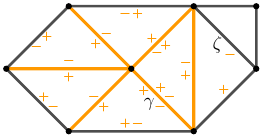}}
&
\scalebox{0.949}{\includegraphics{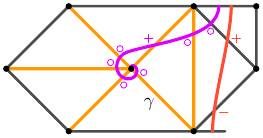}}
&
\scalebox{0.949}{\includegraphics{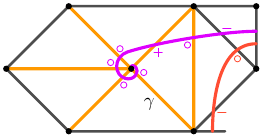}}
\end{tabular}
\caption{Exponents in $x_\gamma\cdot\hy_\orange\big|_{y=1}$ and shear coordinates of $\alpha$ and the red arc, \cref{g prop red subcase 2.4} in the proof of \cref{g prop}.\ref{g prop red}.}
\label{hydigon4}
\end{figure}
The arc $\gamma$ and the other orange arc with an endpoint $q$ bound a triangle whose third side is also orange.
The middle picture includes the possibility that $\alpha$ coincides with an arc of $T$ except for being notched at both sides.
\end{subcase}

\begin{subcase}
\label{g prop red subcase 2.5}
All arcs at $p$ are orange and there is more than one orange arc at $q$ aside from $\gamma$.
This case is \cref{hydigon5}, which is like the previous figures, except that the top picture shows the exponent vector of $x_\gamma\cdot\hy_\orange\big|_{y=1}$ and the bottom two pictures show the shear coordinates of $\kappa(\alpha)$ and of $\kappa$ applied to the red arc. \qedhere
\begin{figure}
\begin{tabular}{c}
\scalebox{0.949}{\includegraphics{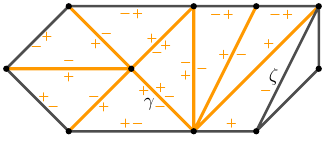}}
\end{tabular}
\begin{tabular}{ccc}
\scalebox{0.949}{\includegraphics{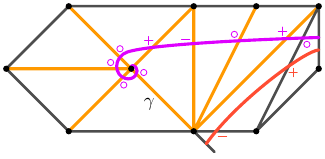}}
&&
\scalebox{0.949}{\includegraphics{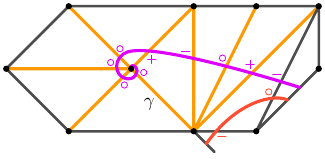}}
\end{tabular}
\caption{Exponents in $x_\gamma\cdot\hy_\orange\big|_{y=1}$ and shear coordinates of $\alpha$ and the red arc, \cref{g prop red subcase 2.5} in the proof of \cref{g prop}.\ref{g prop red}.}
\label{hydigon5}
\end{figure}
\end{subcase}
\end{case}
\end{proof}

\section{Hyperbolic geometry and the coefficient-free case}\label{sec:hyperbolic}
In this section, we prove the analog of \cref{main} in the coefficient-free case.
This is a specialization of \cref{main} by setting all of the tropical variables $y_\gamma$ to $1$.
The coefficient-free case is not used in the proof of Theorems~\ref{main} in \cref{sec:tropical}, but rather serves to illustrate the remaining key idea of the proof using the simpler of the two hyperbolic-geometric settings of~\cite{FominThurston}.
We continue the notation $(\S,\M)$ for a marked surface and $T$ for a fixed tagged triangulation, again assumed to have all arcs tagged plain, except possibly at some punctures $p$ incident to exactly two tagged arcs of $T$, identical except for opposite taggings at $p$.
We also continue the notation $T^\circ$ for the ordinary triangulation corresponding to $T$.

The coefficient-free case of the theorem associates the same poset $P_\alpha$ to a tagged arc $\alpha$, but with a different weighting.
For $e \in P_\alpha$, we define the weight $\tilde w(e)$ to be the monomial obtained from $w(e)$ by setting all tropical variables $y_\gamma$ to $1$.
Thus when $w(e)$ is $\hy_\gamma$ for some $\gamma\in T$, then~$\tilde w(e)=\prod_{\delta\in T}x_\delta^{b_{\delta\gamma}}$.
When $w(e)=\frac{\hy_\beta}{\hy_\gamma}$, then $\tilde w(e)=1$ because in this case $b_{\delta\beta}=b_{\delta\gamma}$ for all~$\delta\in T$.

\begin{theorem}\label{main lite}
The coefficient-free cluster variable associated to a tagged arc $\alpha$ is
\begin{equation}
\label{main lite eq}
x_\alpha=\g_\alpha\cdot\sum_I \prod_{e \in I}\tilde w(e),
\end{equation}
where the sum is over all order ideals $I$ in $P_\alpha$.
\end{theorem}

The entire discussion in \cref{red sec} holds in the coefficient-free case verbatim (but omitting the tropical variables $y_\gamma$ or setting them all to $1$), so in particular we have the following version of \cref{reduce}.

\begin{proposition}\label{reduce lite}
Assume that \cref{main lite eq} holds for every tagged arc $\alpha$ with the following property:
When $\alpha$ has an endpoint that is the interior vertex of a self-folded triangle in $T^\circ$, the tagging of $\alpha$ is plain at that endpoint.
Then \cref{main lite eq} holds for every tagged arc.
\end{proposition}

\subsection{Lambda lengths}\label{lam sec}
In preparation for the proof of \cref{main lite}, we explain how cluster variables can be realized in terms of certain hyperbolic lengths.
(Proofs of statements left unjustified here are found in \cite[Chapters~7--8]{FominThurston}.)
More precisely, the initial cluster variables can be taken as indeterminates (constrained to be positive) whose values determine certain hyperbolic lengths in the surface and thus determine a hyperbolic metric and a choice of certain curves called horocycles.
Although it is the initial cluster variables that determine the metric, it is more pleasant to explain the construction in the reverse order, starting from a metric, choosing horocycles, and measuring certain lengths that determine the initial cluster variables.

Given $(\S,\M)$ and $T$, choose a finite-area hyperbolic metric with constant curvature $-1$ on $\S\setminus\M$ with a cusp at each point in $\M$, with the property that each boundary segment is a geodesic.
Since the arcs in $T^\circ$ are defined up to isotopy, we can also assume that they are geodesics.
At each marked point $p$, choose a \newword{horocycle} $h_p$, meaning a curve that is orthogonal to every geodesic that limits to~$p$.
This hyperbolic metric together with the choice of horocycles is a \newword{decorated hyperbolic structure} on $(\S,\M)$.
Given a geodesic $\gamma$ that limits to marked points $p$ and $q$, the distance along the geodesic between $h_p$ and $h_q$ is finite (in contrast to the distance along the geodesic between $p$ and $q$).
The \newword{signed length} $l(\gamma)$ is this finite distance, counted as positive if the segment of~$\gamma$ between $h_p$ and $h_q$ is outside the disks they define, zero if $h_p$, $h_q$, and $\gamma$ all intersect at the same point, and otherwise negative if the segment of~$\gamma$ between $h_p$ and $h_q$ is inside the disks they define.
The \newword{lambda-length} of $\gamma$ is defined to be $\lambda(\gamma)=\exp(l(\gamma)/2)$.

The first key result about lambda lengths \cite[Theorem~7.4]{FominThurston} is that an arbitrary assignment of a positive number to each arc in $T^\circ$ and to each boundary segment of $(\S,\M)$ uniquely determines a decorated hyperbolic structure such that the arcs and boundary segments are geodesics and the assigned numbers are the lambda lengths.
(In the language of \cite[Theorem~7.4]{FominThurston}, there is a bijection from the set of such assignments to the \newword{decorated Teichm\"{u}ller space}, the space of decorated hyperbolic structures.
Indeed, \cite[Theorem~7.4]{FominThurston} says that the map is a homeomorphism.)

When $p$ is a puncture, the horocycle $h_p$ has a finite positive length $L(h_p)$ in the hyperbolic metric, and $L(h_p)$ specifies $h_p$ uniquely among horocycles at $p$.
The \newword{conjugate horocycle} $\bar h_p$ is the unique horocycle at $p$ whose length satisfies $L(\bar h_p)L(h_p)=1$.
The \newword{lambda length} of a \emph{tagged} arc $\gamma$ is defined just as in the untagged case, except that at each endpoint $p$ where $\gamma$ is tagged notched, we use the conjugate horocycle $\bar h_p$ in place of $h_p$.

The lambda lengths of the tagged arcs in $T$ determine the lambda lengths of the ordinary arcs in $T^\circ$ and vice versa.
Specifically, suppose $\beta$ and $\gamma$ are tagged arcs in $T$ that coincide except that their taggings are both plain at one endpoint $q$ but disagree at the other endpoint $p$, and suppose that $\delta$ is the ordinary arc in $T^\circ$ with both endpoints at $q$ and follows $\beta$ and $\gamma$ closely to form a loop around $p$.
Then 
\begin{equation}\label{loop eq lite}
\lambda(\delta)=\lambda(\beta)\lambda(\gamma).
\end{equation}
(This is \cite[Lemma~8.2]{FominThurston}.)
Thus in particular, we can arbitrarily choose positive lambda lengths~$x_\gamma$ for the tagged arcs in $\gamma\in T$ and for the boundary segments to uniquely determines a hyperbolic metric with constant curvature $-1$ and a choice of horocycles.
The chosen lambda lengths can be thought of as indeterminates which then determine the lambda lengths of all other tagged arcs.
Taking these chosen lambda lengths to be the initial cluster variables for an exchange pattern with initial exchange matrix $B(T)$, the cluster variables in the exchange pattern are indexed by tagged arcs in $(\S,\M)$.
Each tagged arc has a geodesic representative whose lambda length is the associated cluster variable.
This is part of \cite[Theorem~8.6]{FominThurston}.
That theorem does not take coefficient-free cluster algebras but rather boundary coefficients, but we can recover the coefficient-free case by setting the lambda length of each boundary component to be~$1$.

We can now interpret \cref{main lite} as the assertion that a given tagged arc $\alpha$ in $(\S,\M)$ has lambda length $\g_\alpha\cdot\sum_I \prod_{e\in I}\tilde w(e)$, where as always the sum is over all order ideals $I$ in $P_\alpha$.

\subsection{The tile cover of a tagged arc}\label{tile cov sec}
The key to \cref{main lite} is to construct a tagged arc $\alpha'$ in a new surface $(\S',\M')$, equipped with a triangulation $T'$ and a decorated hyperbolic structure such that the lambda length $\lambda(\alpha')$ is the same as $\lambda(\alpha)$, but with simpler combinatorics (that is, so that $(T',\alpha')$ is tidy in the sense of \cref{subsec:tidy}).
This will allow us to use \cref{ideal decomp} in an inductive proof of the theorem.

A given triangle $t$ of $T^\circ$ may not have three distinct edges and three distinct vertices.
Define a \newword{tile} of $t$ to be an ideal triangle with three distinct edges and three distinct vertices and with a hyperbolic metric such that the interior of the tile is isometric to the interior of $t$.
The metric on the interior of the tile uniquely determines a metric on the edges, and since we have assumed that the arcs in $T^\circ$ are geodesic, the edges of the tile are also geodesic.
The tile again has cusps at its vertices.
The tile is naturally equipped with horocycles at each of its vertices, namely the images of the horocycles on vertices of $t$.
If two triangles $t$ and $u$ of $T^\circ$ share an edge, then we can identify a tile of $t$ and a tile of $u$ along that edge in such a way that their decorated hyperbolic structures agree on the edge and define a decorated hyperbolic structure on their union.

Given a tagged arc $\alpha$, the \newword{tile cover} of $\alpha$ is a tagged arc $\alpha'$ in a surface $(\S',\M')$ with a decorated hyperbolic structure and a triangulation~$T'$, constructed as follows and illustrated in \cref{tile ex} and Figures~\ref{tilex1}--\ref{tilex7}.
(Strictly speaking, we should say ``a'' tile cover, since there are choices to be made at the end of the construction.  
We say ``the'' tile cover because the relevant parts of the tile cover are completely determined by $\alpha$ and $T$.)
The crucial property of the tile cover is that~$(T', \alpha')$ is tidy in the sense of \cref{sec:special}.
We only construct the tile cover when $\alpha$ has the following property:  
If $\alpha$ has an endpoint $p$ that is the interior vertex of a self-folded triangle in $T^\circ$, then $\alpha$ is tagged plain there.
(We will only need this case because of \cref{reduce lite} in the coefficient-free case and \cref{reduce} in the principal-coefficients case.)

First, assume $\alpha$ does not coincide, except for tagging, with an arc of $T^\circ$.
Then $\alpha$ can be decomposed into a finite collection of segments, each of which has its endpoints on arcs of $T^\circ$ (possibly at marked points) and otherwise not intersecting the arcs of $T^\circ$.
Thus each segment is contained in some triangle of $T^\circ$.
For each of these segments construct a tile from the triangle containing it.
Since $\alpha$ can intersect the same triangle multiple times, we might construct multiple tiles for a given triangle.
Furthermore, at each endpoint of $\alpha$ that is not the interior point of a self-folded triangle in $T^\circ$, construct a tile for each triangle-corner of a triangle at that endpoint.  
A triangle of $T^\circ$ can have $0$, $1$, $2$, or $3$ corners at that endpoint, so again some triangles may become multiple tiles.
At each endpoint of~$\alpha$ that is the interior point of a self-folded triangle in~$T^\circ$, construct two copies of the tile for the self-folded triangle.
Then identify all of these triangles in the natural way:
For each endpoint of~$\alpha$, identify the tiles on their edges around the endpoint in the order that the triangle corners appear around that endpoint.
(When an endpoint $p$ of $\alpha$ is the interior point of a self-folded triangle in $T^\circ$, arrange the two tiles for that triangle about $p$ to form a digon whose two edges correspond to the exterior edge of the self-folded triangle.
If an endpoint of $\alpha$ is the puncture in a once-punctured digon formed by two triangles in $T^\circ$, identify the tiles for the two triangles so that the same is true of the corresponding endpoint of $\alpha$ in $T'$.)
Then continue identifying edges to join the segments of $\alpha$ in order and lift $\alpha$ to a tagged arc $\alpha'$ (the union of the lifts of the segments) in the union of the tiles.
The union of the tiles is topologically a disk.

If $\alpha$ coincides, except for tagging, with an arc of $T^\circ$, then instead of making tiles for segments of $\alpha$, make two tiles, for the triangles on both sides of $\alpha$.
If $\alpha$ coincides with the interior edge of a self-folded triangle, then the two tiles are identical.
Again at each endpoint where $\alpha$, construct a tile for each triangle corner (without making additional copies of the tiles for the triangles on both sides of $\alpha$).
The edges of these tiles are then identified in the natural way around each endpoint of~$\alpha$.
If $\alpha$ coincides with the interior edge of a self-folded triangle in $T^\circ$, then the tiles from the two sides of $\alpha$ are glued together to form a digon.
If $\alpha$ coincides with an interior edge of a punctured digon in $T^\circ$, then $\alpha'$ also coicides with an interior edge of a punctured digon in $T'$.
The union of the tiles is again a disk.

In either case, embed the disk into some larger marked surface $(\S',\M')$ with a tagged triangulation $T'$ containing the tiles as triangles, with edges of tiles constituting boundary segments of~$\S'$ if and only if they come from boundary segments of $\S$.
This can be done in such a way that $T$ consists of plain-tagged arcs and each triangle of $T$ has three distinct vertices and three distinct edges (as is already true of the tiles).
Crucially, $(T',\alpha')$ is tidy.

\begin{example}\label{tile ex}
\begin{figure}
\includegraphics{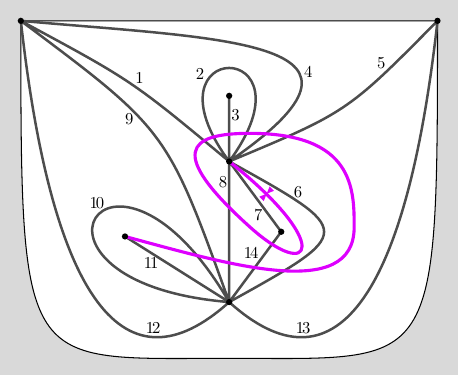}\qquad
\includegraphics{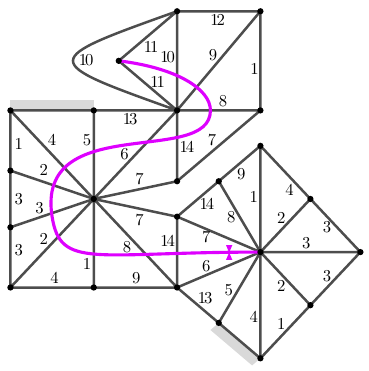}
\caption{The tile cover of a tagged arc.}
\label{tilex1}
\end{figure}
\begin{figure}
\includegraphics{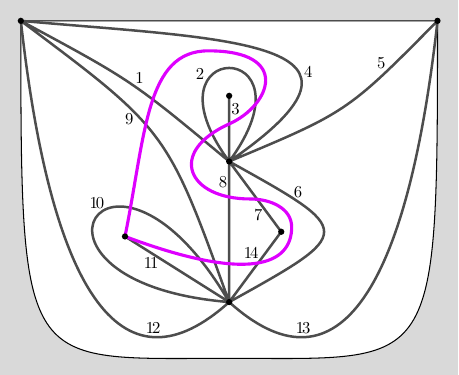}\qquad
\includegraphics{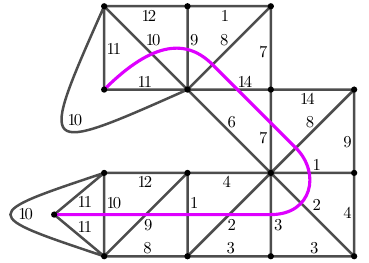}
\caption{The tile cover of a tagged arc.}
\label{tilex8}
\end{figure}
\begin{figure}
\includegraphics{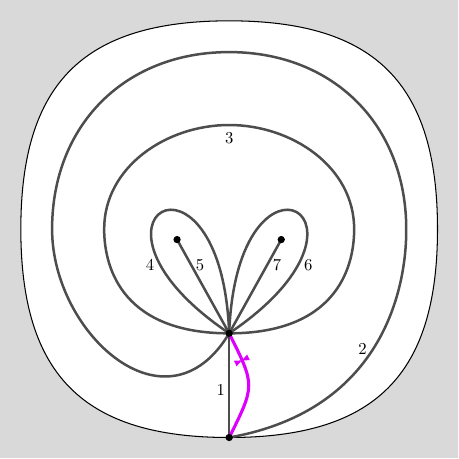}\qquad
\includegraphics{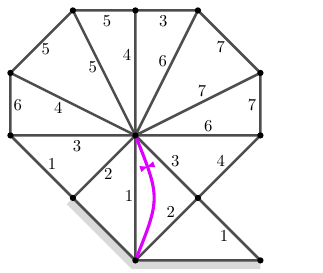}
\caption{The tile cover of a tagged arc.}
\label{tilex3}
\end{figure}
\begin{figure}
\includegraphics{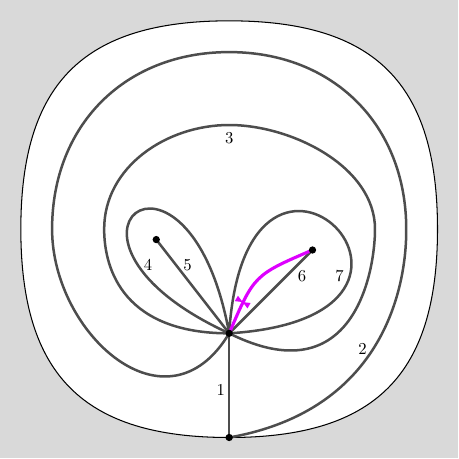}\qquad
\includegraphics{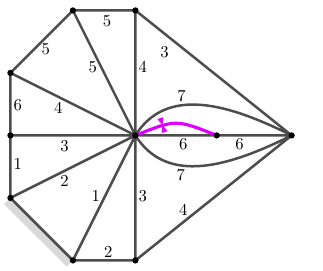}
\caption{The tile cover of a tagged arc.}
\label{tilex4}
\end{figure}
\begin{figure}
\includegraphics{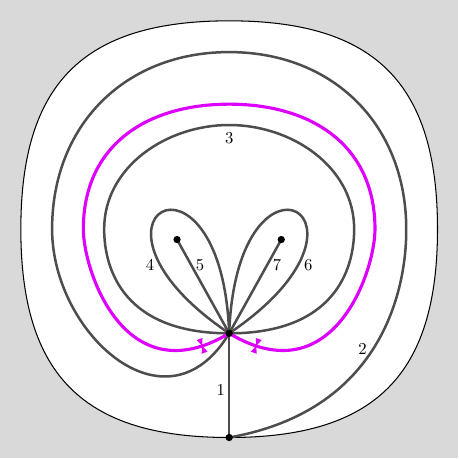}\qquad
\includegraphics{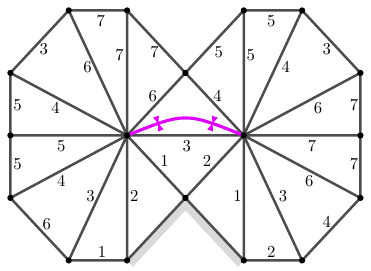}
\caption{The tile cover of a tagged arc.}
\label{tilex5}
\end{figure}
\begin{figure}
\includegraphics{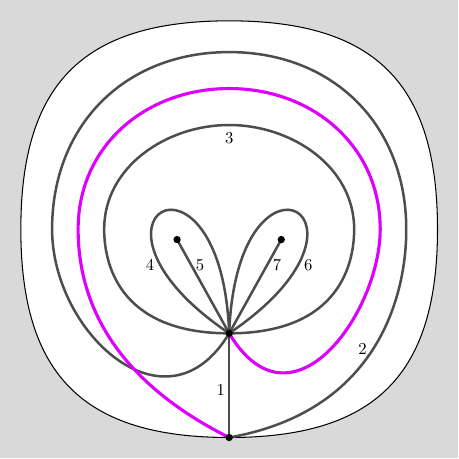}\qquad
\includegraphics{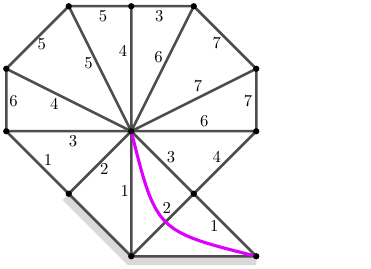}
\caption{The tile cover of a tagged arc.}
\label{tilex6}
\end{figure}
\begin{figure}
\includegraphics{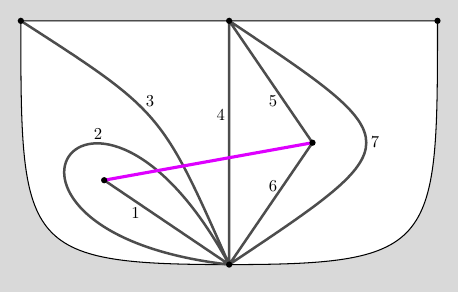}\qquad
\includegraphics{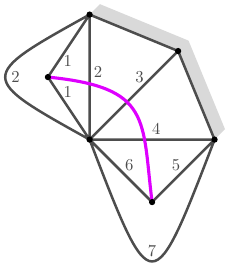}
\caption{The tile cover of a tagged arc.}
\label{tilex7}
\end{figure}
We illustrate the combinatorics of the tile cover in Figures~\ref{tilex1}--\ref{tilex7}.
In each figure, the left picture shows $\alpha$ as a tagged arc in $(\S,\M)$.
The marked surface $(\S,\M)$ and the tagged triangulation $T$ are the same in Figures~\ref{tilex1}--\ref{tilex8} and in Figures~\ref{tilex3}--\ref{tilex6}.
The ordinary triangulation~$T^\circ$ is pictured, and the labeling of arcs of $T^\circ$ implies a labeling of the tagged arcs of $T$ as usual.
(In \cref{tilex1} for example, $2$ labels the tagged arc that agrees with $3$ but with a notch at highest puncture shown, and similarly $10$ labels the tagged arc that agrees with $11$ but with a notch.)
In each figure, the right picture shows the tile cover $\alpha'$ in $(\S',\M')$.
The picture shows only the part of $(\S',\M')$ that is constructed explicitly.
Gray shading indicates certain triangle edges that are boundary segments in $(\S',\M')$.
All other edges shown are arcs, not boundary segments.
\cref{tilex3} continues the example of \cref{in T one fig}.
\cref{tilex4} continues the example of \cref{self-fold coinciding fig} and illustrates the reduction in \cref{red sec}:  In \cref{tilex4}, the arc $\alpha$ is no longer tagged notched at the right puncture, and correspondingly, the roles of arcs $6$ and $7$ have been exchanged.
\cref{tilex5} continues the example of \cref{in T two same fig}.
\end{example}

\begin{remark}\label{more complicated}
The definition of the tile cover uses more tiles than are required for the results of this section:
We could have left out the triangles surrounding endpoints of $\alpha$ that are tagged plain.
The extra triangles around punctures will be needed in \cref{sec:tropical} when we prove the full version of \cref{main}.
\end{remark}

The tiles that form the tile cover of $\alpha$ come equipped with a decorated hyperbolic structure:
By construction, each tile gets this structure from a triangle of $T^\circ$, and the structure agrees on shared edges of tiles.
The chosen horocycles at marked points in $\M$ lift to tiles along with the hyperbolic metric, and horocycles in adjacent tiles join to make horocycles in $(\S',\M')$.
In order to apply the results of \cref{sec:special} to the tile cover of $\alpha$, we need to correctly determine the lambda lengths of the arcs of $T'$ in this decorated hyperbolic structure to $(\S',\M')$.
(In fact, it is only necessary to determine the lambda lengths of the edges of tiles, because these determine the hyperbolic metric on the union of the tiles.)
This is essentially straightforward, but with one small wrinkle coming from self-folded triangles.

As before, let $x_\beta$ stand for the lambda length of a tagged arc $\beta\in T$, so that the initial cluster is $\set{x_\beta}{\beta\in T}$.
We emphasize that $\beta$ is a \emph{tagged} arc in $T$, not an ordinary arc in $T^\circ$.
For most arcs $\beta$, the distinction is meaningless, but if $\beta$ is tagged notched at one endpoint (and thus by our assumptions on $T$ coincides with another arc $\gamma$ of $T$ that is tagged plain), then $\beta$ is \emph{not} the ordinary arc (loop) that forms the outer edge of a self-folded triangle of $T^\circ$.
However, it is that loop that lifts to become an arc in the triangulation $T'$ of $(\S',\M')$.
In this situation, Equation~\eqref{loop eq lite} says that the loop has lambda length $x_\beta x_\gamma$.
Thus the lambda lengths of tile edges are as follows:
If the edge comes from a boundary segment of $(\S,\M)$, then its lambda length is $1$.
If the edge comes from an arc $\beta$ that is \emph{not} the outer edge of a self-folded triangle of $T^\circ$, then its lambda length is $x_\beta$.
If the edge comes an arc $\delta$ that \emph{is} the outer edge of a self-folded triangle of $T^\circ$, then its lambda length is $x_\beta x_\gamma$, where $\beta$ and $\gamma$ are the two tagged arcs of $T$ corresponding to that self-folded triangle.

\begin{example}\label{tile ex 2}
Continuing \cref{tile ex}, we describe the lambda lengths of the arcs and boundary segments shown in tile cover (the right picture) in \cref{tilex7}.
We use the numbering of arcs of $T^\circ$ shown in the left picture.
The boundary segments (indicated by gray shading) each have lambda length $1$.
Arcs labeled $2$ have lambda length $x_1x_2$.
Any other arc labeled $i$ has lambda length $x_i$.
\end{example}

\pagebreak

\subsection{Proof of the coefficient-free version of the main theorem}\label{proof lite sec}
The last ingredients in the proof of \cref{main lite} are three lemmas that relate a tagged arc~$\alpha$ to its tile cover $\alpha'$. 
All three lemmas have the hypothesis that wherever $\alpha$ has an endpoint that is the interior vertex of a self-folded triangle in $T^\circ$, the tagging of $\alpha$ is plain at that endpoint.
In each lemma, $\alpha'$ is the tile cover of $\alpha$ in $(\S',\M')$.

\begin{lemma}\label{lambda lift}
The lambda length $\lambda(\alpha)$ in $(\S,\M)$ equals the lambda length $\lambda(\alpha')$ in $(\S',\M')$.
\end{lemma}
\begin{proof}
We have decomposed~$\alpha$ into a finite collection of segments according to the triangulation~$T^\circ$, constructed a tile in~$T'$ for each segment with the same local hyperbolic metric, and transported the horocycles in the corresponding corners of the tiles.
Crucially, whenever $\alpha$ is tagged notched at one of its endpoints $p$, the length of the horocycle $h_p$ is the same as the length of the horocycle at the corresponding point in the tile cover.
(The neighborhood of every puncture $p$ maps isometrically to the neighborhood of the corresponding puncture in the tile cover, \emph{except} when $p$ is the interior vertex of a self-folded triangle of $T^\circ$.
However, by hypothesis, $\alpha$ must be tagged plain at any such puncture.)
Therefore, in any case where the lambda length of $\alpha$ involves a conjugate horocycle, the lift of that conjugate horocycle to the tile cover is precisely the conjugate horocycle in $(\S',\M')$.
We see that the length~$l(\alpha)$ of the segment of~$\alpha$ in~$(\S,\M)$ between the horocycles corresponding to its endpoints equals the length~$l(\alpha')$ of the segment of~$\alpha'$ in~$(\S',\M')$ between the horocycles corresponding to its endpoints.
Thus $\lambda(\alpha')=\exp(l(\alpha')/2)=\exp(l(\alpha)/2)=\lambda(\alpha)$.
\end{proof}

\begin{lemma}\label{g lift}
The $\g$-vector $\g_\alpha$ in $(\S,\M)$ equals the $\g$-vector $\g_{\alpha'}$ in $(\S',\M')$.
\end{lemma}
\begin{proof}
The monomial $\g_{\alpha'}$ is the product of the lambda lengths $\lambda(\gamma')$ raised to the exponent $-b_{\gamma'}(T',\kappa(\alpha'))$ over all arcs $\gamma\in T'$.
(Recall that $b_{\gamma'}(T',\kappa(\alpha'))$ is the entry, in position $\gamma$, of the shear coordinate of $\kappa(\alpha')$ with respect to $T'$.)
Since $T'$ has all arcs tagged plain, for any~$\gamma\in T'$, there is quadrilateral in $T'$ with $\gamma'$ as a diagonal.
There is a contribution $\lambda(\gamma')^{\pm1}$ to $\g_{\alpha'}$ every time~$\kappa(\alpha')$ crosses that quadrilateral through opposite sides, but by construction, $\kappa(\alpha')$ makes at most one crossing of every quadrilateral in $T'$.
For every quadrilateral crossed, there is a corresponding pair of triangles in $T$ (in $(\S,\M)$), crossed by $\alpha$, and sharing an edge that is an arc~$\gamma\in T^\circ$.
Possibly $\alpha$ crosses $\gamma$ many times, but this crossing makes the same contribution $\pm1$ as the contribution from $\alpha'$ crossing $\gamma'$.

If~$\gamma$ is not the exterior edge of a self-folded triangle in $T^\circ$, then $\lambda(\gamma)$ is $x_\gamma$ and, by construction~$\lambda(\gamma')$ is also $x_\gamma$.
If $\gamma$ is the exterior edge of a self-folded triangle in $T^\circ$, then let $\delta$ be the interior edge.
There are two possibilities:  
Either this crossing of $\gamma$ by $\kappa(\alpha)$ happens as $\kappa(\alpha)$ passes through the self-folded triangle, or this crossing happens as $\kappa(\alpha)$ spirals into the interior vertex of the triangle.

The case where $\kappa(\alpha)$ passes through the triangle is illustrated in \cref{tilex8shear} (\cref{g ex}).
In this case, by the definition of the shear coordinate in the presence of self-folded triangles, this crossing of $\gamma$ contributes the same $\pm1$ to the shear coordinates at the entries $\gamma$ and $\delta$.
Thus the contribution of this crossing of $\gamma$ is $x_\gamma x_\delta$.
By construction $\lambda(\gamma')$ is also $x_\gamma x_\delta$.
(In this case where~$\kappa(\alpha)$ passes through the self-folded triangle, it is impossible for $\gamma$ to be the interior edge of a self-folded triangle in $T^\circ$, because if so, any crossing of $\gamma$ by $\alpha$ lifts to a crossing of $\gamma'$ by $\alpha$ that contributes nothing to the shear coordinate.)

The case where $\kappa(\alpha)$ spirals into the interior edge appears in \cref{tilex8shear} but also more specifically in \cref{digonshear}.
This crossing either contributes $+1$ to the $\gamma$-entry of the shear coordinates of $\kappa(\alpha)$ and contributes nothing to the $\delta$-entry or contributes nothing to the $\gamma$-entry and contributes $-1$ to the $\delta$-entry, thus contributing either $x_\gamma$ or $x_\delta^{-1}$ to~$\g_\alpha$.
In either case, since $\alpha'$ is tagged plain at this endpoint, there is a contribution $-1$ to the shear coordinate of one of the arcs in $T'$ that is a lift of the interior edge.
(This occurs because $\kappa(\alpha')$ crosses that arc of $T'$ and then hits an arc of $T'$ that does not come from a tile.)
When the crossing contributes $+1$ to the $\gamma$-entry of the shear coordinates of $\kappa(\alpha)$, the crossing by $\kappa(\alpha')$ contributes $+1$ to the shear coordinate given by the lift of $\gamma$, and the net contribution to $\g_{\alpha'}$ is $(x_\gamma x_\delta)\cdot x_\delta^{-1}=x_\gamma$.
When the crossing contributes~$-1$ to the $\delta$-entry of the shear coordinates of $\kappa(\alpha)$, the crossing by $\kappa(\alpha')$ contributes $0$ to the shear coordinate given by the lift of $\gamma$, and the net contribution to $\g_{\alpha'}$ is~$x_\delta^{-1}$.
\end{proof}

\begin{example}\label{g ex}   
\cref{tilex8shear} shows the computations of $\g_\alpha$ and $\g_{\alpha'}$, continuing the example of \cref{tilex8}.
\begin{figure}
\begin{tabular}{cc}
\scalebox{1.05}{\includegraphics{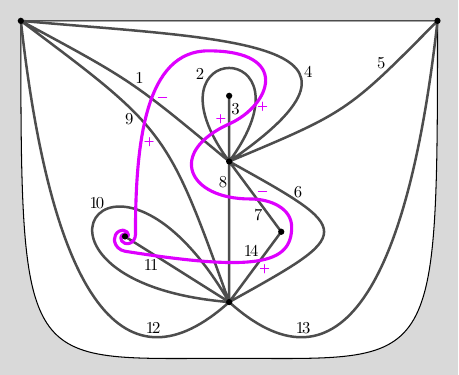}}
&
\begin{tabular}{l}\\[-197pt]
\hspace{-10pt}
\scalebox{0.85}{\includegraphics{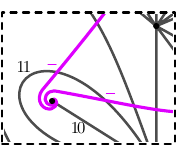}}\\[11pt]
\hspace*{3pt}
\scalebox{0.95}{\includegraphics{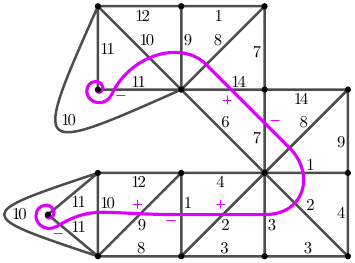}}
\end{tabular}\\
$\g_\alpha=x_{11}x_{14}^{-1}x_7x_2^{-1}x_3^{-1}x_1x_9^{-1}x_{11}$
&
$\g_{\alpha'}=x_{11}x_{14}^{-1}x_7(x_2x_3)^{-1}x_1x_9^{-1}x_{11}$
\end{tabular}
\caption{Computing $\g_\alpha$ and $\g_{\alpha'}$.}
\label{tilex8shear}
\end{figure}
Both endpoints of $\alpha$ are at the same puncture, the interior vertex of a self-folded triangle.
Following $\kappa(\alpha)$ to the right from this point and keeping in mind the minus sign in the exponent in \cref{g def}, the contributions to $\g_\alpha$ are $x_{11}x_{14}^{-1}x_7x_2^{-1}x_3^{-1}x_1x_9^{-1}x_{11}$.
If we follow~$\alpha'$ in the same direction (from the top of the picture on the right), the contributions to $\g_{\alpha'}$ are~$x_{11}x_{14}^{-1}x_7(x_2x_3)^{-1}x_1x_9^{-1}x_{11}$.
\cref{digonshear} shows two much simpler examples that illustrate the two possibilities for shear coordinates of $\kappa(\alpha)$ and $\kappa(\alpha')$ when~$\alpha$ has an endpoint at the interior vertex of a self-folded triangle.
\begin{figure}
\begin{tabular}{ccc}
\includegraphics{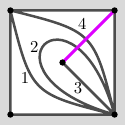}
&
\hspace*{10pt}
\includegraphics{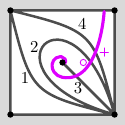}
\includegraphics{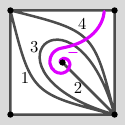}
\hspace*{10pt}
&
\includegraphics{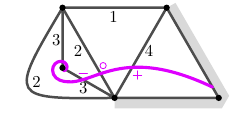}\\
$\alpha$ & $\g_\alpha=x_3x_4^{-1}$&$\g_{\alpha'}=x_3x_4^{-1}$\\[10pt]
\hline\\[10pt]
\includegraphics{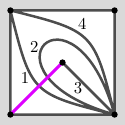}
&
\hspace*{10pt}
\includegraphics{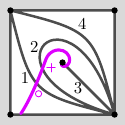}
\includegraphics{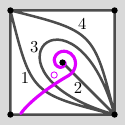}
\hspace*{10pt}
&
\includegraphics{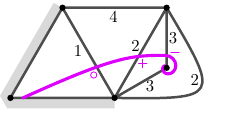}\\
$\alpha$ & $\g_\alpha=x_2^{-1}$&$\g_{\alpha'}=(x_2x_3)^{-1}x_3$\\[10pt]
\hline
\end{tabular}
\caption{Computing $\g_\alpha$ and $\g_{\alpha'}$ (self-folded triangle possibilities).}
\label{digonshear}
\end{figure}
\end{example}

\begin{lemma}\label{F lift lite}
$F(P_{\alpha})\big|_{y=1}=F(P_{\alpha'})\big|_{y=1}$, where the notation of specializing to ``$y=1$'' is shorthand for setting every tropical variable to $1$.
\end{lemma}
\begin{proof}
The posets are isomorphic by construction, and we check that corresponding ideals have the same weight in the sense of $\tilde w$ (setting the tropical variables to $1$ in $w$).
Indeed, it is simple to check that corresponding elements have the same weight.

If $\gamma$ is an arc in $T^\circ$, then $\hy_\gamma\big|_{y=1}$ is the product $\prod_{\beta\in T}x_\beta^{b_{\beta\gamma}}$.
Here $b_{\beta\gamma}$ is the $\beta\gamma$-entry of the exchange matrix (\ie signed adjacency matrix) associated to $T$.
If $\gamma$ is not the interior edge of a self-folded triangle, then there is a positive contribution to the exponent of $x_\beta$ for every triangle of~$T^\circ$ in which $\beta$ precedes $\gamma$ in clockwise order, and, if $\beta$ is the interior edge of a self-folded triangle with exterior edge $\delta$, a positive contribution to the exponent of $x_\beta$ for every triangle of~$T^\circ$ in which $\beta$ precedes $\delta$ in clockwise order.
Negative contributions to the exponent of $x_\beta$ are the same, but with ``clockwise'' replaced with ``counterclockwise'' throughout.
(If $\gamma$ is the exterior edge of a self-folded triangle and $\beta$ is the interior edge, then the positive and negative contributions to the exponents of $x_\beta$ cancel.)
If $\gamma$ is the interior edge of a self-folded triangle, then by the definition of the signed adjacency matrix, the product $\prod_{\beta\in T}x_\beta^{b_{\beta\gamma}}$ is the same as for the exterior edge.

For a lift of $\gamma$ to the edge $\gamma'$ of a tile in the tile cover $(\S',\M')$, the corresponding product is~$\prod_{\beta'\in T'}\lambda(\beta')^{b_{\beta'\gamma'}}$.
The contributions to this product have the same description, except that there are no self-folded triangles in $T'$.
Instead, when $\beta'$ is the lift of the exterior edge $\beta$ of a folded triangle whose interior edge is $\delta$, the lambda length $\lambda(\beta')$ is $x_\beta x_\delta$.
When $\gamma$ the exterior edge of a self-folded triangle with interior edge $\delta$, the tile for that triangle has edges with lambda lengths $x_\gamma x_\delta$, $x_\delta$, and $x_\delta$.
Thus the contributions of $x_\delta$ to $\prod_{\beta'\in T'}\lambda(\beta')^{b_{\beta'\gamma'}}$ cancel, just as they cancel in $\prod_{\beta\in T}x_\beta^{b_{\beta\gamma}}$.
When $\gamma$ is the interior edge of a self-folded triangle with exterior edge $\delta$, the lift $\gamma'$ is the edge of two tiles each with edges of lambda lengths $x_\gamma$, $x_\gamma$ and $x_\gamma x_\delta$, and the edges with lambda lengths $x_\gamma x_\delta$ are adjacent in the quadrilateral formed by the two tiles.  
Thus all of the contributions to $\prod_{\beta'\in T'}\lambda(\beta')^{b_{\beta'\gamma'}}$ cancel, so that the product is~$1$.

Suppose $e$ is an element of $P_\alpha$ and $e'$ is the corresponding element of $P_{\alpha'}$.
The above considerations show that if $w(e)=\hy_\gamma$ (so that $\gamma$ is not the interior edge of a self-folded triangle), then $\tilde w(e)$ and $\tilde w(e')$ both equal $\prod_{\beta\in T}x_\beta^{b_{\beta\gamma}}$.
They also show that if $w(e)=\hy_\gamma/\hy_\delta$ (so that $\gamma$ is the interior edge of a self-folded triangle with exterior edge $\delta$), then $\tilde w(e)$ and $\tilde w(e')$ both equal~$1$.
\end{proof}

Finally, we prove the coefficient-free case of the main result.

\begin{proof}[Proof of \cref{main lite}]
By \cref{reduce lite}, we need only prove the theorem for tagged arcs $\alpha$ that are plain at any endpoint that is the interior vertex of a self-folded triangle.
We write \eqref{main lite eq} as~${\lambda(\alpha)=\g_\alpha\cdot F(P_\alpha)\big|_{y=1}}$, with notation as in \cref{F lift lite}.

Suppose $\alpha$ is a tagged arc in $(\S,\M)$.
We argue by induction on the number of elements of $P_\alpha$ that $\lambda(\alpha)=\g_\alpha\cdot F(P_\alpha)\big|_{y=1}$.
Specifically, the inductive hypothesis is that the formula holds for any tagged arc $\alpha'$ in \emph{any} triangulated marked surface such that $P_{\alpha'}$ has fewer elements than $P_\alpha$.
The base of the induction is when $P_\alpha$ is empty.
In this case, the sum is empty, and $\alpha$ is either an arc in $T$ or is the exterior edge of a self-folded triangle in $T^\circ$.
If $\alpha$ is an arc in $T$, then $\g_\alpha=x_\alpha$, as desired.
If $\alpha$ is the exterior edge of a self-folded triangle in $T^\circ$, then $\g_\alpha=x_\beta x_{\beta'}$, where $\beta$ and $\beta'$ are the two tagged arcs in $T$ associated to the self-folded triangle.
This is correct in light of \eqref{loop eq lite}.

If $P_\alpha$ is not empty, then construct the tile cover $\alpha'$ in a marked surface $(\S',\M')$ with a triangulation $T'$.
By construction, $\alpha'$ has distinct endpoints and does not cross itself, so it is a tagged arc in $(\S',\M')$.
Also, $P_{\alpha'}$ is isomorphic (ignoring the weighting for a moment) to $P_\alpha$.
Furthermore,~$(T',\alpha')$ is tidy, so there is an arc $\gamma'\in T'$ such that $\hy_{\gamma'}$ occurs exactly once as a label in $P_\alpha$, and \cref{ideal decomp} applies.
Applying \eqref{ideal decomp eq} to $\alpha'$ and multiplying through by $x_\gamma\cdot\g_{\alpha'}$, we obtain 
\begin{equation}\label{ideal exch eq lite}
x_{\gamma'}\cdot\g_{\alpha'}\cdot F(P_{\alpha'})=x_{\gamma'}\cdot\g_{\alpha'}\cdot F(\Pb)+x_{\gamma'}\cdot\g_{\alpha'}\cdot\hy_\orange\cdot F(\Pr).
\end{equation}
The labeled posets $P_\blue$ and $P_\red$, their labels in the $\hy$-monomials, and the monomial $\hy_\orange$ all refer, of course, to the triangulation $T'$ of $(\S',\M')$.

The quantity $F(P_\blue)$ in \cref{ideal exch eq lite} is the product of one or two polynomials $F(P_{\delta'})$ for tagged arcs $\delta'$ in $(\S',\M')$, and each of these posets $P_{\delta'}$ has strictly fewer elements than $P_{\alpha'}$, which has the same number of elements as~$P_\alpha$.
By induction, the lambda length $\lambda({\delta'})$ in $(\S',\M')$ is~$\g_{\delta'}\cdot F(P_{\delta'})\big|_{y=1}$.
Each $\delta'$ is a tagged arc, so this lambda length is a cluster variable $x_{\delta'}$.
Similarly,~$F(P_\red)$ is the product of one or two polynomials $F(P_{\delta'})$ and by induction, the lambda length of each $\delta'$ is $\g_{\delta'}\cdot F(P_{\delta'})\big|_{y=1}$, which is again a cluster variable $x_{\delta'}$.

Specializing all tropical variables to $1$ and appealing to \cref{g prop}, the right side of \eqref{ideal exch eq lite} becomes the right side of the exchange relation exchanging $x_{\alpha'}$ and $x_{\gamma'}$.
We conclude that the cluster variable $x_{\alpha'}$ is given by $\g_{\alpha'}\cdot F(P_{\alpha'})\big|_{y=1}$.
This cluster variable is the lambda length $\lambda(\alpha')$.

We have seen that $\lambda(\alpha')=\g_{\alpha'}\cdot F(P_{\alpha'})\big|_{y=1}$, which shows that~$\lambda(\alpha)=\g_{\alpha}\cdot F(P_{\alpha})\big|_{y=1}$ by Lemmas~\ref{lambda lift},~\ref{g lift} and~\ref{F lift lite}.
\end{proof}

\section{Tropical hyperbolic geometry and the principal coefficients case}\label{sec:tropical}

We now prove the full version of \cref{main}, which can be rephrased as follows: the cluster variable~$x_\alpha$ equals the $\g$-vector~$\g_\alpha$ times the weighted sum~$F(P_\alpha)$.
The proof follows the same outline and uses some of the same results as in \cref{sec:hyperbolic}.
We begin by explaining how cluster variables with principal coefficients can be realized in terms of laminations and certain hyperbolic lengths and corresponding tropical hyperbolic lengths.
More details are found in \cite[Chapters~9--15]{FominThurston}.

\subsection{Opened surfaces}\label{opened sec}
As before, we are given a marked surface $(\S,\M)$ and a tagged triangulation $T$ with all arcs tagged plain, except possibly at some punctures $p$ incident to exactly two tagged arcs of~$T$, identical except for opposite taggings at $p$.
The corresponding ordinary triangulation is $T^\circ$.
The tagged triangulation determines a multi-lamination $\L$ of $(\S,\M)$.
General choices of~$\L$ lead to arbitrary coefficients of geometric type, but for principal coefficients at $T$, the multi-lamination~$\L$ consists of elementary laminations $L_\gamma$ for tagged arcs $\gamma\in T$, as defined in \cref{subsec:exchange}.
We will also consider a multilamination $\AA$ consisting of elementary laminations for arcs in~$T^\circ$.

Recall the construction reviewed in \cref{lam sec}:
Once $(\S,\M)$ and $T$ are fixed, there is a bijection from the set of tuples of positive numbers indexed by $T$ to the set of decorated hyperbolic structures on $(\S,\M)$---meaning hyperbolic structures with constant curvature $-1$ and distinguished horocycles about each marked point---with all boundary segments having lambda length~$1$.
The tuples of positive numbers describe the lambda lengths $x_\alpha$ of the arcs $\alpha\in T$.
Once a particular decorated hyperbolic structure is chosen (or equivalently, once the lambda lengths $x_\alpha$ are chosen), every tagged arc has a lambda length.
The lambda length of a tagged arc $\gamma$, as a function of the lambda lengths $x_\alpha$, is the coefficient-free cluster variable associated to $\gamma$.

The analogous construction with principal coefficients is summarized as follows:
Once $(\S,\M)$ and $T$ are fixed (and thus $\L$ is fixed to be the elementary laminations associated to $T$), we choose \emph{two} tuples of positive numbers, each tuple indexed by $T$.
For each tagged arc $\alpha\in T$, we choose a positive real number $x_\alpha$ that will eventually be its laminated lambda length.
For each elementary lamination $L_\gamma$ (for $\gamma\in T$), we choose a positive real weight $y_\gamma$ that will play the role of a tropical variable.
(In \cite{FominThurston}, these tropical variables are denoted $q_i$.)
These pairs of tuples determine an opening of $\S$ (a surface in which some of the punctures are replaced by holes, meaning removed disks with an orientation), a hyperbolic metric with curvature $-1$ on the opened surface and horocycles on all of the non-opened marked points, with boundary conditions on the metric that determine the lambda lengths of the boundary segments and the holes in terms of the $y_\gamma$.
Once an opening, hyperbolic metric, and horocycles are chosen (or equivalently, once the laminated lambda lengths $x_\alpha$ and the weights $y_\gamma$ are chosen), every tagged arc has a laminated lambda length.
The laminated lambda length of an arc $\gamma$, as a function of the $x_\alpha$ and the $y_\gamma$, is the principal-coefficients cluster variable associated to $\gamma$.
We now provide more details.

The \newword{transverse measure} $l_{L_\gamma}(p)$ of a puncture $p$ with respect to an elementary lamination~$L_\gamma$ is $1$ for each endpoint of $\gamma$ that is at~$p$ and tagged plain and $-1$ for each endpoint of $\gamma$ that is at~$p$ and tagged notched, and otherwise zero.
Since it is possible for $\gamma$ to have both endpoints at~$p$ if they have the same tagging, these contributions can combine to make $2$ or $-2$ but can never cancel.
(In this definition of transverse measure of a puncture, we have combined \cite[(14.2)]{FominThurston} with the definition of $L_\gamma$ and the definition of a lift of a lamination that will be given below.)
The \newword{transverse measure} $l_{L_\gamma}(\beta)$ of a boundary segment $\beta$ with respect to $L_\gamma$ is the number of times ($0$, $1$, or $2$) that $L_\gamma$ has an endpoint on $\beta$.  
Thus if $q$ is the right endpoint of $\beta$ (right as we stand on~$\beta$ looking towards the interior of $\S$), then in light of the definition of $L_\gamma$, the transverse measure of $\beta$ is the number of endpoints of $\gamma$ that are at $q$.
The \newword{tropical lambda length} is defined to be $c(s)=\prod_{\gamma\in T}y_\gamma^{l_{ L_\gamma}(s)}$ for $s$ a puncture or boundary segment.
(Later, when we define tropical lambda lengths of tagged arcs, we use the notation $c_{\overline\L}$.
There, $\overline\L$ stands for a particular lift of the multilamination $\L$ to the opened surface.
However, the tropical lambda lengths of punctures and boundary segments are independent of the lift, and we suppress the subscript $\overline\L$ here.)

To make an \newword{opening} of $(\S,\M)$, we choose some subset $\P$ of the punctures and open those by replacing each $p\in\P$ with a hole $C_p$ (a new boundary component where an open disk is removed from $\S$) with marked point $M_p$ on $C_p$.
Each $C_p$ for $p\in\P$ is given an orientation, clockwise or counterclockwise.
The opening depends on the choice of the weights $y_\gamma$.
Specifically, we use these weights to assign lambda lengths to holes and boundary segments in the opened surface, thus constraining the choice of which holes are opened, the orientations of opened holes, the hyperbolic metric, and the choice of horocycles at marked points on the boundary, as we now explain.

Assuming some opening of $(\S,\M)$, choose a finite-area hyperbolic metric $\sigma$ with constant curvature~$-1$ with a cusp at each \emph{non-opened} marked point, with each boundary segment a geodesic, including the holes $C_p$.
(We emphasize that there are no cusps at the points $M_p$ for opened punctures~$p$.)
The \newword{signed length} $l(p)$ of an opened puncture $p$ is the hyperbolic length of $C_p$ if~$C_p$ is oriented clockwise in the opened surface or the negative of the length of $C_p$ if it is oriented counterclockwise.
The signed length of an unopened puncture is~$0$.
Assuming also some choice of horocycles at marked points on the boundary, we define the \newword{signed length} $l(\overline\alpha)$ of a boundary component between the horocycles at its endpoints as in \cref{lam sec}.
The \newword{lambda length} of a puncture or boundary segment~$s$ is defined to be $\lambda(s)=e^{l(s)/2}$.

We constrain the choice of the hyperbolic metric $\sigma$ by specifying lambda lengths for the punctures and boundary segments:
Every puncture or boundary segment $s$ had lambda length $\lambda(s)$ equal to the tropical lambda length~$c(s)$.
The tropical lambda lengths of punctures and boundary segments depend only on the weights~$y_\gamma$.
(For more general coefficients, they depend on the choice of a weighted multilamination, but we have fixed principal coefficients and thus fixed the multilamination.)
Thus the assigned lambda lengths of punctures and boundary segments also depend only on the weights $y_\gamma$.
In particular, the weights $y_\gamma$ determine the opening of surface:
If $c(p)=1$ for some puncture $p$, then the puncture $p$ is not opened.
If $c(p)>1$, then $p$ is opened and~$C_p$ is oriented clockwise, and if $c(p)<1$ then $p$ is opened and $C_p$ is oriented counterclockwise.

\subsection{Laminated lambda lengths}\label{lam lam sec}
We have seen how the weights $y_\lambda$ determine an opening of $(\S,\M)$ and place constraints on the hyperbolic metric and horocycles by specifying lambda lengths of the opened punctures and the boundary segments.
We will eventually impose additional constraints coming from the~$x_\gamma$.
But first, we define laminated lambda lengths of tagged arcs.

At each $M_p$, for $p$ an opened puncture, there is a \newword{perpendicular horocyclic segment}~$h_p$ near $C_p$.
This is a short segment of the horocycle that contains $M_p$ and is perpendicular to every geodesic spiraling into $C_p$ in the direction of the chosen orientation.
We also define a \newword{conjugate perpendicular horocycle} $\overline h_p$ near $C_p$:
Let $\overline M_p$ be the point on $C_p$ a \emph{signed} distance~$2\ln\left|\lambda(p)-\lambda(p)^{-1}\right|$ from $M_p$ in the direction \emph{against} the orientation of $C_p$.
Let $\overline h_p$ be the segment at $\overline M_p$ 
that is perpendicular to every geodesic spiraling into $C_p$ opposite of the chosen orientation.

Now let $\alpha$ be any tagged arc in $(\S,\M)$.
We will need two different types of lifts of $\alpha$ to the opened surface.
First, an arbitrary lift $\overline\alpha$ in a topological sense, meaning that $\overline\alpha$ maps to $\alpha$ under that map that collapses all opened punctures back down to punctures in $(\S,\M)$.
Second, a corresponding geodesic~$\alpha_\sigma$ in the opened surface:
The geodesic $\alpha_\sigma$ is like a lift of $\alpha$ except that at any endpoint $p$ of alpha that is opened, the geodesic spirals around $C_p$, with the direction of the chosen orientation of $C_p$ if $\alpha$ is tagged plain at $p$ or against the orientation of $C_p$ if $\alpha$ is tagged notched there.

We will define the laminated lambda length of $\alpha$ as the ratio of the ordinary lambda length and tropical lambda length of the lift $\overline\alpha$.
In particular, this ratio will not depend on the choice of lift.

We first define the lambda length of $\overline\alpha$, beginning with the case where, at any endpoints of $\alpha$ that are opened punctures, $\overline\alpha$ twists far enough about $C_p$ in the direction that $\alpha_\sigma$ spirals.
We choose a segment $[\alpha_\sigma]$ of $\alpha_\sigma$ that starts at the intersection of $\alpha_\sigma$ with a horocycle or perpendicular horocyclic segment at one endpoint of $\alpha$ (a conjugate horocycle or conjugate perpendicular horocyclic segment if $\alpha$ is notched at that endpoint), follows $\alpha_\sigma$ and ends at the intersection of $\alpha_\sigma$ with a horocycle or perpendicular horocyclic segment at the other endpoint (again conjugate if $\alpha$ is notched there).
When $\alpha_\sigma$ spirals about an opened puncture at one or both endpoints, there are infinitely many choices of such a segment, but we choose it so that when we extend $[\alpha_\sigma]$ in a particular way, we obtain a curve that is isotopic to $\overline\alpha$:
At an endpoint $p$ of $\alpha$ that is not an opened puncture, we extend $[\alpha_\sigma]$ along $\alpha_\sigma$ all the way to~$p$.
At an endpoint $p$ of $\alpha$ that is an opened puncture where~$\alpha$ is tagged plain, we extend $[\alpha_\sigma]$ by following the perpendicular horocyclic segment~$h_p$ to~$M_p$.
At an endpoint $p$ of $\alpha$ that is an opened puncture where $\alpha$ is tagged notched, we extend~$[\alpha_\sigma]$ by following the conjugate perpendicular horocyclic segment $\overline h_p$ to $\overline M_p$ and then following $C_p$ to~$M_p$, in the direction that \emph{agrees} with the orientation of $C_p$.
The \newword{signed length} $l([\alpha_\sigma])$ is the hyperbolic length of $[\alpha_\sigma]$, signed as in \cref{lam sec}.
(Thus $l([\alpha_\sigma])$ is negative when the endpoint of~$[\alpha_\sigma]$ where~$\alpha_\sigma$ crosses a horocycle $h_p$ is closer to the other endpoint of~$\alpha$ than it is to $p$.
When $p$ is an opened puncture, one can choose the lift $\overline\alpha$ so that $[\alpha_\sigma]$ spirals far enough about $C_p$ to make~$l([\alpha_\sigma])$ positive.)
The \newword{lambda length} $\lambda(\overline\alpha)$ is $e^{l([\alpha_\sigma])/2}$.

If $\overline\alpha$ does not spiral far enough about some opened puncture, its lambda length is defined by the following rule:
When we pass from $\overline\alpha$ to a different lift with one additional spiral in the clockwise direction, we multiply the lambda length by $\lambda(p)$.
By adding clockwise and/or counterclockwise twists at one or both endpoints of $\overline\alpha$, we obtain a different lift that ``twists far enough'' as in the paragraph above, and then apply the rule to find the lambda length of $\overline\alpha$.
(The well-definition of the lambda length, in light of this rule, is establishes as \cite[Lemma~10.7]{FominThurston}.  
See \cite[Definition~10.6]{FominThurston}.)

We next define the tropical lambda length $c_{\overline\L}(\overline\alpha)$ of $\overline\alpha$.
Following \cite[Section~14]{FominThurston}, we define the tropical lambda length for a lift of $\alpha$ to a ``fully'' opened surface, meaning a surface where \emph{every} puncture of $(\S,\M)$ is opened.  
We reuse the notation $\overline\alpha$ for this lift.
As explained below in \cref{fully rem}, this is harmless because the tropical lambda length is independent of the details of how $\alpha$ is lifted to the fully opened surface near punctures that are not opened in the partially opened surface. 

We arbitrarily choose, for each elementary lamination $L_\gamma$, a \newword{lifted lamination} $\overline L_\gamma$ in the (fully) opened surface that corresponds to $L_\gamma$ except that anywhere $L_\gamma$ spirals into a puncture $p$ that is opened, the lift $\overline L_\gamma$ instead ends on $C_p\setminus\{M_p\}$ and formally retains a ``memory'' of the spiral in the following sense:
$\overline L_\gamma$ formally includes the data of an orientation of $C_p$ in the direction \emph{opposite} to the direction that $L_\gamma$ spirals.
(Recall that the opening of the surface also includes an orientation of $C_p$ for every puncture $p$ that is opened.
The orientations in the opening and orientations in the lifted laminations are independent, and we will be careful to distinguish them.)
We write $\overline\L$ for the multilamination consisting of the chosen lift $\overline L_\gamma$ of each $L_\gamma$.
The laminated lambda lengths of tagged arcs depend on this choice of liftings, but we fix the choice and we need not be concerned with the dependence.
For details, see \cite[Remark 14.8]{FominThurston}.

To define the \newword{tropical lambda length} $c_{\overline\L}(\overline\alpha)$ of $\overline\alpha$, we first suppose that, at each endpoint $p$ of~$\alpha$, the lift $\overline\alpha$ twists sufficiently many times around $C_p$ in the direction given by $\overline L_\gamma$ or in the opposite direction if $\alpha$ is tagged notched at the shared endpoint.
The tropical lambda length is defined by way of the transverse measure.
We assume that isotopy representatives of $\overline\alpha$ and~$\overline L_\gamma$ have been chosen to minimize intersections.
The \newword{transverse measure} $l_{\overline L_\gamma}(\overline\alpha)$ of $\overline\alpha$ relative to a lifted elementary lamination $\overline L_\gamma$ is the number of intersections between $\overline\alpha$ and $\overline L_\gamma$, plus an additional contribution $|l_{L_\gamma}(p)|$ for each endpoint $p$ where $\alpha$ is tagged notched.
(The additional contribution~$|l_{L_\gamma}(p)|$ is $1$ or $2$; it is $2$ if and only if $\gamma$ has both endpoints at~$p$.)
The \newword{tropical lambda length} is defined to be $c_{\overline\L}(\overline\alpha)=\prod_{\gamma\in T}y_\gamma^{l_{\overline L_\gamma}(\overline\alpha)}$.

If $\overline\alpha$ does not twist sufficiently far about some opened puncture~$p$, then the tropical lambda length $c_{\overline\L}(\overline\alpha)$ is determined using the following rule:  When we pass from $\overline\alpha$ to a different lift with one additional spiral in the clockwise direction (or counterclockwise if $\alpha$ is notched at $p$), we multiply the tropical lambda length by $c(p)$.

\begin{remark}\label{fully rem}
It is harmless to define the tropical lambda length of $\overline\alpha$ in terms of a lift to the fully opened surface.
To see why, recall that a puncture $p$ is un-opened if and only if $c(p)=1$.
Since lifts of $\alpha$ that differ only in the number of twists at $p$ are related by multiples of $c(p)$, the tropical lambda length of $\overline\alpha$ is independent of how we lift $\overline\alpha$ further to the fully opened surface.
\end{remark}

The \newword{laminated lambda length} of $\alpha$ is $x_{\overline\L}(\alpha)=\frac{\lambda(\overline\alpha)}{c_{\overline\L}(\overline\alpha)}$.
Because of the rules for how lambda lengths and tropical lambda lengths change when we change the lift of $\alpha$ and because ${\lambda(p)=c(p)}$ for every puncture $p$, the laminated lambda length is well defined (i.e.\ independent of the choice of lift $\overline\alpha$).

Once the tagged triangulation $T$ and the lifts $\overline\L$ of the elementary laminations~$L_\gamma$ for~${\gamma\in T}$ are fixed, the choice of positive real weights $y_\gamma$ and the choice of positive real numbers~$x_\gamma$ uniquely determines an opening of $(\S,\M)$, hyperbolic metric on the opened surface, and horocycles about the non-opened marked points such that $x_{\overline\L}(\gamma)=x_\gamma$ for every $\gamma\in T$ \cite[Corollary~15.5]{FominThurston}.  
Thus, given $T$, $\overline\L$, the $y_\gamma$ and the $x_\gamma$, every tagged arc in $(\S,\M)$ has a well-defined (principal-coefficients) laminated lambda length.
The laminated lambda lengths of tagged arcs, when viewed as functions of the $x_\gamma$ and the $y_\gamma$, are the cluster variables in the cluster algebra with initial cluster $\{x_\gamma:\gamma\in T\}$, initial exchange matrix equal to the signed adjacency matrix of $T$, principal coefficients at $T$, and tropical variables $y_\gamma$ for $\gamma\in T$.
(This is part of \cite[Theorem~15.6]{FominThurston}.)

Recall from \cref{lam sec} (\cref{loop eq lite}) that the lambda lengths of the tagged arcs in $T$ determine the lambda lengths of the ordinary arcs in $T^\circ$ and vice versa.
Precisely the same thing is true for laminated lambda lengths.
Suppose $\beta$ and $\gamma$ are tagged arcs in $T$ that coincide except that both are tagged plain at one endpoint $q$ but they have different taggings at $p$, and suppose that $\delta$ is the ordinary arc in $T^\circ$ that is a loop at $q$, following $\beta$ and $\gamma$ around $p$.
Then 
\begin{equation}\label{loop eq}
x_{\overline\L}(\delta)=x_{\overline\L}(\beta)x_{\overline\L}(\gamma).
\end{equation}
The ordinary arc $\delta$ is not a tagged arc, but it has a laminated lambda length, measured from~$h_q$ to~$h_q$.
\cref{loop eq} is a combination of \cite[Lemma~10.14]{FominThurston}, which says that $\lambda(\overline\delta)=\lambda(\overline\beta)\lambda(\overline\gamma)$ for appropriate lifts $\overline\delta$, $\overline\beta$, and $\overline\gamma$, and \cite[Lemma~14.10]{FominThurston}, which says that $c_{\overline\L}(\overline\delta)=c_{\overline\L}(\overline\beta)c_{\overline\L}(\overline\gamma)$.

\subsection{Proof of the main theorem}\label{proof sec}
In this section, we prove \cref{main}.

Let $\alpha$ be a tagged arc in $(\S,\M)$.
For the proof of \cref{main lite}, we constructed the tile cover~$\alpha'$ of $\alpha$ in a surface $(\S',\M')$ with triangulation $T'$ and lifted the hyperbolic structure of $(\S,\M)$ to the tile cover, so that lambda lengths were preserved.
For the proof of \cref{main}, we lift all of the data that defines the laminated lambda lengths to the tile cover.
As before, when passing to the tile cover, we assume that when $\alpha$ has an endpoint that is the interior vertex of a self-folded triangle in $T^\circ$, the tagging of $\alpha$ is plain at that endpoint.

Start with the tagged triangulation $T$ of $(\S,\M)$ and choose positive real numbers $x_\gamma$ and $y_\gamma$ for each tagged arc $\gamma\in T$.
Let these choices determine an opening of $(\S,\M)$, a hyperbolic metric on the opened surface, and a choice of horocycles about the non-opened marked points, as in \cref{lam lam sec}.
In particular, the laminated lambda length of each tagged arc $\gamma\in T$ is $x_\gamma$.

We lift these laminated lambda lengths to determine laminated lambda lengths of all tile edges in $(\S,\M)$.
\cref{loop eq} determines the laminated lambda lengths of all arcs in $T^\circ$ that are not in~$T$:
For every self-folded triangle in $T^\circ$ corresponding to tagged arcs $\beta$ and $\gamma$ in $T$, with $\beta$ tagged notched and $\gamma$ tagged plain at the interior vertex, the exterior edge of the self-folded triangle has laminated lambda length $x_\beta x_\gamma$.
Thus we lift laminated lambda lengths of tagged arcs in~$T$ to laminated lambda lengths of tile edges in~$T'$ exactly as we lifted lambda lengths of tagged arcs in \cref{tile cov sec}.
(See \cref{tile ex 2}.)

The weighted multilamination on $(\S,\M)$ that defines an opening (and places some constraints on the metric) consists of elementary laminations $L_\gamma$ with weights $y_\gamma$, for $\gamma\in T$.
Since the tagged arcs $\gamma\in T$ don't necessarily all lift to the tile edges in $T'$ (but rather the arcs in $T^\circ$ do), this weighted multilamination does not lift to the right weighted lamination on $(\S',\M')$.
We will define an alternative weighted multilamination $\AA$ that defines the same opening and places the same constraints on the metric.
Specifically, $\AA$ consists of the elementary lamination~$L_\gamma$ for each~$\gamma\in T^\circ$, with weight is $y_\gamma$ unless $\gamma$ is an edge of a self-folded triangle in $T^\circ$.
For the edges of a self-folded triangle corresponding to tagged arcs $\beta$ and $\gamma$ in $T$, with $\beta$ tagged notched and $\gamma$ tagged plain at the interior vertex, the weights are as follows:
Let $\delta\in T^\circ$ be the exterior edge of the self-folded triangle and identify the interior edge with $\gamma$.
Then the weight of $L_\gamma$ in $\AA$ is $\frac{y_\gamma}{y_\beta}$ and the weight of $L_\delta$ is $y_\beta$.

\begin{example}\label{tile ex 2 remix}
We describe the weights in the alternative multilamination $\AA$ in the left picture of \cref{tilex1}.
Arcs labeled $3$ have weights $\frac{y_3}{y_2}$ and arcs labeled $11$ have weights $\frac{y_{11}}{y_{10}}$.
Any other arc labeled $i$ has weight $y_i$.
\end{example}

The following lemma is immediate from the definition of tropical lambda lengths of boundary segments and punctures.

\begin{lemma}\label{alt OK}
The weighted multilaminations $\L$ and $\AA$ assign the same tropical lambda lengths to punctures and boundary segments as the principal coefficients weighted lamination.
\end{lemma}

\cref{alt OK} means that the weighted multilamination $\AA$ defines the same opening of $(\S,\M)$ as~$\L$.
Furthermore, the lambda lengths of punctures and boundary components are the same for the two weighted laminations.
We will see that, for the right choice of lifting of the two multilaminations to the opened surface, the tropical lambda lengths of all tagged arcs are also the same.

Choose a lift $\overline\L$ of the principal-coefficients weighted multilamination $\L$ to the opened surface so that, whenever two tagged arcs $\beta,\gamma\in T$ coincide except that $\beta$ is notched at $p$ and $\gamma$ is tagged plain at $p$, the corresponding lifted elementary laminations~$\overline L_\beta$ and $\overline L_\gamma$ differ only in that~$\overline\beta$ twists one additional time about $C_p$ in the clockwise direction (and also necessarily differ in the orientation of~$C_p$, with $\overline L_\gamma$ having a clockwise orientation).
Choose a lift $\overline\AA$ of the alternative weighted lamination $\AA$ that agrees with the lift of $\L$ where the arcs in $T^\circ$ coincide with plain-tagged arcs in $T$ and for every arc $\delta$ in $T^\circ$ that is the exterior edge of a self-folded triangle with interior edge $\gamma$, the lift of $L_\delta$ traces closely along the lift of $L_\gamma$, goes around the interior vertex, and traces closely back along the lift of $\gamma$.

\begin{lemma}\label{alt really OK}
If $\alpha$ is a tagged arc and $\overline\alpha$ is any lift of $\alpha$, then $c_{\overline\L}(\overline\alpha)=c_{\overline\L}(\overline\alpha)$.
\end{lemma}
\begin{proof}
Suppose $\delta$ is the exterior edge of a self-folded triangle with interior edge $\gamma$ and $\beta$ is the tagged arc that agrees with $\gamma$ but is tagged notched at the interior vertex~$p$.
We need to check that the (multiplicative) contributions of $\overline L_\beta$ and $\overline L_\gamma$ to $c_{\overline\L}(\overline\alpha)$ are the same as the contributions of~$\overline L_\gamma$ and $\overline L_\delta$ to~$c_{\overline\AA}(\overline\alpha)$.

Away from any endpoint of $\alpha$ that is at $p$, the lifts $\overline L_\beta$ and $\overline L_\gamma$ agree, and there is contribution $y_\beta y_\gamma$ to $c_{\overline\L}(\overline\alpha)$ each time $\overline\alpha$ intersects $\overline L_\beta$ and $\overline L_\gamma$.
The corresponding contribution to $c_{\overline\AA}(\overline\alpha)$ is $y_\beta^2\frac{y_\gamma}{y_\beta}$, because the weight of $\delta$ is $y_\beta$, the weight of $\gamma$ is $\frac{y_\gamma}{y_\beta}$, and (away from $p$), $\overline L_\delta$ traces along $\overline L_\gamma$ twice.

To deal with contributions to $c_{\overline\L}(\overline\alpha)$ and $c_{\overline\AA}(\overline\alpha)$ at any endpoints of $\alpha$ that are at $p$, we first look at how these contributions change when $\overline\alpha$ changes.
Since $\overline L_\beta$ carries a counterclockwise orientation of $C_p$ and $\overline L_\gamma$ carries a clockwise orientation, whenever $\overline\alpha$ is given an additional clockwise twist at~$C_p$, the tropical lambda length $c_{\overline\L}(\overline\alpha)$ is multiplied by a factor of $\frac{y_\gamma}{y_\beta}$ (or $\frac{y_\beta}{y_\gamma}$ if $\alpha$ is notched at~$p$).
Since the weight of $\gamma$ in $\AA$ is $\frac{y_\gamma}{y_\beta}$, when $\overline\alpha$ is given an additional clockwise twist, $c_{\overline\AA}(\overline\alpha)$ is also multiplied by a factor of $\frac{y_\gamma}{y_\beta}$ (or $\frac{y_\beta}{y_\gamma}$ if $\alpha$ is notched).
Thus, to prove that the contributions to~$c_{\overline\L}(\overline\alpha)$ and $c_{\overline\AA}(\overline\alpha)$ at $p$ are the same, we may make a convenient choice of the lift $\overline\alpha$.
By matching the spirals of $\overline\alpha$ to those of $\overline L_\gamma$, we will assume that $\overline\alpha$ does not intersect $\overline L_\gamma$ at~$p$.
(See \cref{endpoint contrib}, where labels $\beta$, $\gamma$, and $\delta$ are shorthand for $\overline L_\beta$, $\overline L_\gamma$, and $\overline L_\delta$.
The left picture shows $\overline\L$ and the right picture shows $\overline\AA$.)
\begin{figure}
\begin{tabular}{ccc}
\scalebox{0.8}{\includegraphics{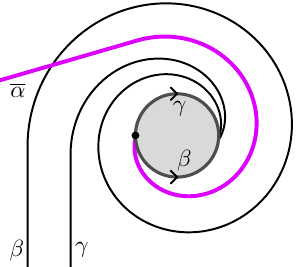}}&\qquad\qquad&\scalebox{0.8}{\includegraphics{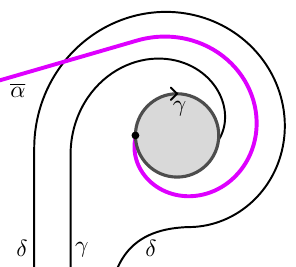}}  \end{tabular}
\caption{An illustration of the proof of \cref{alt really OK}.}
\label{endpoint contrib}
\end{figure}

For a moment, assume that the endpoint of $\alpha$ is tagged plain at~$p$.
Then $\overline L_\gamma$ contributes a factor $1$ to $c_{\overline\L}(\overline\alpha)$.
Since $\overline L_\beta$ twists clockwise about $C_p$ once more than $\overline L_\gamma$, a lift of $\alpha$ with one more clockwise twist would have contribution $1$ from $\overline L_\beta$ as well.
Since we obtain $\overline\alpha$ by adding one counterclockwise twist to a lift with contribution $1$ from $\overline L_\beta$ and since $\overline L_\beta$ carries a counterclockwise orientation at $p$, the contribution of $\overline L_\beta$ to $c_{\overline\L}(\overline\alpha)$ is $y_\beta$.
Because $\overline L_\delta$ follows $\overline L_\gamma$ to $p$, goes around~$p$, and then follows $\overline L_\gamma$ again, $\overline\alpha$ intersects $\overline L_\delta$ exactly once at $p$.
Since also $\overline\alpha$ does not intersect $\overline L_\gamma$ at $p$, the contribution to  $c_{\overline\AA}(\overline\alpha)$ at $p$ is again $y_\beta$.

If the endpoint of $\alpha$ is tagged notched at~$p$, there are additional contributions of $y_\beta^{-1}$ and $y_\gamma$ to~$c_{\overline\L}(\overline\alpha)$ at $p$ and there is an additional contribution of $\frac{y_\gamma}{y_\beta}$ to $c_{\overline\AA}(\overline\alpha)$ at~$p$.
\end{proof}

\begin{remark}\label{ambiguity rem}
It may seem surprising that different weighted multilaminations give the same laminated lambda lengths.
To avoid this non-uniqueness of multilaminations, one would need to replace laminations with the quasi-laminations of~\cite[Section~4]{unisurface}.
However, for the present purposes, the non-uniqueness is a benefit, because it allows us to lift the multilamination to the tile cover in the right way.
Indeed, the fact that we lift the alternative multilamination rather than the principal coefficients lamination accounts for the fractional labels on some elements of~$P_\alpha$.
\end{remark}

In light of \cref{alt OK,alt really OK}, we lift the alternative weighted multilamination $\AA$ to the tile cover.
This is done by making the lamination correct in each tile.
The result is a principal coefficients weighted multilamination on the tile cover where the weights have been specialized to match the weights on the alternative weighted multilamination in $(\S,\M)$.

\begin{example}\label{alt ex}
Continuing \cref{tile ex 2 remix}, the weights on the right picture of \cref{tilex1} are $\frac{y_3}{y_2}$ for arcs labeled $3$, $\frac{y_{11}}{y_{10}}$ for arcs labeled $11$, and otherwise $y_i$ for arcs labeled~$y_i$.
\end{example}

Now that we know the weights on the arcs in $T'$, we can prove the following strengthening of \cref{F lift lite}.

\begin{lemma}\label{F lift}
If $\alpha'$ is the tile cover of~$\alpha$, then $F(P_{\alpha})=F(P_{\alpha'})$.
\end{lemma}

\begin{proof}
Elements of the posets have ``weights'', and there are also ``weights'' $y_\gamma$ for $\gamma\in T$.
Both kinds of weights appear in this proof, but for clarity, we will use the word ``weight'' in this proof only for the weights on the posets and refer to the $y_\gamma$ as ``$y$-variables''.

The two posets are isomorphic by construction, and we check that corresponding elements have the same weights.
We checked in the proof of \cref{F lift lite} that corresponding elements of the posets have the same weight when all the $y$-variables are specialized to~$1$.
Since each weight is a Laurent monomial in the $x$- and $y$-variables, we complete the proof by showing that corresponding elements of the posets have the same weight when all the $x$-variables are specialized to~$1$.
We make that specialization for the remainder of the proof.

In $P_\alpha$, an element corresponding to crossing an arc $\gamma\in T^\circ$ is weighted $y_\gamma$ unless $\gamma$ is the interior edge of a self-folded triangle, in which case it is weighted $\frac{y_\gamma}{y_\beta}$, where $\beta$ is the exterior edge.
Every element of $P_{\alpha'}$ is weighted by a single $y$-variable in the sense of the weighted lamination on~$(\S',\M')$.
Since we lifted the alternative weighted lamination to $(\S',\M')$, this is again $y_\gamma$ unless~$\gamma$ is an interior edge, in which case it is $\frac{y_\gamma}{y_\beta}$, where $\beta$ is the exterior edge.
%
\end{proof}

The opening of $(\S,\M)$ can be made by cutting corners from triangles in $T^\circ$.
We obtain the opening of $(\S',\M')$ by making corresponding cuts in the tiles and glueing the tiles together so that the removed corners become holes.
At any endpoints of $\alpha$ that are opened, the neighborhood of the hole is the same in $(\S,\M)$ as in $(\S',\M')$, except when an endpoint of $\alpha$ is the interior vertex of a self-folded triangle.
In that case, the neighborhood of the hole in $(\S',\M')$ is a double cover of the neighborhood of the hole in $(\S,\M)$.
We lift the hyperbolic metric to the tiles (minus removed corners) and also lift the horocycles, conjugate horocycles, perpendicular horocyclic segments, and conjugate perpendicular horocyclic segments.
In the case where the neighborhood of the hole in $(\S',\M')$ is a double cover, we obtain two marked points (and two perpendicular horocyclic segments).
Of the two marked points, we retain the one that is the endpoint of $\alpha'$ and ignore the other (and its perpendicular horocyclic segment).

Moving forward, we assume all the structure on $(\S',\M')$ that was lifted from $(\S,\M)$ as described above.
The following is the analog of \cref{lambda lift} for the principal coefficients case.

\begin{lemma}\label{lambda lift opened}
If $\alpha'$ is the tile cover of~$\alpha$, then the laminated lambda length $\lambda(\alpha)$ in $(\S,\M)$ equals the laminated lambda length $\lambda(\alpha')$ in $(\S',\M')$.
\end{lemma}
\begin{proof}
Take any lift $\overline\alpha'$ of the tile cover $\alpha'$ that lies within the tiles in the opening of $(\S',\M')$.
There is a projection of the opening of $(\S',\M')$ to the opening of $(\S,\M)$ by sending each tile isometrically to the corresponding triangle, and we write $\overline\alpha$ for the projection of $\overline\alpha'$.
As the notation suggests,~$\overline\alpha$ is a lift of $\alpha$ to the opening of $(\S,\M)$.

The simple idea of the proof is as follows:
The lambda length of $\overline\alpha'$ equals the lambda length of~$\overline\alpha$ because the hyperbolic structure on the opening of $(\S',\M')$ is lifted, tile by tile, from the hyperbolic structure on the opening of $(\S,\M)$---everywhere where $\overline\alpha$ passes through the surface---and furthermore, all horocycles, conjugate horocycles, perpendicular horocyclic segments, and conjugate perpendicular horocyclic segments relevant to $\overline\alpha'$ are lifted from the corresponding horcycles, etc.\ relevant to $\overline\alpha$.
Similarly, the weighted multilamination on $(\S',\M')$ is lifted from the alternative weighted multilamination on $(\S,\M)$, so the tropical lambda length of $\overline\alpha'$ equals the tropical lambda length of $\overline\alpha$ with respect to the alternative weighted multilamination, and thus by \cref{alt really OK} equals the principal coefficients tropical lambda length of $\overline\alpha$.
Since the laminated lambda length is the ratio of the lambda length over the tropical lambda length, it follows that the laminated lambda length of $\alpha$ equals the laminated lambda length of $\alpha'$.

This simple argument is complete and correct in most cases, but fails when one or both endpoints of $\alpha$ is at the interior vertex $p$ of a self-folded triangle in $T^\circ$.
The issue is that the neighborhood of~$C_p$, which is vital to computing both lambda lengths and tropical lambda lengths, does not lift to the neighborhood of an opened puncture in the opening of $(\S',\M')$.
Instead, the neighborhood of~$C_p$ is \emph{doubly covered} by the neighborhood of an opened puncture in the opening of $(\S',\M')$.
The argument is easily rescued in such cases.
As in the proof of \cref{lambda lift}, a crucial point is that in such cases, $\alpha$ is tagged plain at~$p$.
Another crucial point is how lambda lengths and tropical lambda lengths change when the lifts are changed by inserting additional twists about~$C_p$.

Let $\delta$ be the exterior edge of the self-folded triangle whose interior vertex is $p$, let $\gamma$ be the interior edge, and let $\beta$ be the tagged arc that agrees with $\gamma$ but is tagged notched at~$p$.
The puncture $p'$ in $(\S',\M')$ corresponding to this endpoint of $\alpha$ is contained in two tiles.
The two edges incident to $p'$ both have laminated lambda lengths $x_\gamma$.
The weights on the corresponding elementary laminations are both $\frac{y_\gamma}{y_\beta}$.
Thus the tropical lambda length of $p'$ is $c(p')=\frac{y_\gamma^2}{y_\beta^2}$, and this also equals $\lambda(p')$.

In $(\S,\M)$, on the other hand, there is one edge of $T^\circ$ incident to $p$, with lambda length $x_\gamma$.
The associated elementary lamination has weight $\frac{y_\gamma}{y_\beta}$, so $c(p)=\frac{y_\gamma}{y_\beta}=\lambda(p)$.
The fact that $c(p')=c(p)^2$ is consistent with the fact that the neighborhood of $C_{p'}$ doubly covers the neighborhood of~$C_p$.

We can now easily see that the tropical lambda lengths of $\overline\alpha$ and $\overline\alpha'$ are still equal, despite the local double cover.
The orientation of $C_p$ in $\overline L_\gamma$ is clockwise, as is the orientation of $C_{p'}$ in the two copies of $\overline L_\gamma$ that were lifted to $(\S',\M')$.
For a moment, assume that the lift $\overline\alpha'$ so that it twists far enough clockwise about $C_{p'}$ and so that $\overline\alpha$ also twists far enough about $C_p$.
(We can do this at both endpoints of $\alpha$ if necessary.
Since the requirement is that $\overline\alpha$ must twist far enough clockwise at both endpoints, this can be done even if the endpoints of $\alpha$ are both at $p$.)
Now every contribution of $\overline L_\gamma$ to the tropical lambda length of $\overline\alpha$ in $(\S,\M)$ lifts to a corresponding contribution (with the same weight) to the tropical lambda length of $\overline\alpha'$ in $(\S',\M')$.

We see that, for a particular choice of $\overline\alpha'$, the tropical lambda lengths of $\overline\alpha$ and $\overline\alpha'$ are equal.
When we change the lift of $\alpha'$ by adding a clockwise twist about $C_{p'}$, we multiply the tropical lambda length of the lift by $c(p')=y_\gamma^{2}/y_\beta^{2}$.
This changes the corresponding lift of $\alpha$ by adding \emph{two} clockwise twists about $C_p$, thus multiplying the tropical lambda length by $(y_\gamma/y_\beta)^2$.
Thus for any particular choice of $\overline\alpha'$, the tropical lambda lengths of $\overline\alpha$ and $\overline\alpha'$ are equal.

We can similarly see that the lambda lengths of $\overline\alpha$ and $\overline\alpha'$ are equal.
In the definition of these lambda lengths, we need the lifts to twist far enough in a particular direction.  
The direction around~$C_{p'}$ and the direction around $C_p$ are related by the projection from the tile cover to~$(\S,\M)$.
Thus we can take a lift of $\alpha'$ that twists far enough and projects to a lift of $\alpha$ that twists far enough.  
In this case, the segment $[\alpha'_{\sigma'}]$ projects to the segment $[\alpha_\sigma]$, where we have written $\sigma'$ for the hyperbolic metric on $(\S',\M')$.
Since $\sigma'$ is the lift of $\sigma$, these two segments have the same signed length and therefore the two lifts have the same lambda length.
To change these two lifts to $\overline\alpha'$ and $\overline\alpha$, we multiply them respectively by $\lambda(p')^k=y_\gamma^{2k}/y_\beta^{2k}$ and $\lambda(p)^{2k}=y_\gamma^{2k}/y_\beta^{2k}$ for some~$k$.
Thus $\lambda(\overline\alpha')=\lambda(\overline\alpha)$.
(We can make the same argument if both endpoints of $\alpha$ are interior vertices of a self-folded triangle, and if the endpoints of $\alpha$ coincide, the directions of ``twisting far enough'' are the same at both endpoints of $\alpha'$.)
\end{proof}

We now prove the main result of the paper.
The proof parallels the proof of \cref{main lite} and uses some of the same subsidiary results.

\begin{proof}[Proof of \cref{main}]
Suppose $\alpha$ is a tagged arc in $(\S,\M)$.
By \cref{reduce}, we can assume that when $\alpha$ has an endpoint that is the interior vertex of a self-folded triangle in $T^\circ$, the tagging of $\alpha$ is plain at that endpoint.  We need to show that
\[
\lambda(\alpha)=\g_\alpha\cdot F(P_\alpha).
\]
We argue by induction on the number of elements of $P_\alpha$, more specifically taking this inductive hypothesis: the formula holds for any tagged arc $\alpha'$ in \emph{any} triangulated marked surface such that~$P_{\alpha'}$ has fewer elements than $P_\alpha$.
The base of the induction, when $P_\alpha$ is empty, is easy, and details are in the proof of \cref{main lite}.

If $P_\alpha$ is not empty, then construct the tile cover $\alpha'$ in a marked surface $(\S',\M')$ with a triangulation $T'$.
Lift the opening, hyperbolic metric, weighted multilamination $\AA$, and specified laminated lambda lengths to $(\S',\M')$ as explained above.
By construction, $P_{\alpha'}$ is isomorphic with $P_\alpha$ as unweighted posets, and $(T',\alpha')$ is tidy.
Thus there is an arc $\gamma'\in T'$ such that $\hy_{\gamma'}$ occurs exactly once as a label in $P_\alpha$.
Thus, just as in the proof of \cref{main lite} (\cref{ideal exch eq lite}), \cref{ideal decomp} says that 
\begin{equation}\label{ideal exch eq}
x_{\gamma'}\cdot\g_{\alpha'}\cdot F(P_{\alpha'})=x_{\gamma'}\cdot\g_{\alpha'}\cdot F(\Pb)+x_{\gamma'}\cdot\g_{\alpha'}\cdot\hy_\orange\cdot F(\Pr).
\end{equation}
The quantity $F(P_\blue)$ in \cref{ideal exch eq} is the product of one or two polynomials $F(P_{\delta'})$ for tagged arcs $\delta'$ in $(\S',\M')$.
Each of these posets $P_{\delta'}$ has strictly fewer elements than $P_\alpha$.
By induction, the laminated lambda length of $\delta'$ in $(\S',\M')$ is $\g_{\delta'}\cdot F(P_{\delta'})$, and this is the cluster variable associated to $\delta'$ in $(\S',\M')$.
Similarly, $F(P_\red)$ is the product of one or two polynomials $F(P_{\delta'})$ and by induction, the laminate lambda length of each $\delta'$ is $\g_{\delta'}\cdot F(P_{\delta'})$, which is again a cluster variable~$x_{\delta'}$.
\cref{g prop}.\ref{g prop blue} (applied to $\alpha'$) says that the first term on the right side of \cref{ideal exch eq} is~${\g_\blue\cdot F(P_\blue)}$, which is the product of the cluster variables for the (one or two) blue arcs.
Writing $y_\orange$ for the product of the coefficient variables $y_\gamma$ for orange arcs $\gamma\in T$, \cref{g prop}.\ref{g prop red} says that the second term on the right side of \cref{ideal exch eq} is $y_\orange\cdot\g_\red\cdot F(P_\red)$, which is $y_\orange$ times the product of the cluster variables for the (one or two) red arcs.
Thus the right side of \eqref{ideal exch eq} is the right side of the exchange relation exchanging the cluster variables for $\alpha'$ and $\gamma$.
(See \cref{F or x}.)
We conclude that the cluster variable for $\alpha'$ in $(\S',\M')$ is $\g_{\alpha'}\cdot F(P_{\alpha'})$.
This cluster variable is the laminated lambda length of $\alpha'$ in $(\S',\M')$.
\cref{lambda lift opened} now says that the laminated lambda length of $\alpha$ in $(\S,\M)$ is $\g_{\alpha'}\cdot F(P_{\alpha'})$, which equals $\g_\alpha\cdot F(P_\alpha)$, by \cref{g lift,F lift}.
\end{proof}

\section{Laminated lambda lengths of tagged geodesics}\label{geo sec}
Given a hyperbolic metric on $\S\setminus\M$, we will use the term \newword{tagged geodesic} for any geodesic in~$\S$ that limits to points in $\M$ at both ends and is tagged plain or notched at each of its endpoints that is a puncture.
(The definition of a tagged arc requires that the arc not intersect itself, that it may not bound a once-punctured monogon, and that, if its two endpoints coincide, they must have the same tagging.
For tagged geodesics, we drop all of these requirements, but we keep the requirement that tags at boundary points are plain.)
The construction in \cite[Definition~15.3]{FominThurston}, reviewed in \cref{sec:hyperbolic,sec:tropical}, extends to the definition of a (laminated) lambda length $x_\alpha$ for any tagged geodesic $\alpha$.
In this section, we extend \cref{main,main lite} to all tagged geodesics.
We also give a combinatorial characterization of tagged geodesics.

\subsection{Extending the theorems}\label{ex sec}
The following theorem extends \cref{main} by weakening the hypotheses on $\alpha$.

\begin{theorem}\label{extended}
The principal-coefficients laminated lambda length of a tagged geodesic~$\alpha$ is
\begin{equation}
\label{ex eq}
x_\alpha=\g_\alpha\cdot\sum_I \prod_{e\in I}w(e),
\end{equation}
where the sum is over all order ideals $I$ in $P_\alpha$.
\end{theorem}

The theorem associates to any tagged geodesic $\alpha$ both a $\g$-vector and an $F$-polynomial, in the same sense as for cluster variables.
The $\g$-vector is $\g_\alpha$ and the $F$-polynomial is $F(P_\alpha)$.

If we set all the $y_\gamma$ to $1$, then $c(p)=1$ for every puncture, so the opened surface has none of the punctures opened.
Thus the lift $\overline\L$ is the same as $\L$ and also every tagged geodesic $\alpha$ has~$\overline\alpha=\alpha$, so lambda lengths in the (not) opened surface are the same as lambda lengths in the coefficient-free model.
Also, $c_{\overline\L}(\overline\alpha)=1$ for every $\alpha$, so laminated lambda lengths are the same as ordinary lambda lengths.
Thus the entire laminated lambda length construction collapses to the coefficient-free lambda length construction.
We conclude that setting all $y_\gamma$ to $1$ in \cref{extended} implies the following theorem, which extends \cref{main lite}.

\begin{theorem}\label{extended lite}
Suppose all boundary components of $(\S,\M)$ have lambda length~$1$.
The lambda length of a tagged geodesic $\alpha$ is
\begin{equation}\label{ex lite eq}
\lambda(\alpha)=\g_\alpha\cdot\sum_I \prod_{e\in I}\tilde w(e),
\end{equation}
where the sum is over all order ideals $I$ in $P_\alpha$.
\end{theorem}


We deduce \cref{extended} from \cref{main} using a variation on the tile cover construction.

\begin{proof}[Proof of \cref{extended}]
Suppose $\alpha$ is a tagged geodesic.
If $\alpha$ is a tagged arc, then we are done by \cref{main}, so assume $\alpha$ is not a tagged arc.
In particular, $\alpha$ does not coincide, except for tagging, with an arc of $T^\circ$.

We lift $\alpha$ to a tagged arc $\alpha'$ using a variation of the tile cover construction.
As in \cref{tile cov sec}, decompose $\alpha$ into segments, each having its endpoints on arcs of $T^\circ$ and otherwise not intersecting the arcs of $T^\circ$.
Construct a tile for each segment as before, except at any endpoint of $\alpha$ that is the interior vertex of a self-folded triangle.
At each such endpoint, keep the self-folded triangle instead of making a tile.
Identify all of these triangles in the natural way. 

Lift the opening, multilamination, hyperbolic structure, and specified lamination lengths to the tile cover as in \cref{tile cov sec}, with the obvious simplifications due to keeping some self-folded triangles.
\cref{g lift,lambda lift opened,F lift} hold by essentially the same proofs.
(For \cref{g lift}, the proof is simpler because the endpoints of $\alpha$ remain like the left pictures of \cref{digonshear} instead of the right pictures.
For \cref{F lift}, the proof is simpler because there are no double covers to account for.)
Since $\alpha'$ is a tagged arc \cref{ex eq} holds for $\alpha'$ by \cref{main}.
By the analogs of these three lemmas, \cref{ex eq} holds for~$\alpha$.
\end{proof}

\subsection{Combinatorial tagged geodesics}\label{comb tag sec}

\cref{extended} applies to all tagged geodesics in $(\S,\M)$, but it is more convenient to apply the theorem to curves that are determined combinatorially, rather than by the metric.
As a byproduct of the proof of the theorem, we explain how to do that.

A \newword{combinatorial tagged geodesic} in $(\S,\M)$ consists of two (not necessarily distinct) marked points $p$ and $q$ in $\M$, a tagging at each of the points that is a puncture, and a sequence of arcs of $T^\circ$, either a single arc in $T^\circ$ with endpoints $p$ and $q$ or a sequence satisfying the following conditions:  
the sequence starts with the an arc opposite $p$ in some triangle in $T^\circ$; 
the sequence ends with the arc opposite $q$ in some triangle in $T^\circ$; 
every pair of adjacent entries in the sequence are \emph{distinct} arcs in the \emph{same} triangle of $T^\circ$; and
the sequence of triangles determined by pairs of adjacent entries does not have the same triangle twice in a row.
Combinatorial tagged geodesics are defined up to the symmetry of swapping endpoints and reversing the sequence.

A combinatorial tagged geodesic determines, up to isotopy, a tagged curve in $\S$ connecting the two endpoints and crossing arcs of $T^\circ$ according to the sequence (or instead coinciding with the one arc in the sequence).
In the other direction, if $\alpha$ is a tagged geodesic in $(\S,\M)$, then $\alpha$ determines (up to reversing) a sequence of arcs of $T^\circ$ in the obvious way:
If $\alpha$ coincides with an arc in $T^\circ$, then the sequence consists of that arc.
Otherwise, the sequence consists of the arcs of~$T^\circ$ crossed by $\alpha$ on the way from one endpoint of $\alpha$ to the other.

\begin{proposition}\label{tagged geo prop}
If $\alpha$ is a tagged geodesic in $(\S,\M)$, then the endpoints of $\alpha$, its taggings, and the sequence determined by $\alpha$ constitute a combinatorial tagged geodesic in $(\S,\M)$.
Conversely, given any combinatorial tagged geodesic in $(\S,\M)$, the curve it determines is isotopic to a unique tagged geodesic in $(\S,\M)$.
\end{proposition}
\begin{proof}
Suppose $\alpha$ is a tagged geodesic in $(\S,\M)$.
If $\alpha$ coincides with an arc in $T^\circ$, there is nothing to check, so suppose not.
Since the arcs in $T^\circ$ are also geodesics and since the hyperbolic metric on $(\S,\M)$ has constant curvature $-1$, $\alpha$ does not intersect the same arc in $T^\circ$ twice without intersecting some other arc in between.
Thus the sequence of arcs crossed by $\alpha$ satisfies all the conditions of a combinatorial tagged geodesic.

Conversely, given a combinatorial tagged geodesic, let $\alpha$ be a tagged curve in $\S$ determined by the combinatorial tagged geodesic.
We construct a tile cover $\alpha'$ for $\alpha$, ignoring the requirement that $\alpha$ is plain at any endpoint that is the interior point of a self-folded triangle.
Since $\alpha'$ is isotopic to an arc in $(\S',\M')$, there is an isotopy representative of $\alpha'$ that is geodesic.
For the reasons given in the previous paragraph, this geodesic representative visits the same sequence of arcs of $T'$ as~$\alpha'$.
The geodesic descends to a geodesic in $(\S,\M)$ that is isotopic to~$\alpha$.
\end{proof}

\subsection*{Acknowledgments}
We thank Salvatore Stella for helpful conversations.


\bibliographystyle{alpha}
\bibliography{clusterideals}
\label{sec:biblio}

\end{document}